\documentclass[12pt]{amsart}

\usepackage{amsmath,amssymb}
\usepackage{bm}
\usepackage{graphicx}
\usepackage{ascmac}
\usepackage{amsthm}
\usepackage{amsfonts}
\usepackage{mathrsfs}
\usepackage[all]{xy}

\theoremstyle{definition}
\newtheorem{thm}{Theorem}[section]
\newtheorem{dfn}[thm]{Definition}

\newtheorem{cor}[thm]{Corollary}
\newtheorem{prop}[thm]{Proposition}
\newtheorem{lem}[thm]{Lemma}
\newtheorem{rem}[thm]{Remark}

\newtheorem*{oftp}{Organization of this paper}
\newtheorem*{ack}{Acknowledgment}
\newtheorem*{claim}{Claim}
\newtheorem*{nota}{Notations}

\numberwithin{thm}{section}

\newcommand{\Zpn}{\mathbb{Z}_{>0}}
\newcommand{\Znn}{\mathbb{Z}_{\geq 0}}
\newcommand{\Z}{\mathbb{Z}}
\newcommand{\Zp}{{\mathbb{Z}}_p}
\newcommand{\Zpt}{{\mathbb{Z}}_{p}^{\times}}

\newcommand{\Q}{\mathbb{Q}}
\newcommand{\Qp}{{\mathbb{Q}}_p}
\newcommand{\Qpt}{{\mathbb{Q}}_{p}^{\times}}

\newcommand{\Qpbar}{\overline{\mathbb{Q}}_p}
\newcommand{\Zl}{{\mathbb{Z}}_{\ell}}
\newcommand{\Ql}{{\mathbb{Q}}_{\ell}}
\renewcommand{\O}{\mathrm{O}}
\newcommand{\Zhat}{\widehat{\mathbb{Z}}}
\newcommand{\Qbar}{\overline{\mathbb{Q}}}
\newcommand{\R}{\mathbb{R}}
\newcommand{\C}{\mathbb{C}}
\newcommand{\Fp}{\mathbb{F}_p}

\newcommand{\Fpbar}{\overline{\mathbb{F}}_p}
\newcommand{\A}{\mathbb{A}}
\newcommand{\Sbar}{\overline{S}}
\newcommand{\D}{\mathbb{D}}
\newcommand{\G}{\mathbb{G}}
\newcommand{\bS}{\mathbb{S}}

\newcommand{\bG}{\mathbf{G}}
\newcommand{\M}{\mathcal{M}}

\newcommand{\sS}{\mathscr{S}}
\newcommand{\U}{\mathrm{U}}
\newcommand{\bV}{\mathbf{V}}
\newcommand{\X}{\mathbb{X}}
\newcommand{\bX}{\mathbf{X}}

\newcommand{\Fbr}{\breve{F}}
\newcommand{\bLambda}{\mathbf{\Lambda}}

\newcommand{\abar}{\overline{a}}

\newcommand{\sbar}{\overline{s}}
\renewcommand{\hbar}{\overline{h}}
\newcommand{\B}{\mathbb{B}}

\renewcommand{\L}{\mathbf{L}}
\newcommand{\E}{\mathbf{E}}
\newcommand{\F}{\mathbf{F}}
\newcommand{\bT}{\mathbf{T}}
\newcommand{\e}{\mathbf{e}}
\newcommand{\f}{\mathbf{f}}
\newcommand{\cB}{\mathcal{B}}
\renewcommand{\P}{\mathbb{P}}

\newcommand{\cA}{\mathcal{A}}
\newcommand{\cE}{\mathcal{E}}

\newcommand{\cL}{\mathcal{L}}

\newcommand{\cZ}{\mathcal{Z}}

\DeclareMathOperator{\Ker}{Ker}
\DeclareMathOperator{\im}{Im}
\DeclareMathOperator{\Hom}{Hom}
\DeclareMathOperator{\End}{End}
\DeclareMathOperator{\id}{id}

\DeclareMathOperator{\Lie}{Lie}
\DeclareMathOperator{\Frac}{Frac}
\DeclareMathOperator{\tr}{Tr}
\DeclareMathOperator{\trd}{Trd}
\DeclareMathOperator{\nrd}{Nrd}
\DeclareMathOperator{\ord}{ord}

\DeclareMathOperator{\GL}{GL}
\DeclareMathOperator{\GSp}{GSp}
\DeclareMathOperator{\GU}{GU}
\DeclareMathOperator{\GSpin}{GSpin}

\DeclareMathOperator{\SU}{SU}
\DeclareMathOperator{\Spin}{Spin}
\DeclareMathOperator{\SO}{SO}

\DeclareMathOperator{\spin}{spin}

\DeclareMathOperator{\N}{N}
\DeclareMathOperator{\nilp}{Nilp}

\DeclareMathOperator{\spec}{Spec}
\DeclareMathOperator{\spf}{Spf}
\DeclareMathOperator{\naive}{naive}
\DeclareMathOperator{\loc}{loc}
\DeclareMathOperator{\sml}{sim}
\DeclareMathOperator{\red}{red}
\DeclareMathOperator{\length}{length}
\DeclareMathOperator{\OGr}{OGr}
\DeclareMathOperator{\inv}{inv}
\DeclareMathOperator{\Res}{Res}
\DeclareMathOperator{\si}{ss}
\DeclareMathOperator{\Proj}{Proj}
\DeclareMathOperator{\diag}{diag}
\DeclareMathOperator{\rk}{rank}

\DeclareMathOperator{\VL}{VL}

\DeclareMathOperator{\BT}{BT}

\DeclareMathOperator{\Irr}{Irr}
\DeclareMathOperator{\Sh}{Sh}
\DeclareMathOperator{\ab}{ab}

\DeclareMathOperator{\der}{der}
\DeclareMathOperator{\bs}{basic}
\DeclareMathOperator{\stab}{Stab}

\DeclareMathOperator{\DiL}{DiL}

\DeclareMathOperator{\Lat}{Lat}
\DeclareMathOperator{\SpL}{SpL}
\DeclareMathOperator{\Ht}{height}

\usepackage{geometry}
\title[Supersingular locus of $\mathrm{GU}(2,2)$-Shimura variety]{On the supersingular locus of the Shimura variety for $\GU(2,2)$ over a ramified prime}
\author[Y.~Oki]{Yasuhiro Oki}
\address{Graduate School of Mathematical Sciences, 
the University of Tokyo, 3-8-1 Komaba, Meguro-ku, Tokyo 153-8914, Japan.}
\email{oki@ms.u-tokyo.ac.jp}
\subjclass[2010]{11G18, 14G35}
\keywords{Unitary Shimura varieties, Rapoport--Zink spaces, supersingular loci}

\begin{document}
\maketitle

\begin{abstract}
We study the structure of the supersingular locus of the Rapoport--Zink integral model of the Shimura variety for $\mathrm{GU}(2,2)$ over a ramified odd prime with the special maximal parahoric level. We prove that the supersingular locus equals the disjoint union of two basic loci, one of which is contained in the flat locus, and the other is not. We also describe explicitly the structure of the basic loci. More precisely, the former one is purely $2$-dimensional, and each irreducible component is birational to the Fermat surface. On the other hand, the latter one is purely $1$-dimensional, and each irreducible component is birational to the projective line. 
\end{abstract}

\tableofcontents

\section{Introduction}\label{intr}

This paper contributes to the study of bad reductions of integral models of unitary Shimura varieties. More precisely, we study the geometric structure of an important part, which is called the supersingular locus, of the geometric special fibers of the Rapoport--Zink integral models of Shimura varieties for some unitary groups over certain ramified primes. 

Before describing our results, we recall notions of integral models and supersingular loci of Shimura varieties for unitary similitude groups. Let $L=\Q(\sqrt{\Delta})$ an imaginary quadratic field, where $\Delta \in \Z_{<0}$ is square-free. Fix an isomorphism $L\otimes_{\Q}\R\cong \C$, and let $U$ be an $L/\Q$-hermitian space of signature $(r,s)$. Put $\bG:=\GU(U)$, the unitary similitude group of $U$, and denote by $\bX$ the Grassmanian of negative definite $s$-dimensional subspaces of $U\otimes_{\Q}\R$. Then the pair $(\bG,\bX)$ is a Shimura datum. Hence, for a sufficiently small compact open subgroup $K$ of $\bG(\A_f)$ where $\A_f^p$ is the finite ad{\`e}le ring of $\Q$, we obtain the (canonical model of) Shimura variety $\Sh_{K}(\bG,\bX)$ over $E_{r,s}$, where 
\begin{equation*}
E_{r,s}:=
\begin{cases}
\Q &\text{if }r=s,\\
L&\text{otherwise}. 
\end{cases}
\end{equation*}
It is known that $\Sh_{K}(\bG,\bX)$ is of dimension $rs$. 

Next, we consider an integral model of $\Sh_{K}(\bG,\bX)$. Let $\nu$ be a place in $E_{r,s}$, and $p$ a prime which lies below $\nu$. We assume $p\neq 2$ if $p$ does not split in $L$. Moreover, we assume that $K=K^pK_p$, where $K^p$ is a sufficiently small compact open subgroup of $\bG(\A_f^p)$, and $K_p$ is a Bruhat--Tits subgroup, that is, a subgroup of $\bG(\Qp)$ which is a stabilizer of a self-dual lattice chain in $U\otimes_{\Q}\Qp$ in the sense of \cite[Definition 3.13]{Rapoport1996a}. Here $\A_{f}^{p}$ is the finite ad{\`e}le ring of $\Q$ without $p$-component. Then, Rapoport and Zink constructed in \cite{Rapoport1996a} a scheme $\sS_{K,U}$ over the ring of integers $O_{E_{r,s,\nu}}$ of $E_{r,s,\nu}$, which is defined as the moduli space of (families of) abelian varieties and additional structures. Here $E_{r,s,\nu}$ is the completion of $E$ at $\nu$. It is known that the generic fiber of $\sS_{K,U}$ is isomorphic to $\Sh_{K}(\bG,\bX)\otimes_{E_{r,s}}E_{r,s,\nu}$. We define the supersingular locus $\sS_{K,U}^{\si}$ of $\sS_{K,U}$ as the locus of the geometric special fiber of $\sS_{K,U}$ where the corresponding abelian varieties are supersingular. 

It is known that studying the structure of supersingular loci is important for number theory. For example, it can be applied to study some relations between intersection multiplicities of special divisors of Shimura varieties and the first derivatives of Eisenstein series. If $\bG=\GU(1,n-1)$, then it is conjectured by \cite[Conjecture 14.5.1]{Kudla2014}, and recently solved in \cite{Li2019}. 

We recall some results on the descriptions of $\sS_{K,U}^{\si}$. First, suppose that $p$ splits in $L$. It is known that $\sS_{K,U}^{\si}$ is non-empty only if $r=s$. If $(r,s)=(1,1)$ and $K_p$ is hyperspecial, then $\sS_{K,U}^{\si}$ is a finite set of $\Fpbar$-valued points. On the other hand, Fox \cite{Fox2018} (for $p\neq 2$) and Wang \cite{Wang2021} (for any $p$) gave explicit descriptions of $\sS_{K,U}^{\si}$ for $(r,s)=(2,2)$ and $K_p$ is hyperspecial. In this case, any irreducible component is birational to $\P^1$. 

Second, consider the case that $p$ inerts in $L$. For $(r,s)=(1,n-1)$ and $K_p$ is hyperspecial, the structure of $\sS_{K,U}^{\si}$ is given by Vollaard \cite{Vollaard2010} (for $n=3$) and by Vollaard and Wedhorn \cite{Vollaard2011} (for general $n$). Moreover, Cho \cite{Cho2018} analyzed in the case $(r,s)=(1,n-1)$ and $K_p$ is a parahoric subgroup which is the stabilizer of a single lattice in $U\otimes_{\Q}\Qp$. On the other hand, Howard and Pappas \cite{Howard2014} described concretely a structure of $\sS_{K,U}^{\si}$ when $(r,s)=(2,2)$ and $K$ is hyperspecial. In this case, every irreducible component is birational to the Fermat surface. 

Third, we assume that $p$ ramifies in $L$. In this case, $\sS_{K,U}$ may not be flat over $O_{E_{r,s,\nu}}$. Let $\sS_{K,V}^{\loc}$ be the scheme-theoretic closure of the generic fiber of $\sS_{K,V}$. When $(r,s)=(1,n-1)$ and $K_p$ is the stabilizer of a self-dual lattice, then Rapoport, Terstiege and Wilson \cite{Rapoport2014a} considered a description of the structure of $\sS_{K,U}^{\si}\cap \sS_{K,U}^{\loc}$. Note that $K$ is not parahoric if $n$ is even. On the other hand, Wu \cite{Wu2016} studied $\sS_{K,U}^{\si}\cap \sS_{K,U}^{\loc}$ when $(r,s)=(1,n-1)$ and $K_p$ is given by the stabilizer of a single lattice $\Lambda$ in $V\otimes_{\Q}\Qp$ satisfying $\Lambda^{\vee}=\sqrt{\Delta}^{-1}\Lambda$ if $n$ is even, or $\Lambda \subset \Lambda^{\vee}\subset \sqrt{\Delta}^{-1}\Lambda$ and $\length(\Lambda^{\vee}/\Lambda)=n-1$ if $n$ is odd. Here $\Lambda^{\vee}$ is the dual lattice of $\Lambda$ with respect to the hermitian form on $V\otimes_{\Q}\Qp$. In this case, the special fiber of $\sS_{K,V}^{\loc}$ is known to be smooth (which is called ``exotic good reduction''). 

In this paper, we consider the case where $p$ ramifies in $L$ and $(r,s)=(2,2)$. For it, we also consider the similar problem on the corresponding Rapoport--Zink spaces. We explain our results in the sequel. 

\subsection{Main results}\label{mtim}

Keep the notations as above (in particular, $p$ is an odd prime). We assume the following: 
\begin{itemize}
\item $(r,s)=(2,2)$,
\item $p$ is ramified in $L$ and $\nu=p$,
\item $\bG \otimes_{\Q}\Qp$ is quasi-split over $\Qp$, 
\item $K_p$ is a special maximal parahoric subgroup of $\bG(\Qp)$. 
\end{itemize}
Then the reflex field of the Shimura datum $(\bG,\bX)$ is $E_{2,2}=\Q$, and the integral model $\sS_{K,U}$ of $\Sh_{K}(\bG,\bX)\otimes_{\Q}\Qp$ is defined over $\Zp$. Moreover, the $\Zp$-scheme is defined as a moduli space of quadruples $(A,\iota,\lambda,\overline{\eta}^p)$, where
\begin{itemize}
\item $A$ is an abelian $4$-fold,
\item $\iota \colon O\rightarrow \End(A)$ is a ring homomorphism,
\item $\lambda$ is a polarization of $A$ such that 
\begin{equation*}
\ord_p(\#\ker(\lambda))=\ord_p(\#\ker(\iota(\sqrt{\Delta}))),
\end{equation*}
\item $\overline{\eta}^p$ is a $K^p$-level structure on $A$. 
\end{itemize}
Here $O$ is an order of $L$ which is maximal at $p$. We will give a more precise definition of $\sS_{K,U}$ in Definition \ref{rzit}. 

\begin{thm}\label{grim}(Theorem \ref{imdc} (i)--(iii))
\emph{
\begin{enumerate}
\item There is a decomposition into open and closed subschemes
\begin{equation*}
\sS_{K,U}=\sS_{K,U}^{\loc}\sqcup (\sS_{K,U}\setminus \sS_{K,U}^{\loc}). 
\end{equation*}
\item The scheme $\sS_{K,U}^{\loc}$ is flat over $\Zp$ and regular of dimension $5$. 
\item The scheme $\sS_{K,U}\setminus \sS_{K,U}^{\loc}$ is smooth over $\Fp$ and $3$-dimensional.
\end{enumerate}}
\end{thm}

\begin{thm}\label{ssdc}(Theorem \ref{ssbs})
\emph{The supersingular locus $\sS_{K,U}^{\si}$ consists of two basic loci $\sS_{K,U,[b_0]}$ and $\sS_{K,U,[b_1]}$, where $\sS_{K,U,[b_0]}\subset \sS_{K,U}^{\loc}$, and $\sS_{K,U,[b_1]}\subset \sS_{K,U}\setminus \sS_{K,U}^{\loc}$. }
\end{thm}

We give concrete descriptions the structure of these two basic loci. 

\begin{thm}\label{ntgr}(Theorem \ref{irg0} (i), (ii))
\emph{
\begin{enumerate}
\item The scheme $\sS_{K,U,[b_0]}$ is purely $2$-dimensional. Any irreducible component is birational to the Fermat surface $F_p$ defined by
\begin{equation*}
x_0^{p+1}+x_1^{p+1}+x_2^{p+1}+x_3^{p+1}=0
\end{equation*}
in $\P_{\Fpbar}^{3}=\Proj \Fpbar[x_0,x_1,x_2,x_3]$. 
\item Let $C$ be an irreducible component of $\sS_{K,U,[b_0]}$. Then, for any irreducible component $C'\neq C$ of $\sS_{K,U,[b_0]}$, the intersection $C\cap C'$ is either the empty set, a single point or birational to $\P_{\Fpbar}^{1}$. Moreover, the following hold: 
\begin{gather*}
\#\{C'\mid C\cap C'\text{ is a single point}\}\leq p(p+1)(p^2+1),\\
\#\{C'\mid C\cap C'\text{ is birational to }\P_{\Fpbar}^{1}\}\leq (p+1)(p^2+1). 
\end{gather*}
\end{enumerate}}
\end{thm}

\begin{thm}\label{nngr}(Theorem \ref{irg1})
\emph{\begin{enumerate}
\item The scheme $\sS_{K,U,[b_1]}$ is purely $1$-dimensional. Any irreducible component is birational to $\P_{\Fpbar}^{1}$. 
\item Let $C$ be an irreducible component of $\sS_{K,U,[b_1]}$. Then, for any irreducible component $C'\neq C$ of $\sS_{K,U,[b_1]}$, the intersection $C\cap C'$ is either the empty set or a single point. Moreover, we have
\begin{equation*}
\#\{C'\mid C\cap C'\text{ is a single point}\}\leq p(p^2+1). 
\end{equation*}
\end{enumerate}}
\end{thm}

\begin{rem}
We have more precise results on non-smooth points of $\sS_{K,U}^{\loc}$ and numbers of connected and irreducible components of $\sS_{K,U,[b_j]}$. See Theorems \ref{irg0}, \ref{cti0} and \ref{cti1}. 
\end{rem}

To prove Theorems \ref{ntgr} and \ref{nngr}, we deduce them from an explicit description of the underlying topological spaces of the corresponding Rapoport--Zink spaces by using the $p$-adic uniformization theorem \cite[Theorem 6.30]{Rapoport1996a}. Next, we explain the results on Rapoport--Zink spaces as mentioned above. 

\subsection{Results on the Rapoport--Zink spaces}\label{mtrz}

Let $F$ be a ramified quadratic extension of $\Qp$, and $O_F$ the ring of integers of $F$. We denote by $a\mapsto \abar$ the non-trivial Galois automorphism of $F$ over $\Qp$. We fix a uniformizer $\varpi$ of $F$. Put $W:=W(\Fpbar)$, the ring of Witt vectors over $\Fpbar$. Then, for a $W$-scheme $S$ in which $p$ is locally nilpotent, we consider a triple $(X,\iota,\lambda)$, where
\begin{itemize}
\item $X$ is a $p$-divisible group of dimension $4$ and height $8$ over $S$, 
\item $\iota \colon O_F\rightarrow \End(X)$ is a ring homomorphism,
\item $\lambda \colon X\rightarrow X^{\vee}$ is a quasi-polarization, 
\end{itemize}
satisfying the following conditions for any $a\in O_F$: 
\begin{itemize}
\item $\det(T-\iota(a)\mid \Lie(X))=(T^2-\tr_{F/\Qp}(a)T+\N_{F/\Qp}(a))^2$,
\item $\lambda \circ \iota(a)=\iota(\abar)^{\vee}\circ \lambda$. 
\end{itemize}
Let us fix such a triple $b=(\X_0,\iota_0,\lambda_0)$ over $\spec \Fpbar$ such that $\X_0$ is isoclinic of slope $1/2$ and $\lambda_0$ is an isogeny satisfying $\Ker(\lambda_0)=\Ker \iota_0(\varpi)$. We define $\M_{b,\mu}^{(0)}$ (the subscript $\mu$ will be introduced in Section \ref{lpel}) to be the deformation space of $(\X_0,\iota_0,\lambda_0)$ by $O_F$-linear quasi-isogenies. It is a formal scheme over $\spf W$, which is formally locally of finite type. We denote by $\M_{b,\mu}^{(0),\red}$ the underlying reduced subscheme of $\M_{b,\mu}^{(0)}$. 

With a fixed $b$, we associate an $F/\Qp$-hermitian space $C_{b}$ of dimension $4$. Then the $\Qp$-valued points of the unitary group $\U(C_{b})(\Qp)$ of $C_{b}$ acts on $\M_{b,\mu}^{(0)}$. We say that a lattice $T$ in $C_b$ is a \emph{vertex lattice} if $T\subset T^{\vee}\subset \varpi^{-1}T$. We call $t:=\length_{O_F}(T^{\vee}/T)$ the type of $T$. The type $t$ is an even integer satisfying $0\leq t\leq 4$. We denote by $\VL_{b}(t)$ the set of all vertex lattices of type $t$, and put $\VL_{b}:=\bigcup_{t=0}^{2}\VL_{b}(2t)$. We construct a locally closed stratification of $\M_{b,\mu}^{(0),\red}$ by means of vertex lattices in $C_{b}$. 

We consider two cases whether $C_{b}$ is split or not as an $F/\Qp$-hermitian space. We say that $b$ is \emph{$\mu$-neutral} if $C_{b}$ is split. 

\textbf{Case 1. The $\mu$-neutral case. }
In this case, $\VL_{b}(2t)$ is non-empty for any $t\in \{0,1,2\}$. We introduce an order $\leq$ on $\VL_{b}$ as follows. For distinct $T,T'\in \VL_{b}$, we have $T<T'$ if one of the following holds: 
\begin{itemize}
\item $T\in \VL_{b}(4)$, $T'\in \VL_{b}(0)$ and $T\subset T'$, 
\item $T\in \VL_{b}(4)$, $T'\in \VL_{b}(2)$ and $T\subset (T')^{\vee}$, 
\item $T\in \VL_{b}(0)$, $T'\in \VL_{b}(2)$ and $T\subset (T')^{\vee}$. 
\end{itemize}
Moreover, put
\begin{equation*}
d_b(T):=
\begin{cases}
0 &\text{if }T\in \VL_{b}(4),\\
4 &\text{if }T\in \VL_{b}(0),\\
6 &\text{if }T\in \VL_{b}(2). 
\end{cases}
\end{equation*}

With $T\in \VL_{b}$, we associate a reduced locally closed subscheme $\BT_{b,\mu,T}^{(0)}$ of $\M_{b,\mu}^{(0),\red}$. See Section \ref{cbbt}. 

\begin{thm}\label{ntlr}
\emph{Suppose that $b$ is $\mu$-neutral. }
\begin{enumerate}
\item (Theorem \ref{sgrt} (i), Theorem \ref{mthm} (iv)) \emph{The formal scheme $\M_{b,\mu}^{(0)}$ is flat over $\spf W$, and regular of dimension $5$. The non-formally smooth locus of $\M_{b,\mu}^{(0)}$ equals $\coprod_{T\in \VL_{b}(4)}\BT_{b,\mu,T}^{(0)}$. }
\item (Theorem \ref{dcnt}) \emph{There is a decomposition into open and closed formal subschemes
\begin{equation*}
\M_{b,\mu}^{(0)}=\M_{b,\mu}^{(0,0)}\sqcup \M_{b,\mu}^{(0,1)}. 
\end{equation*}
Moreover, $\M_{b,\mu}^{(0,i),\red}:=\M_{b,\mu}^{(0,i)}\cap \M_{b,\mu}^{\red}$ are connected and isomorphic to each other. }
\item (Theorem \ref{mthm} (i)) \emph{There is a locally closed stratification
\begin{equation*}
\M_{b,\mu}^{(0),\red}=\coprod_{T\in \VL_{b}}\BT_{b,\mu,T}^{(0)}. 
\end{equation*}}
\item (Theorem \ref{mthm} (ii)) \emph{For $x\in \VL_{b}$, let $\M_{b,\mu,T}^{(0),\red}$ be the closure of $\BT_{b,\mu,T}^{(0)}$ in $\M_{b,\mu}^{(0),\red}$. Then we have
\begin{equation*}
\M_{b,\mu,T}^{(0),\red}=\coprod_{T'\leq T}\BT_{b,\mu,T'}^{(0)}. 
\end{equation*}}
\item (Theorem \ref{mthm} (iii), Theorem \ref{dlnt}) \emph{For $T\in \VL_{b}$, every connected component of $\BT_{b,\mu,T}^{(0)}$ is isomorphic to a generalized Deligne--Lusztig variety for the split $\SO_{d_b(T)}$ (it is non-classical if $T\in \VL_{b}(0)\sqcup \VL_{b}(2)$). Moreover, the following hold: 
\begin{itemize}
\item if $T\in \VL_{b}(4)$, then $\M_{b,\mu,T}^{(0),\red}$ is a single point, 
\item if $T\in \VL_{b}(0)$, then $\M_{b,\mu,T}^{(0),\red}$ is isomorphic to two copies of $\P^1$, 
\item if $T\in \VL_{b}(2)$, then $\M_{b,\mu,T}^{(0),\red}$ is isomorphic to two copies of the Fermat surface $F_p$. 
\end{itemize}
In particular, $\M_{b,\mu}^{(0),\red}$ is purely $2$-dimensional. }
\end{enumerate}
\end{thm}

\textbf{Case 2. The non-$\mu$-neutral case. }
In this case, $\VL_{b}(2t)$ is non-empty if and only if $t\in \{0,1\}$. Moreover, put
\begin{equation*}
d_b(T):=4-t(T). 
\end{equation*}

With $T\in \VL_{b}$, we associate a reduced locally closed subscheme $\BT_{b,\mu,T}^{(0)}$ of $\M_{b,\mu}^{(0)}$. See Section \ref{btnn}. 

\begin{thm}\label{nnlr}
\emph{Suppose that $b$ is not $\mu$-neutral. }
\begin{enumerate}
\item (Theorem \ref{sgrt} (ii)) \emph{The formal scheme $\M_{b,\mu}^{(0)}$ equals $\M_{b,\mu}^{(0)}\times_{\spf W}\spec \Fpbar$, which is formally smooth of dimension $3$ over $\spec \Fpbar$. }
\item (Theorem \ref{dcnn}) \emph{There is a decomposition into open and closed formal subschemes
\begin{equation*}
\M_{b,\mu}^{(0)}=\M_{b,\mu}^{(0,0)}\sqcup \M_{b,\mu}^{(0,1)}. 
\end{equation*}
Moreover, $\M_{b,\mu}^{(0,i),\red}:=\M_{b,\mu}^{(0,i)}\cap \M_{b,\mu}^{\red}$ are connected and isomorphic to each other. }
\item \emph{There is a locally closed stratification
\begin{equation*}
\M_{b,\mu}^{(0),\red}=\coprod_{T\in \VL_{b}}\BT_{b,\mu,T}^{(0)}
\end{equation*}}
\item (Theorem \ref{nnbt} (ii)) \emph{For $T\in \VL_{b}$, let $\M_{b,\mu,T}^{(0),\red}$ be the closure of $\BT_{b,\mu,T}^{(0)}$ in $\M_{b,\mu}^{(0),\red}$. Then we have
\begin{equation*}
\M_{b,\mu,T}^{(0),\red}=\coprod_{T'\subset T}\BT_{b,\mu,T'}^{(0)}. 
\end{equation*}}
\item (Theorem \ref{mtdm}, Theorem \ref{dlnn}) \emph{For $T\in \VL_{b}$, the stratum $\BT_{b,\mu,T}^{(0)}$ is isomorphic to the disjoint union of two (classical) Deligne--Lusztig varieties for the non-split $\SO_{d_b(T)}$ associated with Coxeter elements. Moreover, the following hold: 
\begin{itemize}
\item if $T\in \VL_{b}(2)$, then $\M_{b,\mu,T}^{(0),\red}$ consists of two points, 
\item if $T\in \VL_{b}(0)$, then $\M_{b,\mu,T}^{(0),\red}$ is isomorphic to two copies of $\P^1$. 
\end{itemize}
In particular, $\M_{b,\mu}^{(0),\red}$ is purely $1$-dimensional. }
\end{enumerate}
\end{thm}

First, we sketch the proof of Theorem \ref{ntlr}. For the proof, we use the results of \cite{Oki2019} on a Rapoport--Zink space for $\GU_{2}(D)$, the quaternionic unitary similitudes group of degree $2$. Let $\M_{D}^{(0)}$ be the height $0$ locus of the Rapoport--Zink space for $\GU_{2}(D)$ appeared in \cite{Oki2019}. Then we construct a closed immersion of $\M_{D}^{(0)}$ into $\M_{b,\mu}^{(0)}$, and prove equality of topological spaces $\U(C_{b})(\Qp)\M_{D}^{(0),\red}=\M_{b,\mu}^{(0),\red}$. On the other hand, in the same argument as \cite{Howard2014}, we construct a $6$-dimensional quadratic space $\L_{b,0}$ over $\Qp$ as a subspace of $\End^0(\X_0):=\End(\X_0)\otimes_{\Z}\Q$ which realizes an exceptional isomorphism associated to the identity $A_3=D_3$ of Dynkin diagrams. Moreover, we introduce the notion of the vertex lattices in $\L_{b,0}$. Then Theorem \ref{ntlr} (ii) follows by using some properties of vertex lattices in $\L_{b,0}$ and the closed immersion $\M_{D}^{(0)}\rightarrow \M_{b,\mu}^{(0)}$. Moreover, we attach a closed formal subscheme $\M_{b,\mu,\Lambda}^{(0,0)}$ of $\M_{b,\mu}^{(0,0)}$ for each vertex lattice $\Lambda$ in $\L_{b,0}$, and construct a locally closed stratification of $\M_{b,\mu}^{(0,0),\red}$ indexed by the vertex lattices in $\L_{b,0}$. Finally, we give a connection between the vertex lattices in $\L_{b,0}$ and those in $C_b$, and construct a locally closed stratification of $\M_{b,\mu}^{(0)}$ indexed by the vertex lattices in $C_{b}$ by using the correspondence of vertex lattices as above and the stratification of $\M_{b,\mu}^{(0,0)}$. 

Note that the use of the above exceptional isomorphism gives us a moduli description of any irreducible component of $\M_{b,\mu}^{(0)}$. We expect an application of such a result, for example an analogue of \cite[Section 8]{Oki2019}. 

Second, we give an outline of the proof of Theorem \ref{nnlr}. The proof is essentially the same as \cite{Wu2016}. In this case, we more directly construct a locally closed stratification indexed by the vertex lattices in $C_{b}$. 

\begin{rem}
We can also describe the whole structure of the supersingular locus $\sS_{K,U}^{\si}$ of the Rapoport--Zink integral model for $(r,s)=(1,3)$ when $p$, $L$ and $K$ are as in Section \ref{mtim}. In this case, $\sS_{K,U}^{\si}\cap \sS_{K,U}^{\loc}$ is studied by \cite{Wu2016}. It consists of two basic loci, one of which is $\sS_{K,U}^{\si}\cap \sS_{K,U}^{\loc}$, and the other is $\sS_{K,U}^{\si}\setminus \sS_{K,U}^{\loc}$. The former one is purely $1$-dimensional, and each irreducible component is isomorphic to $\P^1$. The latter one is purely $2$-dimensional, and each irreducible component is isomorphic to the Fermat surface $F_p$. 

On the other hand, it is natural to consider a description of the whole structure of the supersingular locus $\sS_{K,U}^{\si}$ for $(r,s)=(1,3)$ when $p$ and $L$ are as above, and $K_p$ is the stabilizer of a \emph{self-dual} lattice in $U\otimes_{\Q}\Qp$. In this case, the structure of $\sS_{K,U}^{\si}\cap \sS_{K,U}^{\loc}$ is determinied by \cite{Rapoport2014a}. To achieve this question, it suffices to study $\sS_{K,U}^{\si}$ in the case where ``$(r,s)=(1,3)$'' is replaced with ``$(r,s)=(2,2)$''. This is one of our future's problem. 
\end{rem}

\begin{oftp}
In Section \ref{rzsp}, we define Rapoport--Zink spaces which will be considered in this paper. In Section \ref{sgrz}, we study the singularities of the Rapoport--Zink spaces, and prove Theorem \ref{sgrt}. In Section \ref{btdl}, we investigate the generalized Deligne--Lusztig varieties for even special orthogonal groups that are appeared in Theorems \ref{ntlr} and \ref{nnlr}. In Sections \ref{btn1} and \ref{btn2}, we consider the Rapoport--Zink space $\M_{b,\mu}$ in the $\mu$-neutral case. More precisely, in Section \ref{btn1}, we construct a $6$-dimensional quadratic space $\L_{b,0}$ over $\Qp$, and give the definition of vertex lattices. Moreover, we construct the Bruhat--Tits stratification of $\M_{b,\mu}^{(0,0)}$ indexed by vertex lattices in $\L_{b,0}$. In Section \ref{btn2}, we introduce the notion of the vertex lattices in $C_{b}$, and construct the Bruhat--Tits stratification of $\M_{b,\mu}^{(0)}$. In particular, we prove Theorem \ref{ntlr} in this section. In Section \ref{btnn}, we construct the Bruhat--Tits stratification of the Rapoport--Zink space $\M_{b,\mu}$ in the non-$\mu$-neutral case, and prove Theorem \ref{nnlr}. In Section \ref{slus}, we recall the precise definition of the Rapoport--Zink integral models of some Shimura varieties for $(\bG,\bX)$, and prove the results in Section \ref{mtim}. 
\end{oftp}

\begin{ack}
I would like to thank my advisor Yoichi Mieda for his constant support and encouragement. 

This work was carried out with the support from the Program for Leading Graduate Schools, MEXT, Japan. This work was also supported by the JSPS Research Fellowship for Young Scientists and KAKENHI Grant Number 19J21728.
\end{ack}

\begin{nota}
We use the following notations in this paper. 
\begin{itemize}
\item \textbf{$p$-adic fields and their extensions. } For a field $k$ of characteristic $p$, let $W(k)$ be the Cohen ring over $k$. Note that it is the ring of Witt vectors if $k$ is perfect. We denote by $\sigma$ the $p$-th power Frobenius on $k$ or its lift to $W(k)$. If $k=\Fpbar$, write $W$ and $K_0$ for $W(k)$ and $\Frac(W(k))$ respectively. We normalize the $p$-adic valuations $\ord_p$ on $\Q$, $\Qp$ and $\Frac(W(k))$ so that $\ord_p(p)=1$. 
\item For a scheme $S$, we denote by $\Irr(S)$ the set of irreducible components of $S$. 
\end{itemize}
\end{nota}

\section{Rapoport--Zink spaces for a ramified $\GU(2,2)$}\label{rzsp}

\subsection{Local PEL datum}\label{lpel}

We fix a datum which will be used until Section \ref{btnn}. Let $p$ be an odd prime number, $F$ a ramified quadratic extension field of $\Qp$, and $O_F$ the integer ring of $F$. Fix a uniformizer $\varpi$ of $O_F$, and denote by $a\mapsto \overline{a}$ the non-trivial Galois automorphism of $F$ over $\Qp$. Moreover, let $V$ be a $4$-dimensional $F$-vector space, and $(\,,\,)$ a $\Qp$-valued alternating form on $V$ satisfying $(ax,y)=(x,\overline{a}y)$ for any $a\in F$ and $x,y\in V$. Note that there is a unique $F/\Qp$-hermitian form $\langle \,,\,\rangle$ on $V$ such that
\begin{equation*}
(\,,\,)=\frac{1}{2}\tr_{F/\Qp}(\varpi^{-1}\langle\,,\,\rangle). 
\end{equation*}

In this paper, \emph{we assume that the $F/\Qp$-hermitian space $(V,\langle\,,\,\rangle)$ is split}, that is, there is an $F$-basis $\e_1,\ldots,\e_4$ of $V$ such that $\langle \e_i,\e_j\rangle=\delta_{i,5-j}$. We fix $\e_1,\ldots,\e_4$ as above, and put 
\begin{equation*}
\bLambda:=O_F\e_1\oplus O_F\e_2\oplus O_F(\varpi \e_3)\oplus O_F(\varpi \e_4). 
\end{equation*}
Then it is a $\varpi$-modular lattice, that is, $\bLambda^{\vee}=\varpi^{-1}\bLambda$. 

We define an algebraic group $G$ over $\Qp$ by
\begin{equation*}
G(R)=\{(g,c)\in \GL_{F\otimes_{\Qp}R}(V\otimes_{\Qp}R)\times R^{\times}\mid (g(x),g(y))=c(x,y)\text{ for all }x,y\in V\otimes_{\Qp}R\}
\end{equation*}
for any $\Qp$-algebra $R$. Note that $G$ is reductive and quasi-split over $\Qp$. We denote by $\sml_G \colon G\rightarrow \G_m$ the similitude character of $G$. 

We regard $G$ as a subgroup of $\Res_{F/\Qp}\GL_{F}(V)$ by the fixed basis $\e_1,\ldots,\e_4$, and put
\begin{equation*}
\mu \colon \G_{m}\rightarrow G;z\mapsto \diag(1,1,z,z). 
\end{equation*}

\subsection{$p$-divisible groups and isocrystals with additional structures}\label{pigs}

For a field extension $k/\Fpbar$, let $F_{W(k)}:=F\otimes_{\Zp}W(k)$ and $O_{F,W(k)}:=O_{F}\otimes_{\Zp}W(k)$. Note that $F_{W(k)}$ is a field and $O_{\Fbr}$ is the integer ring of $\Fbr$. Especially, we write $\Fbr$ and $O_{\Fbr}$ for $F_{W(\Fpbar)}$ and $O_{F,W(\Fpbar)}$ respectively. 

We denote by $\sigma$ the map $\id_{F}\otimes \sigma$ on $O_{F,W(k)}$. Moreover, we extend the map $a\mapsto \overline{a}$ to $F_{W(k)}$ as a $\Frac W(k)$-linear map. 

On the other hand, let $\nilp$ be the category of $W$-schemes in which $p$ is locally nilpotent. For $S\in \nilp$, put $\overline{S}:=S\times_{\spec W}\spec \Fpbar$. 

\begin{dfn}\label{pdiv}
Let $S\in \nilp$. 
\begin{enumerate}
\item A \emph{$p$-divisible group with $(G,\mu)$-structure} over $S$ is a triple $(X,\iota,\lambda)$ consisting of the following data: 
\begin{itemize}
\item $X$ is a $p$-divisible group over $S$ of dimension $4$ and height $8$, 
\item $\iota \colon O_F \rightarrow \End(X)$ is a ring homomorphism, 
\item $\lambda \colon X \rightarrow X^{\vee}$ is a quasi-isogeny so that $\lambda^{\vee}=-\lambda$, 
\end{itemize}
satisfying the following conditions for any $a\in O_F$: 
\begin{itemize}
\item (Kottwitz condition) $\det(T-\iota(a) \mid \Lie(X))=(T^2-\tr_{F/\Qp}(a)T+\N_{F/\Qp}(a))^2$, 
\item $\lambda \circ \iota(a)=\iota(\overline{a})^{\vee}\circ \lambda$. 
\end{itemize}
\item For two $p$-divisible groups with $(G,\mu)$-structures $(X,\iota,\lambda)$ and $(X',\iota',\lambda')$ over $S$, a quasi-isogeny from $(X,\iota,\lambda)$ to $(X',\iota',\lambda')$ is an $O_F$-linear quasi-isogeny $\rho \colon X\rightarrow X'$ satisfying $\rho^{\vee}\circ \lambda' \circ \rho=c(\rho)\lambda$ for some locally constant function $c(\rho)\colon \overline{S}\rightarrow \Qp^{\times}$. 
\end{enumerate}
\end{dfn}

\begin{rem}
The cocharacter $\mu$ appears in the Kottwitz condition. In fact, the right-hand side can be written as $\det(a\mid V_0)$ in the sense of \cite[3.23 a)]{Rapoport1996a}, where 
\begin{equation*}
V_0:=\{x\in V\otimes_{\Qp}K_0\mid \mu(z)x=x\text{ for all }z\in \G_m\}=\Fbr \e_1\oplus \Fbr\e_2. 
\end{equation*}
\end{rem}

It is convenient to introduce the notion of isocrystals with $G$-structures. 
\begin{dfn}
Let $k$ be an algebraically closed field of characteristic $p$. 
\begin{enumerate}
\item An isocrystal with $G$-structure over $k$ is a triple $(N,\F_{N},(\,,\,)_{N})$, where
\begin{itemize}
\item $N$ is a finite-dimensional $F\otimes_{\Zp}W(k)$-vector space,
\item $\F_{N}\colon N\rightarrow N$ is a $\sigma$-linear map,
\item $(\,,\,)_{N}\colon N\times N\rightarrow \Frac W(k)$ is a non-degenerate alternating form satisfying $(ax,y)_{N}=(x,\overline{a}y)_{N}$ for any $a\in F\otimes_{\Zp}W(k)$. 
\end{itemize}
\item For two isocrystals with $G$-structures $(N,\F_{N},(\,,\,)_{N})$ and $(N',\F_{N'},(\,,\,)_{N'})$, a morphism of isocrystals with $G$-structures is a homomorphism of $F\otimes_{\Zp}W(k)$-vector space $\varphi \colon N\xrightarrow{\cong}N'$ and $c\in (F\otimes_{\Zp}W(k))^{\times}$ such that $\varphi \circ \F_{N}=\F_{N'}\circ \varphi$ and $(\varphi(x),\varphi(y))_{N'}=c(x,y)_{N}$ for any $x,y\in N$. 
\end{enumerate}
\end{dfn}

We can construct isocrystals with $G$-structure by elements of $G(K_0)$ or $p$-divisible groups with $(G,\mu)$-structures as follow:  
\begin{enumerate}
\item Set $k:=\Fpbar$. For $b\in G(K_0)$, we define an isocrystal over $\Fpbar$ as follows: 
\begin{equation*}
(N_{b},\F_{b}):=(V\otimes_{\Qp}K_0,b(\id \otimes \sigma)). 
\end{equation*}
Then $(N_b,\F_{b},(\,,\,))$ is an isocrystal with $G$-structure over $\Fpbar$ by the natural way. Then, for two elements $b$ and $b'$ of $G(K_0)$, there is an isomorphism between $(N_b,\F_{b},(\,,\,))$ and $(N_{b'},\F_{b'},(\,,\,))$ if and only if $b$ and $b'$ are $\sigma$-conjugate, that is, there is $g\in G(K_0)$ such that $b'=gb\sigma(g)^{-1}$. Hence this correspondence induces a bijection between the set of all $\sigma$-conjugacy class of $G(K_0)$ and the set of isomorphism classes of isocrystals with $G$-structures over $\Fpbar$. 
\item Let $p$-divisible group with $G$-structure $(X,\iota,\lambda)$ over $k$. Consider the covariant rational Dieudonn{\'e} module $\D(X)_{\Q}$ of $X$. Then $\iota$ and $\lambda$ induces an $F$-action and a non-degenerate alternating form on $\D(X)_{\Q}$ respectively, and $\D(X)_{\Q}$ becomes an isocrystal with $G$-structure over $k$ by their structures. 
\end{enumerate}

\subsection{Definition of the Rapoport--Zink spaces}\label{dfrz}

Let $b\in G(K_0)$. Here we assume that $(N_b,\F_b)$ is isoclinic of slope $1/2$. Take a $p$-divisible group with $(G,\mu)$-structure $(\X_0,\iota_0,\lambda_0)$ over $\spec \Fpbar$ such that $\D(\X_0)_{\Q}\cong N_b$ as isocrystals with $G$-structures over $\Fpbar$ and $\Ker(\lambda_0)=\Ker(\iota_0(\varpi))$. We define $\M_{b,\mu}$, the \emph{Rapoport-Zink space for $(G,b,\mu)$}, as the functor that parameterizes the equivalence classes of $(X,\iota,\lambda,\rho)$ for $S\in \nilp_{W}$, where $(X,\iota,\lambda)$ is a $p$-divisible group with $(G,\mu)$-structure over $S$ satisfying $\Ker(\lambda)=\Ker(\iota(\varpi))$, and 
\begin{equation*}
\rho \colon X\times_{S}\overline{S} \rightarrow \X_0 \otimes_{\Fpbar}\overline{S}
\end{equation*}
is a quasi-isogeny of $p$-divisible groups with $(G,\mu)$-structures. Two $p$-divisible groups with $G$-structures with quasi-isogenies $(X,\iota,\lambda,\rho)$ and $(X',\iota',\lambda',\rho')$ are equivalent if $(\rho')^{-1}\circ \rho$ lifts to an isomorphism $X\rightarrow X'$ over $S$. 

The functor $\M_{b,\mu}$ is representable by a formal scheme over $\spf W$, which is formally locally of finite type. Moreover, there is a decomposition into open and closed formal subschemes
\begin{equation*}
\M_{b,\mu}=\coprod_{i\in \Z}\M_{b,\mu}^{(i)},
\end{equation*}
where $\M_{b,\mu}^{(i)}$ is the locus $(X,\iota,\lambda,\rho)$ of $\M_{b,\mu}$ where $\im(c(\rho))\subset p^i\Zpt$. These follow from \cite[Theorem 3.25]{Rapoport1996a}. 

We define another $F/\Qp$-hermitian space $C_{b}$ associated to $b$. Consider the isocrystal with $G$-structure $(N_{b},\F_{b},(\,,\,))$ over $\Fpbar$. Fix a square root $\eta \in W^{\times}$ of $p^{-1}\varpi^2$, and put $\F_{b,0}:=\eta \varpi^{-1}\F_{b}$. By \cite[p.1171]{Rapoport2014a}, we have 
\begin{equation*}
(\F_{b,0}(x),\F_{b,0}(y))=\sigma((x,y)),\quad \langle \F_{b,0}(x),\F_{b,0}(y)\rangle=\sigma(\langle x,y\rangle)
\end{equation*}
for any $x,y\in N_{b}$. Moreover, $(N_{b},\F_{b,0})$ is isoclinic of slope $0$. Let $C_{b}$ be the $\F_{b,0}$-fixed part of $N_{b}$. Then we can consider $(\,,\,)$ and $\langle\,,\,\rangle$ on $C_b$. Moreover, $(C_b,\langle\,,\,\rangle)$ is an $F/\Qp$-hermitian space of dimension $4$. 

We recall an algebraic group $J_{b}$ over $\Qp$ defined by \cite[1.12]{Rapoport1996a}. It is defined by
\begin{equation*}
J_b(R)=\{g\in G(R\otimes_{\Qp}K_0)\mid g\circ \F_b=\F_b\circ g\}
\end{equation*}
for any $\Qp$-algebra $R$. Let $\sml_{J_b}\colon J_b\rightarrow \G_m$ be the character induced by $\sml_G$. As in \cite[1.12]{Rapoport1996a}, the group $J_{b}(\Qp)$ acts on $\M_{b,\mu}$ by 
\begin{equation*}
J_{b}(\Qp)\times \M_{b,\mu}(S)\rightarrow \M_{b,\mu}(S);(g,(X,\iota,\lambda,\rho))\mapsto (X,\iota,\lambda,g\circ \rho). 
\end{equation*}
This action induces an isomorphism $\M_{b,\mu}^{(0)}\cong \M_{b,\mu}^{(i)}$ for each $i\in \Z$. Moreover, $\M_{b,\mu}^{(0)}$ is stable under $J_{b}^1(\Qp)$, where
\begin{equation*}
J_{b}^1:=\{g\in J_{b}\mid \sml_{J_b}(g)=1\}. 
\end{equation*}

The structure of $J_{b}$ is as follows: 
\begin{prop}\label{jbcb}(\cite[Proposition 1.5]{Rapoport2014a})
\emph{There is an isomorphism of $\Qp$-algebraic groups 
\begin{equation*}
J_b\cong \GU(C_b). 
\end{equation*}
Moreover, $\sml_{J_b}$ equals the composition of the similitude character of $\GU(C_b)$ and the above isomorphism. }
\end{prop}

Let $k/\Fpbar$ be a field extension. We give a description of the set of $k$-rational points of $\M_{b,\mu}^{(0)}$ by means of linear algebra. Put $N_{b,W(k)}:=N_{b}\otimes_{W}W(k)$, and let $\DiL^{\varpi}(N_{b,W(k)})$ be the set all $O_{F,W(k)}$-lattices $M$ in $N_{b,W(k)}$ satisfying the following conditions: 
\begin{itemize}
\item $M^{\vee}=\varpi^{-1}M$,
\item $pM\subset \F_{b}^{-1}(pM)\subset M$,
\item $\length_{O_{F,W(k)}}M/\F_{b}^{-1}(pM)=4$ and $\length_{O_{F,W(k)}}\varpi M+\F_{b}^{-1}(pM)/\F_{b}^{-1}(pM)\leq 2$.
\end{itemize}

\begin{prop}\label{hmlt}
\emph{There is a $J_{b}^{1}(\Qp)$-equivariant bijection
\begin{equation*}
\M_{b,\mu}^{(0)}(k)\cong \DiL^{\varpi}(N_{b,W(k)}). 
\end{equation*}}
\end{prop}

\begin{proof}
These follow from the theory of windows. See \cite{Zink2001}. Note that the third condition on $\DiL^{\varpi}(N_{b,W(k)})$ is equivalent to the Kottwitz condition. 
\end{proof}

\subsection{The neutralness condition}\label{ntrl}

Here we introduce the notion of the neutralness condition which is defined in \cite[Definition 2.5]{Rapoport2014b}. For this, we recall the \emph{Kottwitz map} defined by \cite{Kottwitz1985}. Let $B(G)$ be the set of all $\sigma$-conjugacy classes of $G(K_0)$, and $\pi_1(G)$ the Borovoi's algebraic fundamental group of $G$. We denote by $\Gamma$ the absolute Galois group of $\Qp$. Then the Kottwitz map is the map
\begin{equation*}
\overline{\kappa}_{G}\colon B(G)\rightarrow \pi_1(G)_{\Gamma}. 
\end{equation*}
By \cite[Theorem 1.15]{Rapoport1996b}, it induces an isomorphism
\begin{equation*}
B(G)_{\bs}\xrightarrow{\cong} \pi_1(G)_{\Gamma}, 
\end{equation*}
where $B(G)_{\bs}$ is the set of $\sigma$-conjugacy classes of $G(K_0)$ which is basic in the sense of \cite{Kottwitz1985}. 

We can find an explicit description of $\overline{\kappa}_{G}$ in \cite[Section 1.2]{Pappas2009} and \cite[Section 3.3]{Rapoport2017}. More precisely, there is an isomorphism
\begin{equation*}
\psi \colon \pi_1(G)_{\Gamma}\xrightarrow{\cong} \Z \times \Z/2. 
\end{equation*}
Moreover, the composite $\psi \circ \overline{\kappa}_{G}$ is induced by 
\begin{equation*}
\kappa_{G}\colon G(K_0)\rightarrow \Z \times \Z/2;b\mapsto (\ord_p(\sml_G(b)),\det \!{}_{\Fbr}(b)/\sml_G(b)^2\bmod \varpi O_{\Fbr})
\end{equation*}
(note that $\det\!{}_{\Fbr}(g)/\sml_G(g)^2$ has $\Fbr/K_0$-norm $1$). 

\begin{dfn}
Let $\mu$ be as in Section \ref{lpel}, and denote by $\mu^{\natural}$ the image of $\mu$ under the map 
\begin{equation*}
X_{*}(G)\rightarrow X_{*}(G^{\ab})_{\Gamma}\cong \pi_1(G)_{\Gamma}. 
\end{equation*}
An element $b\in G(K_0)$ or $[b]\in B(G)$ is said to be \emph{$\mu$-neutral} if $\kappa_{G}([b])=\mu^{\natural}$. 
\end{dfn}

\begin{rem}
In \cite[Definition 2.5]{Rapoport2014b}, the neutral condition is defined among elements in $B(G)$ which is \emph{acceptable} with respect to $\mu$. However, the above definition makes sense for any $b\in B(G)$. 
\end{rem}

For $b\in G(K_0)$, put $[b]:=\{gb\sigma(g)^{-1}\in G(K_0)\mid g\in G(K_0)\}$. We determine and consider the $\mu$-neutral condition on the set 
\begin{equation*}
B(G)_{1,\bs}:=\{[b]\in B(G)_{\bs}\mid \ord_p(\sml_G(b))=1\}. 
\end{equation*}
Take $[b]\in B(G)_{1,\bs}$. Then it is represented by one of the following: 
\begin{equation*}
b_0:=
\begin{pmatrix}
&&&\varpi\\
&&\varpi&\\
&\varpi &&\\
\varpi &&&
\end{pmatrix}
,\quad b_1:=
\begin{pmatrix}
&&&\varpi\\
&\varpi&0&\\
&0&-\varpi&\\
\varpi &&&
\end{pmatrix}
\end{equation*}
Hence we have an equality
\begin{equation*}
B(G)_{1,\bs}=\{[b_0],[b_1]\}. 
\end{equation*}
Note that $\kappa_{G}([b_j])=(1,j)$ for $j\in \{0,1\}$, and $(N_{b_j},\F_{b_j})$ is isoclinic of slope $1/2$. On the other hand, we have $\psi(\mu^{\natural})=(1,0)$. Hence $[b]$ is $\mu$-neutral if and only if $[b]=[b_0]$. 

We also obtained a criterion for the basic condition by means of an isocrystal. 
\begin{prop}\label{bsss}
\emph{An element $b\in G(K_0)$ satisfies $[b]\in B(G)_{1,\bs}$ if and only if $(N_{b},\F_{b})$ is isoclinic of slope $1/2$. }
\end{prop}

\begin{proof}
By the argument as above, $(N_{b},\F_{b})$ is isoclinic of slope $1/2$ if $[b]\in B(G)_{1,\bs}$. On the other hand, suppose that $(N_{b},\F_{b})$ is isoclinic of slope $1/2$. Then it is clear that $\ord_p(\sml_G(b))=1$. On the other hand, Proposition \ref{jbcb} implies that $J_{b}$ is an inner form, that is, $b$ is basic. Hence we have $[b]\in B(G)_{1,\bs}$. 
\end{proof}

\begin{rem}
Both $[b_0]$ and $[b_1]$ are acceptable with respect to $\mu$. This follows from the weak admissibility (this can be checked directly) and \cite[Proposition 2.9]{Rapoport2014b}. 
\end{rem}

Finally, we give a characterization of the $\mu$-neutralness condition for $b\in G(K_0)$ satisfying $[b]\in B(G)_{1,\bs}$ by means of linear algebra. 

\begin{prop}\label{nteq}
\emph{Let $b\in G(K_0)$ satisfying $[b]\in B(G)_{1,\bs}$. Then the following are equivalent: 
\begin{enumerate}
\item $[b]$ is $\mu$-neutral,
\item $\det_{\Fbr}(b)/\sml_G(b)\in 1+\varpi O_{\Fbr}$,
\item $C_{b}$ is split as an $F/\Qp$-hermitian space, 
\item $\length_{O_{\Fbr}}\varpi M+\bV_{b}(M)/\bV_{b}(M)$ is an even integer for any (some) $M\in \DiL^{\varpi}(N_{b})$, where $\bV_{b}:=p\F_{b}^{-1}$. 
\end{enumerate}}
\end{prop}

\begin{proof}
The assertion (i) $\Longleftrightarrow$ (ii) follows from the description of the Kottwitz map as above. The assertion (iii) $\Longleftrightarrow$ (iv) follow from \cite[Lemma 3.3]{Rapoport2017} and the isomorphism
\begin{equation*}
\bV_{b}\colon M+\F_{b,0}(M)/M\xrightarrow{\cong}\varpi M+\bV_{b}(M)/\bV_{b}(M). 
\end{equation*}
We prove (i) $\Longleftrightarrow$ (iv). We may assume that $b$ equals either $b_0$ or $b_1$. Then we have $\bLambda \in \DiL^{\varpi}(N_{b_j})$ and 
\begin{equation*}
\length_{O_{\Fbr}}\varpi \bLambda+\bV_{b_j}(\bLambda)/\bV_{b_j}(\bLambda)=j
\end{equation*}
for $j\in \{0,1\}$. Hence the equivalence between (i) and (iv) follows. 
\end{proof}

\section{Singularities of the Rapoport--Zink spaces}\label{sgrz}

In this section, we prove the following: 

\begin{thm}\label{sgrt}
\emph{Let $b\in G(K_0)$ be as in Section \ref{dfrz}. 
\begin{enumerate}
\item If $b$ is $\mu$-neutral, then the formal scheme $\M_{b,\mu}$ is regular. Moreover, it is formally smooth over $\spf W$ outside the discrete set of $\Fpbar$-rational points such that the corresponding $p$-divisible group with $(G,\mu)$-structure $(X,\iota,\lambda,\rho)$ over $\Fpbar$ satisfies $\iota(\varpi)=0$ on $\Lie(X)$. 
\item If $b$ is not $\mu$-neutral, then we have $\M_{b,\mu}=\M_{b,\mu}\times_{\spf W}\spec \Fpbar$, and it is formally smooth of dimension $3$ over $\spec \Fpbar$. 
\end{enumerate}}
\end{thm}

To prove Theorem \ref{sgrt}, we determine the singularity of certain scheme, which is called the local model for $(G,\mu)$. Let $\bLambda$ be as in Section \ref{lpel}, and put $\bLambda_0:=\varpi^{-1}\bLambda$. 
\begin{dfn}\label{nvhn}
\begin{enumerate}
\item We define an $\Zp$-scheme $M_{\mu}^{\naive}$ as the functor which parametrizes $\mathcal{F}$ of $\bLambda_{0,\O_S}:=\bLambda_0\otimes_{\Zp}\O_S$ satisfying the following conditions: 
\begin{itemize}
\item $\det(T-a\mid \mathcal{F})=(T^2-\tr_{F/\Qp}(a)T+\N_{F/\Qp}(a))^2$ for any $a\in O_F$,
\item $\mathcal{F}$ is a maximally totally isotropic subspace with respect to the pairing
\begin{equation}\label{smpr}
\bLambda_{0,\O_S}\times \bLambda_{0,\O_S}\xrightarrow{\varpi \times \id}
\bLambda_{\O_S}\times \bLambda_{0,\O_S}\xrightarrow{(\,,\,)}\O_S. 
\end{equation}
Note that $M_{\mu}^{\naive}$ is representable by a projective scheme over $\Zp$. 
\end{itemize}
\item We define $M_{\mu}^{\loc}$ as the scheme theoretic closure of the generic fiber $M_{\mu,\Qp}^{\naive}$ of $M_{\mu}^{\naive}$. 
\end{enumerate}
\end{dfn}

We denote by $[\,,\,]_0$ the pairing (\ref{smpr}), and define an affine group scheme $\mathcal{G}$ over $\Zp$ by
\begin{equation*}
\mathcal{G}(R)=\{(g,c)\in \GL_{O_F\otimes_{\Zp}R}(\bLambda_0\otimes_{\Zp}R)\times R^{\times}\mid [g(x),g(y)]_0=c[x,y]_0\text{ for all }x,y\in \bLambda_0\otimes_{\Zp}R\}
\end{equation*}
for any $\Zp$-algebra $R$. Note that we have $\mathcal{G}(\Qp)=G(\Qp)$, and $\mathcal{G}(\Zp)$ is the stabilizer of $\bLambda_0$ in $G(\Qp)$. Moreover, there is an action of $\mathcal{G}$ on $M_{\mu}^{\naive}$ by 
\begin{equation*}
\mathcal{G}(S)\times M_{\mu}^{\naive}(S)\rightarrow M_{\mu}^{\naive}(S);(g,\mathcal{F})\mapsto g(\mathcal{F})
\end{equation*}
for any $\Zp$-scheme $S$. 

\begin{prop}\label{lcmd}
\emph{
\begin{enumerate}
\item There is a decomposition into open and closed subschemes 
\begin{equation*}
M_{\mu}^{\naive}=M_{\mu}^{\loc}\sqcup (M_{\mu}^{\naive}\setminus M_{\mu}^{\loc}). 
\end{equation*}
\item The scheme $M_{\mu}^{\loc}$ is regular and irreducible of dimension $5$. The singular locus of $M_{\mu}^{\loc}$ consists the unique $\Fp$-rational point $y_0$ such that the corresponding subspace $\mathcal{F}_0$ of $\Lambda_{0,\Fp}$ satisfying $\varpi \mathcal{F}_0=0$. Moreover, there is an isomorphism
\begin{equation*}
\O_{M_{\mu}^{\loc},y_0}\cong (\Zp[t_0,t_1,t_2,t_3,t_4]/(t_0^2+t_1t_2+t_3t_4-\varpi^2))_{(p,t_0,t_1,t_2,t_3,t_4)}. 
\end{equation*}
\item The scheme $M_{\mu}^{\naive}\setminus M_{\mu}^{\loc}$ is the locus of $M_{\mu}^{\naive}$ where the rank of $\varpi$ on the corresponding subspace equals $1$. Moreover, it is a smooth $\Fp$-scheme of dimension $3$. 
\end{enumerate}}
\end{prop}

\begin{proof}[Proposition \ref{lcmd} $\Rightarrow$ Proposition \ref{sgrt}]
By \cite[Theorem 3.33]{Rapoport1996a}, there is a diagram
\begin{equation*}
\M_{b,\mu}\xleftarrow{\varphi} \widetilde{\M}_{b,\mu}\xrightarrow{\psi} \widehat{M}_{\mu}^{\naive}, 
\end{equation*}
where $\widehat{M}_{\mu}^{\naive}$ is the $p$-adic completion of $M_{\mu}^{\naive}\times_{\spec \Zp}\spec W$. Note that the map $\varphi$ is a $\mathcal{G}$-torsor, and $\psi$ is formally smooth. Moreover, there is a decomposition of open closed formal subschemes corresponding to Proposition \ref{lcmd} (i): 
\begin{equation*}
\widehat{M}_{\mu}^{\naive}=\widehat{M}_{\mu}^{\loc}\sqcup (M_{\mu}^{\naive}\setminus M_{\mu}^{\loc}). 
\end{equation*}
Now take $x\in \M_{b,\mu}(\Fp)$ and $y\in \widetilde{\M}_{b,\mu}(\Fpbar)$ satisfying $\varphi(y)=x$. By Propositions \ref{hmlt} and \ref{nteq}, the image of $y$ under $\psi$ lies in $\widehat{M}_{\mu}^{\loc}$ if $b$ is $\mu$-neutral, and $M_{\mu}^{\naive}\setminus M_{\mu}^{\loc}$ if $b$ is not $\mu$-neutral. Hence the assertions follow from Proposition \ref{lcmd} (ii), (iii). 
\end{proof}

We also obtain a moduli description of $M_{\mu}^{\loc}$ as follows: 
\begin{cor}
The scheme $M_{\mu}^{\loc}$ equals the spin local model $M_{\mu}^{\spin}$ (see \cite{Pappas2009} or \cite{Smithling2014} for the definition of the spin local model). 
\end{cor}

\begin{proof}
As explained in \cite[2.4]{Smithling2014}, $M_{\mu}^{\spin}$ is a closed subscheme of $M_{\mu}^{\naive}$. Since $(M_{\mu}^{\spin})^{\red}=M_{\mu}^{\loc}$ by \cite[Theorem 1.3]{Smithling2014}, Proposition \ref{lcmd} (i) implies that the canonical closed immersion $M_{\mu}^{\spin}\hookrightarrow M_{\mu}^{\naive}$ factors through $M_{\mu}^{\loc}\hookrightarrow M_{\mu}^{\naive}$. Hence the assertion follows. 
\end{proof}

We turn to prove Proposition \ref{lcmd}. For it, we recall the structure of the generic fiber of $M_{\mu}^{\loc}$. 

\begin{prop}\label{gnfb} (\cite[1.5.3]{Pappas2009})
\emph{The generic fiber $M_{\mu,\Qp}^{\naive}$ of $M_{\mu}^{\naive}$ is connected and smooth of dimension $4$ over $\Qp$. }
\end{prop}

We rewrite a $\Zp$-basis of $\bLambda_0$ as follows: 
\begin{gather*}
\f_1:=\varpi^{-1}\e_1,\quad \f_2:=\varpi^{-1}\e_2,\quad \f_3:=\e_3,\quad \f_4:=\e_4,\\
\f_5:=\e_1,\quad \f_6:=\e_2,\quad \f_7:=\varpi \e_3,\quad \f_8:=\varpi \e_4. 
\end{gather*}
Then the Gram matrix of $[\,,\,]_0$ with respect to the above basis is
\begin{equation*}
\mathbf{J}=
\begin{pmatrix}
&&&J_2\\
&&-J_2&\\
&-J_2&&\\
J_2&&&
\end{pmatrix},
\end{equation*}
where $J_2$ is the anti-diagonal matrix of size $2$ with coefficient $1$. 

Let $\OGr(\bLambda_{0})$ be the Grassmanian of maximally totally isotropic subspaces of $\bLambda_{0}$. Note that it has exactly two connected components. Then there is a natural closed immersion
\begin{equation*}
M_{\mu}^{\naive}\rightarrow \OGr(\bLambda_{0}). 
\end{equation*}
Since $M_{\mu,\Qp}^{\naive}$ is connected by Proposition \ref{gnfb}, there is a unique connected component $\OGr^{+}(\bLambda_{0})$ of $\OGr(\bLambda_{0})$ containing the image of $M_{\mu,\Qp}^{\naive}$ under the above closed immersion. We denote by $\OGr^{-}(\bLambda_{0})$ the other connected component of $\OGr(\bLambda_{0})$. 

We write $M_{\mu,\Fp}^{\naive}$ for the special fiber of $M_{\mu}^{\naive}$. For $r\in \{0,1,2\}$, let $C_{r}$ be the locus of $M_{\mu,\Fp}^{\naive}$ where the rank of $\varpi$ on the corresponding subspace equals $r$. Then we have a decomposition into locally closed subschemes
\begin{equation*}
M_{\mu,\Fp}^{\naive}=C_0\sqcup C_1\sqcup C_2. 
\end{equation*}
Note that the above decomposition is also the decomposition into $\mathcal{G}$-orbits. Moreover, we have
\begin{equation*}
M_{\mu,\Fp}^{\naive}\cap \OGr^{+}(\bLambda_{0})=C_0\sqcup C_2,\quad M_{\mu,\Fp}^{\naive}\cap \OGr^{-}(\bLambda_{0})=C_1
\end{equation*}
by \cite[7.1.4]{Pappas2009}. In particular, $C_1$ is open and closed in $M_{\mu,\Fp}^{\naive}$. 

\begin{rem}
The spin local model $M_{\mu}^{\spin}$ equals $M_{\mu,\Fp}^{\naive}\cap \OGr^{+}(\bLambda_{0})=C_0\sqcup C_2$ by the definition of the spin condition. See \cite[7.1.4, 7.2.1]{Pappas2009}. 
\end{rem}

Now put
\begin{align*}
\mathcal{F}_0&:=\Fp \f_5\oplus \Fp \f_6\oplus \Fp \f_7\oplus \Fp \f_8,\\
\mathcal{F}_1&:=\Fp \f_1\oplus \Fp \f_5\oplus \Fp \f_6\oplus \Fp \f_7,\\
\mathcal{F}_2&:=\Fp \f_1\oplus \Fp \f_2\oplus \Fp \f_5\oplus \Fp \f_6. 
\end{align*}
Then we have $\mathcal{F}_r\in C_r(\Fp)$ for all $r$. Moreover, for $r\in \{0,1,2\}$, put
\begin{equation*}
\mathcal{F}_{r}^{c}:=\bigoplus_{1\leq i\leq 8,e_i\not\in \mathcal{F}_r}\Fp \f_i,
\end{equation*}
and define an open neighborhood $U_{r}$ of $\mathcal{F}_r$ in $M_{\mu}^{\naive}$ by
\begin{equation*}
U_{r}(S)=\{\mathcal{F}_{r,f}\in M_{\mu,\Fp}^{\naive}\mid f\in \Hom_{\O_S}(\mathcal{F}_r,\mathcal{F}_{r}^{c})\},
\end{equation*}
for any $\Zp$-scheme $S$, where
\begin{equation*}
\mathcal{F}_{r,f}:=\{x+f(x)\in \bLambda_{0,\O_S}\mid x\in \mathcal{F}_r\}. 
\end{equation*}
Then we can represent $\mathcal{F}_{r,f}$ as a column space with respect to $\f_1,\ldots ,\f_8$ whose submatrix with respect to $\{i\in \{1,\ldots,8\} \mid \f_i \not\in \mathcal{F}_r\}$ is a unit matrix of size $4$. 

To study $U_r$, we express the isotropicity condition (with respect to $[\,,\,]_0$) and the $\varpi$-stability condition by means of column spaces. Let
\begin{equation*}
\begin{pmatrix}
X\\Y
\end{pmatrix}
\in M_{\mu}. 
\end{equation*}
Then the isotropicity condition is equivalent to
\begin{equation}\label{ttis}
\begin{pmatrix}
{}^tX&{}^tY
\end{pmatrix}
\mathbf{J}
\begin{pmatrix}
X\\Y
\end{pmatrix}
=0. 
\end{equation}
On the other hand, the $\varpi$-stability condition is equivalent to
\begin{equation}\label{pist}
\begin{pmatrix}
\varpi^2 Y\\X
\end{pmatrix}
=
\begin{pmatrix}
X\\Y
\end{pmatrix}
R
\end{equation}
for some $4\times 4$-matrix $R$. Moreover, the Kottwitz condition is equivalent to the equality
\begin{equation}\label{kott}
\det(T-R\mid \mathcal{F})=(T-\varpi^2)^2. 
\end{equation}

\begin{proof}[Proof of Proposition \ref{lcmd} (iii)]
By the above consideration, it suffices to prove that $U_{1}$ is a smooth $\Fp$-scheme of dimension $3$. The open set $U_1$ can be written as
\begin{equation*}
\begin{pmatrix}
1&&&\\
x_1&z_{11}&z_{12}&z_{13}\\
x_2&z_{21}&z_{22}&z_{23}\\
x_3&z_{31}&z_{32}&z_{33}\\
&1&&\\
&&1&\\
&&&1\\
y_1&y_2&y_3&y_4
\end{pmatrix}. 
\end{equation*}
Then (\ref{ttis}) is equivalent to the following equalities: 
\begin{gather*}
y_1=0,\quad y_2=x_3,\quad y_3=x_2,\quad y_4=-x_1,\\
z_{13}=z_{22}=z_{31}=0,\quad z_{23}=z_{12},\quad z_{32}=-z_{21},\quad z_{33}=z_{11}. 
\end{gather*}
Hence the matrix $R$ in (\ref{pist}) equals
\begin{equation*}
\begin{pmatrix}
&\varpi^{2}&&\\
1&&&\\
x_1&z_{11}&z_{12}&0\\
x_2&z_{21}&0&z_{12}
\end{pmatrix}. 
\end{equation*}
Hence (\ref{kott}) is equivalent to the equalities
\begin{equation*}
z_{12}=0,\quad p=0. 
\end{equation*}
In particular, we have $R^2=0$. This is equivalent to
\begin{equation*}
z_{11}=z_{21}=0. 
\end{equation*}
Consequently, we have $U_{1}\cong \spec \Fp[x_1,x_2,x_3]$ as desired. 
\end{proof}

\begin{proof}[Proof of Proposition \ref{lcmd} (i), (ii)]
By Proposition \ref{gnfb} and the above consideration, it is enough to prove that there are isomorphisms
\begin{equation*}
U_0\cong \spec \Zp[t_0,t_1,t_2,t_3,t_4]/(t_0^2+t_1t_2+t_3t_4-\varpi^2),\quad U_{2}\cong \A_{\Zp}^3. 
\end{equation*}

First, we consider the open set $U_{0}$. Then $U_0$ can be written as
\begin{equation*}
\begin{pmatrix}
Y\\E_4
\end{pmatrix},
\end{equation*}
where $E_4$ is the unit matrix of size $4$. If we write $Y=(y_{ij})_{1\leq i,j\leq 4}$, then (\ref{ttis}) is equivalent to the following equalities: 
\begin{gather*}
y_{11}=-y_{22}=-y_{33}=y_{44},\quad y_{12}=y_{34},\quad y_{13}=-y_{24},\\
y_{14}=y_{23}=y_{32}=y_{41}=0,\quad y_{21}=y_{43},\quad y_{31}=-y_{42}
\end{gather*}
On the other hand, the matrix in (\ref{pist}) equals $Y$, and hence we have $Y^2=\varpi^2E_4$. Note that (\ref{kott}) is then satisfied. The equality $Y^2=\varpi^2E_4$ is equivalent to
\begin{equation*}
y_{11}^2+y_{12}y_{21}+y_{13}y_{31}=\varpi^2. 
\end{equation*}
Therefore we have $U_0\cong \spec \Zp[y_{11},y_{12},y_{21},y_{13},y_{31}]/(y_{11}^2+y_{12}y_{21}+y_{13}y_{31}-\varpi^2)$ as desired. 

Second, we consider the open set $U_{2}$. Then $U_2$ can be written as
\begin{equation*}
\begin{pmatrix}
E_2&0\\
X_1&X_2\\
0&E_2\\
Y_1&Y_2
\end{pmatrix}, 
\end{equation*}
where $E_2$ is the unit matrix of size $2$. Then the matrix $R$ in (\ref{pist}) equals
\begin{equation*}
\begin{pmatrix}
0&\varpi^2E_2\\
E_2&0
\end{pmatrix}. 
\end{equation*}
Note that (\ref{kott}) is then satisfied. Moreover, (\ref{pist}) amounts to saying that 
\begin{equation*}
X_1=Y_2,\quad X_2=\varpi^2Y_1. 
\end{equation*}
On the other hand, if we write $Y_k=(y_{ij}^{(k)})_{1\leq i,j\leq 2}$ for $k\in \{1,2\}$, then (\ref{ttis}) is equivalent to the following equalities: 
\begin{equation*}
y_{11}^{(1)}=-y_{22}^{(1)},\quad y_{12}^{(1)}=y_{21}^{(1)}=0,\quad y_{11}^{(2)}=y_{22}^{(2)},\quad y_{12}^{(2)}=y_{21}^{(2)}. 
\end{equation*}
Therefore we have $U_2\cong \Zp[y_{11}^{(1)},y_{12}^{(1)},y_{11}^{(2)},y_{12}^{(2)}]$ as desired. 
\end{proof}

\section{Deligne--Lusztig varieties for even special orthogonal groups}\label{btdl}

In this section, we study some generalized Deligne--Lusztig varieties for even special orthogonal groups. 

First, we recall the definition of generalized Deligne--Lusztig varieties. Let $G_0$ be a connected reductive group of $\Fp$, and put $G:=G_0\otimes_{\Fp}\Fpbar$. We denote by $\Phi$ the Frobenius of $G$. Moreover, fix a $\Phi$-stable maximal torus $T$ of $G$, and a $\Phi$-stable Borel subgroup $B$ containing $T$. Let $W=W(B,T)$ be the Weyl group associated to $(B,T)$, and write $\Delta^{*}=\Delta^{*}(B,T)$ for the set of simple roots in $W$. 

For $I\subset \Delta^{*}$, let $W_{I}$ be the subgroup of $W$ generated by $I$. Moreover, put $P_{I}:=BW_{I}B$, which is a parabolic subgroup of $G$. Then there is a canonical isomorphism $\pi \colon P_{I}\backslash G/P_{\Phi(I)}\xrightarrow{\cong} W_{I}\backslash W/W_{\Phi(I)}$. Hence we obtain a map 
\begin{equation*}
\inv \colon G/P_{I}\times G/P_{\Phi(I)}\xrightarrow{(g,g')\mapsto g^{-1}g'}P_{I}\backslash G/P_{\Phi(I)}\xrightarrow{\pi} W_{I}\backslash W/W_{\Phi(I)}. 
\end{equation*}

\begin{dfn}
Let $I\subset \Delta^{*}$ and $w\in W_{I}\backslash W/W_{\Phi(I)}$. A \emph{generalized Deligne--Lusztig variety} for $(G_{0},I,w)$ is a locally closed subscheme of $G/P_{I}$ as follows: 
\begin{equation*}
X_{P_{I}}(w):=\{g\in G/P_{I}\mid \inv(g,\Phi(g))=w\}. 
\end{equation*}
\end{dfn}

Here we consider them for even special orthogonal groups. Let $(\Omega_{2d,0},[\,,\,])$ be a quadratic space over $\Fp$ of dimension $2d\geq 4$, and put $\Omega_{2d}:=\Omega_{2d,0}\otimes_{\Fp}\Fpbar$. Fix a basis $e_1,\ldots,e_d,f_1,\ldots,f_d$ of $\Omega_{2d}$ satisfying $[e_i,e_j]=[f_i,f_j]=0,[e_i,f_j]=\delta_{ij}$, and 
\begin{equation*}
\Phi(e_i)=
\begin{cases}
e_i&\text{if }1\leq i\leq d-1,\\
e_d&\text{if }i=d\text{ and }\Omega_{2d,0}\text{ is split},\\
f_d&\text{if }i=d\text{ and }\Omega_{2d,0}\text{ is non-split},
\end{cases}
\quad \Phi(f_i)=
\begin{cases}
f_i&\text{if }1\leq i\leq d-1,\\
f_d&\text{if }i=d\text{ and }\Omega_{2d,0}\text{ is split},\\
e_d&\text{if }i=d\text{ and }\Omega_{2d,0}\text{ is non-split}. 
\end{cases}
\end{equation*}
Moreover, put
\begin{gather*}
\cL_{2d,i}:=\Fpbar e_{1}\oplus \cdots \Fpbar e_{i}\,(i\in \{1,\ldots,d-1\}),\\
\cL_{2d,d}^{+}:=\cL_{2d,d-1}\oplus \Fpbar e_{d},\quad \cL_{2d,d}^{-}:=\cL_{2d,d-1}\oplus \Fpbar f_{d}. 
\end{gather*}

Put $G_{2d,0}:=\SO(\Omega_{2d,0})$ and $G_{2d}:=G_{2d,0}\otimes_{\Fp}\Fpbar$. We regard $G_{2d}$ as a subgroup of $\GL(\Omega_{2d})$ by the above basis. Furthermore, let $B_{2d}\subset G_{2d}$ be the upper-half Borel subgroup, and $T_{2d}\subset G_{2d}$ the diagonal torus. Then there is an isomorphism
\begin{equation*}
W_{2d}:=W(B_{2d},T_{2d})\cong \{w\in \mathfrak{A}_{2d}\mid w(i)+w(2d+1-i)=2d+1\text{ for all }i\in \{1,\ldots,2d\}\}. 
\end{equation*}
Moreover, we have $\Delta_{2d}^{*}:=\Delta^{*}(B_{2d},T_{2d})=\{s_{2d,1},\ldots,s_{2d,d-2},t_{2d}^{+},t_{2d}^{-}\}$, where
\begin{itemize}
\item $s_{2d,i}=(i\ i+1)(2d-i\ 2d+1-i)$ for $i\in \{1,\ldots,d-2\}$,
\item $t_{2d}^{+}=(d-1\ d)(d+1\ d+2)$,
\item $t_{2d}^{-}=(d-1\ d+1)(d\ d+2)$. 
\end{itemize}
Note that we have $\Phi(s_{2d,i})=s_{2d,i}$ for any $i$ and
\begin{equation*}
\Phi(t_{2d}^{\pm})=
\begin{cases}
t_{2d}^{\pm}&\text{if }\Omega_{2d,0}\text{ is split},\\
t_{2d}^{\mp}&\text{if }\Omega_{2d,0}\text{ is non-split}.
\end{cases}
\end{equation*}

\begin{dfn}\label{grso}
For a quadratic space $\Omega_{0}$ over $\Fp$ of dimension $n$, we denote by $\OGr(\Omega_0)_{\Fpbar}$ be the moduli space of maximal totally isotropic subspaces of $\Omega_0\otimes_{\Fp}\Fpbar$. Note that all maximal totally isotropic subspaces of $\Omega_0\otimes_{\Fp}\Fpbar$ are $[n/2]$-dimensional, where $[n/2]:=\max\{n'\in \Z\mid n'\leq n/2\}$. Moreover, put
\begin{equation*}
S_{\Omega_0,\Fpbar}:=\{\cL \in \OGr(\Omega_0)_{\Fpbar}\mid \rk(\cL \cap (\id \otimes \sigma)_{*}(\cL))\geq m\}, 
\end{equation*}
where
\begin{equation*}
m:=
\begin{cases}
\frac{n}{2}-2&\text{if $n\in 2\Z$ and $\Omega_{0}$ is split}, \\
\frac{n}{2}-1&\text{if $n\in 2\Z$ and $\Omega_{0}$ is non-split}, \\
\frac{n-1}{2}&\text{if $n\in \Z\setminus 2\Z$},
\end{cases}
\end{equation*}
Here $(\id \otimes \sigma)_{*}(\cL)$ is the submodule generated by $(\id \otimes \sigma)(\cL)$. It is a closed subscheme of $\OGr(\Omega_0)_{\Fpbar}$ by definition.  
\end{dfn}

We introduce subschemes of $\OGr(\Omega_{0})_{\Fpbar}$ and $S_{\Omega_0,\Fpbar}$ in the case $n=2d$, that is, $\Omega_0=\Omega_{2d,0}$. Let $P_{2d,0}^{+}$ and $P_{2d,0}^{-}$ be the stabilizers of $\cL_{2d,d}^{+}$ and $\cL_{2d,d}^{-}$ in $G_{2d}$ respectively, which are parabolic subgroups of $G_{2d}$. Then the maps $g\mapsto g(\cL_{2d,d}^{+})$ and $g\mapsto g(\cL_{2d,d}^{-})$ induce an isomorphism
\begin{equation*}
(G_{2d}/P_{2d,0}^{+})\sqcup (G_{2d}/P_{2d,0}^{-})\cong \OGr(\Omega_{2d,0})_{\Fpbar}. 
\end{equation*}
We denote by $\OGr^{\pm}(\Omega_{2d,0})_{\Fpbar}$ the image of $G_{2d}/P_{2d,0}^{\pm}$ under the isomorphism above. Note that $\OGr^{\pm}(\Omega_{2d,0})_{\Fpbar}$ are the connected components of $\OGr(\Omega_{2d,0})_{\Fpbar}$. Then, for two elements $\cL,\cL'$ in $\OGr(\Omega_{2d,0})_{\Fpbar}(k)$ where $k$ is an algebraic closed field extension over $\Fpbar$, $\cL$ and $\cL'$ are contained in the same connected component if and only if $\dim_{k}(\cL/\cL\cap \cL')=\dim_{k}(\cL'/\cL\cap \cL')\in 2\Z$. On the other hand, put $S_{\Omega_{2d,0}}^{\pm}:=S_{\Omega_{2d,0}}\cap \OGr^{\pm}(\Omega_{2d,0})_{\Fpbar}$. 

Now we give a connection between $S_{\Omega_{2d,0}}$ and generalized Deligne--Lusztig varieties for $G_{2d,0}$. 

\textbf{Case 1. The case for $\Omega_{2d,0}$ split. }
For $i\in \Znn$, put 
\begin{equation*}
P_{2d,i}^{\pm}:=
\begin{cases}
P_{\{s_{2d,1},\ldots,s_{2d,d-2-i},t_{2d}^{\pm}\}}&\text{if }i\leq d-3,\\
P_{\{t_{2d}^{\pm}\}}&\text{otherwise}. 
\end{cases}
\end{equation*}
Then the following hold: 
\begin{itemize}
\item $P_{2d,0}^{\pm}$ is the stabilizer of $\cL_{2d,d}^{\pm}$,
\item $P_{2d,i}^{\pm}$ is the stabilizer of the flag $\cL_{2d,d-1-i}\subset \cdots \subset \cL_{2d,d-2}\subset \cL_{2d,d}^{\pm}$ if $1\leq i\leq d-3$, 
\item $P_{2d,i}^{\pm}$ is the stabilizer of the flag $\cL_{2d,1}\subset \cdots \subset \cL_{2d,d-2}\subset \cL_{2d,d}^{\pm}$ if $i\geq d-2$. 
\end{itemize}

\begin{prop}\label{sfdc}
\emph{
\begin{enumerate}
\item There is a locally closed stratification
\begin{equation*}
S_{\Omega_{2d,0}}^{\pm}=X_{P_{2d,0}^{\pm}}(\id)\sqcup X_{P_{2d,0}^{\pm}}(t_{2d}^{\mp}). 
\end{equation*}
Moreover, the closure of $X_{P_{2d,0}^{\pm}}(t_{2d}^{\mp})$ in $G_{2d}/P_{2d,0}^{\pm}$ equals $S_{\Omega_{2d,0}}^{\pm}$. 
\item The schemes $S_{\Omega_{2d,0}}^{\pm}$ are irreducible. 
\end{enumerate}}
\end{prop}

\begin{proof}
(i): The proof is the same as \cite[Proposition 5.3]{Rapoport2014a} or \cite[Lemma 3.7]{Howard2014}. 

(ii): It suffices to prove the irreducibility of $X_{P_{2d,0}^{\pm}}(t_{2d}^{\mp})$ by (i). However, this follows from \cite[Theorem 2]{Bonnafe2006}. 
\end{proof}

\begin{prop}
\emph{There is a locally closed stratification
\begin{equation*}
S_{\Omega_{6,0}}^{\pm}=X_{P_{6,0}^{\pm}}(\id)\sqcup X_{P_{6,1}^{\pm}}(t_{6}^{\mp})\sqcup X_{P_{6,1}^{\pm}}(t_{6}^{\mp}s_{6,1}). 
\end{equation*}
More precisely, there are locally closed subschemes $X_0^{\pm}$ and $X_1^{\pm}$ of $S_{\Omega_{6,0}^{\pm}}$ satsifying $X_0^{\pm}\cong X_{P_{6,1}^{\pm}}(t_{6}^{\mp})$ and $X_1^{\pm}\cong X_{P_{6,1}^{\pm}}(t_{6}^{\mp}s_{6,1})$ under the canonical morphism $G_{6}/P_{6,0}^{\pm}\rightarrow G_{6}/P_{6,1}^{\pm}$. Furthermore, we have 
\begin{equation*}
\overline{X_{P_{6,1}^{\pm}}(t_{6}^{\mp})}=X_{P_{6,0}^{\pm}}(\id)\sqcup X_{P_{6,1}^{\pm}}(t_{6}^{\mp}),\quad \overline{X_{P_{6,1}^{\pm}}(t_{6}^{\mp}s_{1})}=S_{\Omega_{6,0}}^{\pm}. 
\end{equation*}
Here we denote by $\overline{X}$ the closure of a subset $X$ of $G_{6}/P_{6,1}^{\pm}$. }
\end{prop}

\begin{proof}
This follows from the same argument as the proof of \cite[Proposition 5.5]{Rapoport2014a} by using Propositoion \ref{sfdc}. 
\end{proof}

\begin{rem}
If $d\geq 4$, then we have
\begin{equation*}
S_{\Omega_{2d,0}}^{\pm}=X_{P_{2d,0}^{\pm}}(\id)\sqcup X_{P_{2d,1}^{\pm}}(t_{2d}^{\mp})\sqcup X_{P_{2d,1}^{\pm}}(t_{2d}^{\mp}s_{2d,d-2})\sqcup X_{P_{2d,1}^{\pm}}(s_{2d,d-2}s_{2d,d-3}t_{2d}^{\mp}s_{2d,d-2}). 
\end{equation*}
The variety $X_{P_{2d,1}^{\pm}}(s_{2d,d-2}s_{2d,d-3}t_{2d}^{\mp}s_{2d,d-2})$ parametrizes all flags $\cL_{d-2} \subset \cL_{d}$ in $\Omega_{2d}$ where $\rk \cL_{i}=i$ satisfying $\cL_{d-2}=\cL_{d}\cap \Phi(\cL_{d})$ and $\rk(\cL_{d-2}/\cL_{d-2}\cap \Phi(\cL_{d-2}))=2$. 
\end{rem}

Next, we give a connection between generalized Deligne--Lusztig varieties for $\SO_{2d-1}$. Let $\omega_{d}:=e_{d}-f_{d}\in \Omega_{2d,0}$, which is an anisotropic vector. We denote by $\Omega_{2d-1,0}$ the perpendicular of $\omega_{d}$ in $\Omega_{2d,0}$. Then we have $\dim_{\Fp}(\Omega_{2d-1,0})=2d-1$. Moreover, for $i\in \{1,\ldots,d-1\}$, put
\begin{equation*}
\cL_{2d-1,i}:=\cL_{2d,i+1}\cap \Omega_{2d-1}. 
\end{equation*}

On the other hand, put $G_{2d-1,0}:=\SO(\Omega_{2d-1,0})$ and $G_{2d-1}:=G_{2d-1,0}\otimes_{\Fp}\Fpbar$. We regard $G_{2d-1}$ as a subgroup of $\GL_{2d-1}$ by the basis $e_1,\ldots,e_{d-1},e_{d}+f_{d},f_{d-1},\ldots,f_{1}$. Moreover, let $B_{2d-1}\subset G_{2d-1}$ be the upper-half Borel subgroup, and $T_{2d-1}\subset G_{2d-1}$ the diagonal torus. Then there is an isomorphism 
\begin{equation*}
W_{2d-1}:=W(B_{2d-1},T_{2d-1})\cong \{w\in \mathfrak{S}_{2d-1}\mid w(i)+w(2d-i)=2d\text{ for all }i\in \{1,\ldots,2d-1\}\}. 
\end{equation*}
Moreover, we have $\Delta_{2d-1}^{*}:=\Delta^{*}(B_{2d-1},T_{2d-1})=\{s_{2d-1,1},\ldots,s_{2d-1,d-2},t_{2d-1}\}$, where
\begin{itemize}
\item $s_{2d-1,i}=(i\ i+1)(2d-i\ 2d+1-i)$ for $i\in \{1,\ldots,d-2\}$,
\item $t_{2d-1}=(d-1\ d+1)$. 
\end{itemize}
For $i\in \Znn$, put 
\begin{equation*}
P_{2d-1,i}:=
\begin{cases}
P_{\{s_{2d-1,1},\ldots,s_{2d-1,d-2-i}\}}&\text{if }i\leq d-3,\\
B_{2d-1}&\text{otherwise}.
\end{cases} 
\end{equation*}
Then $P_{2d-1,i}$ is the stabilizer of the flag $\cL_{2d-1,d-1-i}\subset \cdots \subset \cL_{2d-1,d-1}$. Moreover, there is an isomorphism
\begin{equation*}
G_{2d-1}/P_{2d-1,0}\xrightarrow{\cong}\OGr(\Omega_{2d-1,0})_{\Fpbar};g\mapsto g(\cL_{2d-1,d-1}). 
\end{equation*}

\begin{lem}\label{soeo}
\emph{There is a commutative diagram
\begin{equation*}
\xymatrix@C=64pt{
G_{2d}/P_{2d,0}^{\pm} \ar[r]^{\hspace{-4mm}g\mapsto g(\cL_{2d,d}^{\pm})} \ar[d]_{g\mapsto g\mid_{\Omega_{2d-1}}} 
& \OGr^{\pm}(\Omega_{2d,0})_{\Fpbar} \ar[d]^{\cL \mapsto \cL\cap \Omega_{2d-1}} \\
G_{2d-1}/P_{2d-1,0} \ar[r]^{\hspace{-7.5mm}g\mapsto g(\cL_{2d-1,d-1})}& \OGr(\Omega_{2d-1,0})_{\Fpbar}, }
\end{equation*}
where the all maps are isomorphisms. }
\end{lem}

\begin{proof}
The commutativity of diagram is clear. The isomorphy of the horizontal maps are already mentioned. On the other hand, \cite[Lemma 3.4]{Howard2014} and the definition of $\OGr^{\pm}(\Omega_{2d,0})_{\Fpbar}$ imply the isomorphy of the vertical map in the right-hand side. Hence the isomorphy of the remaining map follows. 
\end{proof}

\begin{prop}\label{lwsp}
\emph{
\begin{enumerate}
\item For $d\leq 3$, the isomorphism in Lemma \ref{soeo} induces isomorphisms
\begin{equation*}
S^{\pm}_{\Omega_{2d,0}}\cong S_{\Omega_{2d-1,0}},\quad X_{P_{2d,0}^{\pm}}(\id)\cong X_{P_{2d-1,0}}(\id), \quad X_{P_{2d,0}^{\pm}}(t_{2d}^{\mp})\cong X_{P_{2d-1,0}}(t_{2d-1}). 
\end{equation*}
\item The isomorphism $X_{P_{6,0}^{\pm}}(t_{6}^{\mp})\cong X_{P_{5,0}}(t_{5})$ in (i) induces isomorphisms
\begin{equation*}
X_{P_{6,1}^{\pm}}(t_{6}^{\pm})\cong X_{P_{5,1}}(t_{5}), \quad X_{P_{6,1}^{\pm}}(t_{6}^{\mp}s_{6,1})\cong X_{B_{5}}(t_{5}). 
\end{equation*}
\end{enumerate}}
\end{prop}

\begin{proof}
(i): First, we prove $S^{\pm}_{\Omega_{2d,0}}\cong S_{\Omega_{2d-1,0}}$. It is clear in the case $d=2$. If $d=3$, then we obtain a closed immersion $S_{\Omega_{2d-1,0}}\rightarrow S^{\pm}_{\Omega_{2d,0}}$. On the other hand, we have $\dim S_{\Omega_{2d-1,0}}=2$ by \cite[Propositions 6.5]{Oki2019}, and $S^{\pm}_{\Omega_{2d,0}}$ is irreducible of dimension $\leq 2$ by Proposition \ref{sfdc} (ii). Hence the above map $S_{\Omega_{2d-1,0}}\rightarrow S^{\pm}_{\Omega_{2d,0}}$ is an isomorphism. 

Next, we prove the remaining isomorphisms. By \cite[Proposition 6.4 (ii)]{Oki2019}, there is a locally closed stratification
\begin{equation*}
S_{\Omega_{2d,0}}^{\pm}=X_{P_{2d,0}^{\pm}}(\id)\sqcup X_{P_{2d,0}^{\pm}}(t_{d}^{\mp}), 
\end{equation*}
Hence it suffices to prove the isomorphism $X_{P_{2d,0}^{\pm}}(\id)\cong X_{P_{2d-1,0}}(\id)$. However, this follows from the definitions of $X_{P_{2d-1,0}}(\id)$ and $\OGr^{+}(\Omega_{2d,0})_{\Fpbar}$. 

(ii): Set
\begin{align*}
S_{\Omega_{6,0}}^{1,\pm}&:=\{(\cL_{1},\cL_{3})\in \P(\Omega_{6,0})_{\Fpbar}\times \OGr^{\pm}(\Omega_{6,0})_{\Fpbar}\mid \cL_{1}=\cL_{3} \cap \Phi(\cL_{3})\},\\
S_{\Omega_{5,0}}^{1}&:=\{(\cL'_{1},\cL'_{2})\in \P(\Omega_{5,0})_{\Fpbar}\times \OGr(\Omega_{5,0})_{\Fpbar}\mid \cL'_{1}=\cL'_{2} \cap \Phi(\cL'_{2})\}. 
\end{align*}
By the definitions of $X_{P_{6,0}^{\pm}}(t_{6}^{\mp})$ and $X_{P_{5,0}}(t_{5})$, we obtain an isomorphism
\begin{equation*}
S_{\Omega_{6,0}}^{1,\pm}\cong X_{P_{6,0}^{\pm}}(t_{6}^{\mp})\cong X_{P_{5,0}}(t_{5})\cong S_{\Omega_{5,0}}^{1}. 
\end{equation*}
This isomorphism is equal to $(\cL_{1},\cL_{3})\mapsto (\cL_{1}\cap \Omega_{5},\cL_{3}\cap \Omega_{5})$, and hence we obtain $\cL_{1}\cap \Omega_{5}=\cL_{1}$. Therefore, for $(\cL_{1},\cL_{3})\in S_{\Omega_{6,0}}^{1,\pm}$, we have $\Phi(\cL_{1})=\cL_{1}$ if and only if $\Phi(\cL_{1}\cap \Omega_{5})=\cL_{1}\cap \Omega_{5}$. This equivalence implies the desired isomorphisms. 
\end{proof}

\textbf{Case 2. The case for $\Omega_{2d,0}$ non-split. }
For $i\in \Znn$, put 
\begin{equation*}
P_{2d,i}:=
\begin{cases}
P_{\{s_{2d,1},\ldots,s_{2d,d-2-i}\}}&\text{if }i\leq d-3,\\
B_{2d-1}&\text{otherwise}. 
\end{cases}
\end{equation*}
Then the following hold: 
\begin{itemize}
\item $P_{2d,i}$ is the stabilizer of the flag $\cL_{2d,d-1-i}\subset \cdots \subset \cL_{2d,d-1}$ if $0\leq i\leq d-3$, 
\item $P_{2d,i}^{\pm}$ is the stabilizer of the flag $\cL_{2d,1}\subset \cdots \subset \cL_{2d,d-1}$ if $i\geq d-2$. 
\end{itemize}

\begin{prop}\label{stns}(\cite[Proposition 3.8]{Howard2014})
\emph{There is a locally closed stratification
\begin{equation*}
S_{\Omega_{2d,0}}=\coprod_{i=0}^{d-1}(X_{P_{2d,i}}(t_{2d}^{+}s_{2d,d-2}\cdots s_{2d,d-i})\sqcup X_{P_{2d,i}}(t_{2d}^{-}s_{2d,d-2}\cdots s_{2d,d-i})). 
\end{equation*}
Moreover, for any $r\in \{0,\ldots,d-1\}$, the closure of $X_{P_{2d,r}}(t_{2d}^{\pm}s_{2d,d-2}\cdots s_{2d,d-r})$ in $G_{2d}/P_{2d,0}$ equals
\begin{equation*}
\coprod_{i=0}^{r}X_{P_{2d,i}}(t_{2d}^{+}s_{2d,d-2}\cdots s_{2d,d-i}). 
\end{equation*}}
\end{prop}

\section{Bruhat--Tits stratification in the neutral case, I}\label{btn1}

In this section, we describe the structure of $\M_{b,\mu}^{(0)}$ for a $\mu$-neutral $b$ by means of certain $6$-dimensional quadratic space over $\Qp$. 

\subsection{Preliminaries}

Here we recall the basic notions on Clifford algebras and spinor similitude groups. 

Let $L$ a quadratic space over a ring $A$. We denote by $C(L)$ the Clifford algebra of $L$. Then there is a $\Z/2$-grading
\begin{equation*}
C(L)=C^{+}(L)\oplus C^{-}(L), 
\end{equation*}
where $C^{+}(L)$ is the $0$-th part, which is an $A$-subalgebra of $C(L)$. The ring $C^{+}(L)$ is called the even Clifford algebra. 

We regard $L$ as an $A$-submodule of $C^{-}(L)$. Let $v\mapsto v'$ be the involution on $C(L)$ satisfying $(v_1\cdot \cdots\cdot v_d)'=v_d\cdots v_1$ for $v_1,\ldots ,v_d\in L$. Then we define an affine group scheme $\GSpin(L)$ over $A$ by
\begin{equation*}
\GSpin(L)(R)=\{g\in C^{+}(L\otimes_{A}R)^{\times}\mid g(L\otimes_{A}R)g^{-1}=L\otimes_{A}R,g'g\in R^{\times}\}
\end{equation*}
for any $A$-algebra $R$. Moreover, denote by $\Spin(L)$ the derived group of $\GSpin(L)$. 

If $R=k$ is a field, we define an action of $\GSpin(L)(k)$ on $L$ by the exact sequence
\begin{equation*}
1\rightarrow \G_m\rightarrow \GSpin(L)\rightarrow \SO(L)\rightarrow 1, 
\end{equation*}
where $\GSpin(L)\rightarrow\SO(L)$ is defined by $g\mapsto [v\mapsto gvg^{-1}]$. In particular, the homomorphism $\GSpin(L)(k)\rightarrow \SO(L)(k)$ is surjective by Hilbert satz 90. 

On the other hand, the above exact sequence induces another exact sequence
\begin{equation*}
1\rightarrow \mu_2 \rightarrow \Spin(L)\rightarrow \SO(L)\rightarrow 1. 
\end{equation*}
Taking the Galois cohomology over $k$, we obtain an exact sequence of groups
\begin{equation*}
1\rightarrow \mu_2(k) \rightarrow \Spin(L)(k) \rightarrow \SO(L)(k) \xrightarrow{\theta_{L}} k^{\times}/(k^{\times})^{2}\rightarrow H^{1}(k,\Spin(L)). 
\end{equation*}
We call the homomorphism $\theta_{L}$ as the \emph{spinor norm} of $L$. Note that all exact sequences as appeared in this section are functorial with respect to homomorphisms of quadratic spaces. 

\begin{prop}\label{h1e0}(\cite{Kneser1965})
\emph{If $k$ is a $p$-adic field and $\dim_{k}(L)\geq 3$, then $H^{1}(k,\Spin(L))=1$. In particular, the spinor norm $\theta_{L}$ of $L$ is surjective. }
\end{prop}

\subsection{Hodge star operator}\label{hgst}

Here we will define a quadratic space $\L_{b,0}$ over $\Qp$ which will be used to describe $\M_{b,\mu}^{(0)}$. 

Until Section \ref{bts1}, we fix $\varepsilon \in W^{\times}$ satisfying $\sigma(\varepsilon)=-\varepsilon$. 

\begin{lem}\label{base}
\emph{There is an $\Fbr$-basis $e_1,e_2,e_3,e_4$ such that
\begin{equation*}
\F(e_1)=\varpi e_3,\quad \F(e_2)=\varpi e_4,\quad \F(e_3)=p\varpi^{-1}e_1,\quad \F(e_4)=p\varpi^{-1}e_2. 
\end{equation*}
and 
\begin{equation*}
\langle e_i,e_j\rangle=
\begin{cases}
\varpi^{-1}&\text{if }(i,j)=(1,3),\\
\varpi&\text{if }(i,j)=(2,4),\\
0&\text{otherwise. }
\end{cases}
\end{equation*}}
\end{lem}

\begin{proof}
Put
\begin{equation*}
b':=\begin{pmatrix}
&&&-p\\
&&p&\\
&1&&\\
-1&&&
\end{pmatrix}. 
\end{equation*}
Then we have $\det_{\Fbr}(b')/\sml_G(b')^2=1$, which implies that $[b']$ is $\mu$-neutral by Proposition \ref{nteq}. Hence we may assume $b=b'$. Then
\begin{equation*}
e_1:=\e_1,\quad e_2:=\varpi \e_2,\quad e_3:=-\varpi^{-1}\e_4,\quad e_4:=\e_3
\end{equation*}
gives a desired $\Fbr$-basis. 
\end{proof}

From now on, replacing $(\X_0,\iota_0,\lambda_0)$ to $(X,\iota,\lambda)$ satisfying $(X,\iota,\lambda,\rho)\in \M_{b,\mu}^{(0)}$ if necessary, we assume
\begin{equation*}
M_0=O_{\Fbr}(\varpi e_1)\oplus O_{\Fbr}e_2\oplus O_{\Fbr}(\varpi e_3)\oplus O_{\Fbr}e_4. 
\end{equation*}

\begin{dfn}\label{omdf}
Put 
\begin{equation*}
\omega:=e_1\wedge e_2\wedge e_3\wedge e_4\in \bigwedge^4_{\Fbr}N_b, 
\end{equation*}
where $e_1,\ldots,e_4$ are as in Lemma \ref{base}. Then we define a \emph{Hodge star operator} $v\mapsto v^{\star}$ on $\bigwedge^2_{\Fbr}N_b$ (with respect to $\omega$) by $v'\wedge v^{\star}=\langle v',v\rangle_{2} \omega$ for any $v'\in \bigwedge^2_{\Fbr}N_b$. Here $\langle \,,\,\rangle_{2}$ is the bilinear form on $\bigwedge^2_{\Fbr}N_b$ satisfying
\begin{equation*}
\langle x_1\wedge x_2,y_1\wedge y_2 \rangle_{2}=\langle x_1,y_1 \rangle \langle x_2,y_2 \rangle-\langle x_1,y_2 \rangle \langle x_2,y_1 \rangle
\end{equation*}
for any $x_1.x_2,y_1,y_2\in N_b$. 

\end{dfn}
By definition, we have $(av)^{\star}=\overline{a}v^{\star}$ and $(v^{\star})^{\star}=v$ for any $a\in \Fbr$ and $v\in \bigwedge_{\Fbr}^{2}N_b$. Now put
\begin{equation*}
\L_{b}:=\left\{v\in \bigwedge^2_{\Fbr}N_b\,\middle|\,v^{\star}=v\right\}. 
\end{equation*}
Then it is a $K_0$-vector space of dimension $6$ which is spanned by the following elements: 
\begin{align*}
x_1&:=e_1\wedge e_2,& x_2&:=e_3\wedge e_4,\\
x_3&:=e_1\wedge e_3-\varpi^{-2}e_2\wedge e_4,& x_4&:=\varpi e_1\wedge e_3+\varpi^{-1}e_2\wedge e_4,\\
x_5&:=e_1\wedge e_4,& x_6&:=e_3\wedge e_2. 
\end{align*}
Note that $\L_b$ depends on the choice of $e_1,\ldots,e_4$, cf. \cite[Section 2.2]{Howard2014}. 

We regard $\bigwedge_{\Fbr}^{2}N_b$ as a $K_0$-subspace of $\End_{K_0}(N_b)$ by the injection
\begin{equation*}
\bigwedge_{\Fbr}^{2}N_b\rightarrow \End_{K_0}(N_b);x\wedge y\mapsto [z\mapsto \langle x,z\rangle y-\langle y,z\rangle x]. 
\end{equation*}
Note that the image of the above injection consists of $K_0$-endomorphisms $f$ of $N_b$ such that $f(ax)=\overline{a}f(x)$ for $a\in \Fbr$ and $x\in N_b$. 

We will frequently use the following relation between $v\mapsto v'$ on $C(\L_{b})$ and bilinear forms on $N_b$: 
\begin{lem}\label{hmpr}
\emph{Let $v\in \L_b$. Then we have $\langle vx,y\rangle=-\overline{\langle x,vy\rangle}$ and $(vx,y)=(x,vy)$ for any $x,y\in N_b$. }
\end{lem}

\begin{proof}
The proof is the same as \cite[Lemma 3.6]{Oki2019}. 
\end{proof}

\begin{prop}\label{lbpr}
\emph{
\begin{enumerate}
\item For $v\in \L_b$, as an element in $\End_{K_0}(N_b)$, we have 
\begin{equation*}
v^2=\langle v,v\rangle. 
\end{equation*}
In particular, the map $Q\colon \L_b\rightarrow \End_{K_0}(N_{b});v\mapsto v^2$ is $K_0$-valued quadratic form.  We regard $\L_b$ as a quadratic space over $K_0$ by $Q$, and put
\begin{equation*}
[\,,\,]\colon \L_b\rightarrow \End_{K_0}(N_{b});(v,w)\mapsto \frac{1}{2}(v\circ w+w\circ v)=\frac{1}{2}(Q(v+w)-Q(v)-Q(w)). 
\end{equation*}
\item The $W$-submodule $L_0:=\{v\in \L_b\mid v(M_0)\subset M_0\}$ of $\L_b$ is a $W$-lattice. Moreover, we have
\begin{equation*}
L_0^{\vee}=\{v\in \L_b\mid v(M_0)\subset \varpi^{-1}M_0\},\quad \length_{W}(L_0^{\vee}/L_0)=1. 
\end{equation*}
\item The inclusion $L_0\subset \End_{W}(M_0)$ induces isomorphisms 
\begin{equation*}
C(L_0)\cong \End_{W}(M_0),\quad C^{+}(L_0)\cong \End_{O_{\Fbr}}(M_0). 
\end{equation*}
\end{enumerate}}
\end{prop}

\begin{proof}
(i): This follows from the same argument as the proof of \cite[Proposition 2.4 (2)]{Howard2014} by considering $e_i\wedge e_j$ for all $i,j$. 

(ii): By computation, the Gram matrix of $[\,,\,]$ with respect to the basis $x_1,\ldots, x_6$ as above is as follows: 
\begin{equation*}
\begin{pmatrix}
0&-\frac{1}{2}&&&&\\
-\frac{1}{2}&0&&&&\\
&&-\varpi^{-2}&0&&\\
&&0&1&&\\
&&&&-\frac{1}{2}&0\\
&&&&0&-\frac{1}{2}\\
\end{pmatrix}. 
\end{equation*}
Moreover, we have 
\begin{equation*}
L_0=Wx_1\oplus Wx_2\oplus W(\varpi^2x_3)\oplus Wx_4\oplus Wx_5\oplus Wx_6
\end{equation*}
by definition. Hence $L_0$ is a $W$-lattice in $\L_b$. Moreover, we have $L_0\subset L_0^{\vee}$ and $\length_{W}(L_0^{\vee}/L_0)=1$. On the other hand, if we put
\begin{equation*}
L_0':=\{v\in \L_b\mid v(M_0)\subset \varpi^{-1}M_0\}, 
\end{equation*}
then we have 
\begin{equation*}
L'_0=Wx_1\oplus Wx_2\oplus Wx_3\oplus Wx_4\oplus Wx_5\oplus Wx_6
\end{equation*}
and $L_0'\subset L_0^{\vee}$ by definition. Since $\length_{W}(L_0^{\vee}/L_0)=1$ and $L_0\subsetneq L_0'$, we obtain the equality $L_0'=L_0^{\vee}$ as desired. 

(iii): This follows from the same argument as the proof of \cite[Proposition 2.4 (4)]{Howard2014}. 
\end{proof}

We define a $\sigma$-linear bijection
\begin{equation*}
\Phi_b \colon \End_{K_0}(N_b)\rightarrow \End_{K_0}(N_b); f\mapsto \F_b\circ f\circ \F_{b}^{-1}. 
\end{equation*}
Then the isocrystal $(\End_{K_0}(N_b),\Phi_b)$ over $\Fpbar$ is isoclinic of slope $0$. Moreover, there is an isomorphism $\End(N_b)^{\Phi_b}\cong \End^0(\X_0)$. On the other hand, we have 
\begin{equation*}
\Phi_{b}(x\wedge y)=p^{-1}\F_b(x)\wedge \F_b(y)
\end{equation*}
under the inclusion $\bigwedge_{\Fbr}^{2}N_b\subset \End_{K_0}(N_b)$. We can prove that $\Phi_{b}$ commutes with $\star$ by the same argument as \cite[Lemma 3.9]{Oki2019}. Hence we obtain an isocrystal $(\L_{b},\Phi_{b})$ over $\Fpbar$ which is isoclinic of slope $0$. We denote by $\L_{b,0}$ the $\Phi_b$-fixed part of $\L_b$. Then $(\L_{b,0},Q)$ becomes a quadratic space of dimension $6$ over $\Qp$. 

We compute a basis of $\L_{b,0}$ explicitly. The $\Qp$-vector space $\L_{b,0}$ is spanned by the following elements: 
\begin{align*}
y_1&:=x_1+p^{-1}\varpi^{2}x_2,& y_2&:=\varepsilon(x_1-p^{-1}\varpi^{2}x_2),\\
y_3&:=\varepsilon x_3,& y_4&:=\varepsilon x_4,\\
y_5&:=x_5+x_6,& y_6&:=\varepsilon(x_5-x_6). 
\end{align*}
The Gram matrix of $[\,,\,]$ with respect to the basis $y_1,\ldots,y_6$ is as follows: 
\begin{equation*}
\begin{pmatrix}
-p^{-1}\varpi^2&&&&&\\
&p^{-1}\varpi^2\varepsilon^{2}&&&&\\
&&-\varpi^{-2}\varepsilon^{2}&&&\\
&&&\varepsilon^2&&\\
&&&&-1&\\
&&&&&\varepsilon^2\\
\end{pmatrix}. 
\end{equation*}
Hence the discriminant and the Hasse invariant of $(\L_{b,0},[\,,\,])$ are $-\varpi^{2}$ and $-1$ respectively. 

We will use the following in Section \ref{vln1}: 
\begin{lem}\label{picl}
\emph{We have $\iota_0(\varpi)=-p\varepsilon^4y_1\cdots y_6$ in $C(\L_{b})\cong \End_{K_0}(N_{b})$. }
\end{lem}

\begin{proof}
This follows from direct computation. 
\end{proof}

\subsection{Exceptional isomorphism}

We define a left action of $G(K_0)$ on $\End_{K_0}(N_b)$ by
\begin{equation*}
G(K_0)\times \End_{K_0}(N_b)\rightarrow \End_{K_0}(N_b);(g,f)\mapsto g\cdot f:=g\circ f\circ g^{-1}. 
\end{equation*}
Then we have 
\begin{equation}\label{gact}
g\cdot (x\wedge y)=\sml_G(g)^{-1}g(x)\wedge g(y)
\end{equation}
for any $g\in G$ by the inclusion $\bigwedge_{\Fbr}^{2}N_b$. On the other hand, let
\begin{equation*}
G^0:=\{g\in G\mid \sml_G(g)^2=\det\!_{F}(g)\}. 
\end{equation*}
Then, for $g\in G(K_0)$, the action $v\mapsto g\cdot v$ on $\bigwedge_{\Fbr}^{2}N_b$ commutes with $\star$ if and only if $g\in G^0(K_0)$. See \cite[Section 2.3]{Howard2014}. 

\begin{prop}\label{eik0}
\emph{
\begin{enumerate}
\item The isomorphism $C^{+}(\L_b)\cong \End_{\Fbr}(N_b)$ in Proposition \ref{lbpr} (iii) induces an isomorphism
\begin{equation*}
\GSpin(\L_b)\cong G^0\otimes_{\Qp}K_0, 
\end{equation*}
which is compatible with the actions on $\L_b$. 
\item For any algebraic closed field $k$ of characteristic $p$, the isomorphism in (i) induces an isomorphism 
\begin{equation*}
\stab_{\GSpin(\L_{b})(\Frac W(k))}(L_{0,W(k)})\cong \stab_{G^0(\Frac W(k))}(M_{0,W(k)}). 
\end{equation*}
\end{enumerate}}
\end{prop}

\begin{proof}
(i): This follow from the same argument as the proof of \cite[Proposition 2.7]{Howard2014} by using Lemma \ref{hmpr}. 

(ii): This follows from (i) and the definition of $L_0$. 
\end{proof}

We denote by $J_{b}^{\der}$ the derived group of $J_b$, and put
\begin{equation*}
J_{b}^{0}:=J_{b}\cap \Res_{K_0/\Qp}G^{0}. 
\end{equation*}

\begin{cor}\label{exis}
\emph{The isomorphism $\GSpin(\L_b)\cong G^0$ in Proposition \ref{eik0} induces an isomorphism
\begin{equation*}
\GSpin(\L_{b,0})\cong J_{b}^{0}
\end{equation*}
and a central isogeny
\begin{equation*}
J_{b}^{\der}\rightarrow \SO(\L_{b,0}),
\end{equation*}
which are compatible with the actions on $\L_{b,0}$. }
\end{cor}

\begin{proof}
This follows from of Proposition \ref{eik0} (i). 
\end{proof}

\subsection{Description of a subset of field valued points}

Here we fix an algebraically closed field $k$ of characteristic $p$. We set $M_{0,W(k)}:=M_0\otimes_{W}W(k)$. 
\begin{dfn}\label{dlm0}
For $j\in \Z/2$, put
\begin{equation*}
\DiL^{\varpi}_{M_{0},j}(N_{b,W(k)}):=\{M\in \DiL^{\varpi}(N_{b,W(k)})\mid \length_{O_{F,W(k)}}(M+M_{0,W(k)})/M_{0,W(k)}\equiv j\bmod 2\}. 
\end{equation*}
Moreover, let $\M_{b,\mu}^{(0,j)}(k)$ be the subset of $\M_{b,\mu}^{(0)}(k)$ corresponding to $\DiL^{\varpi}_{M_{0},j}(N_{b,W(k)})$ under the bijection $\M_{b,\mu}^{(0)}\cong \DiL^{\varpi}(N_{b,W(k)})$ in Proposition \ref{hmlt}. 
\end{dfn}
By definition, we have 
\begin{equation*}
\M_{b,\mu}^{(0)}(k)=\M_{b,\mu}^{(0,0)}(k)\sqcup \M_{b,\mu}^{(0,1)}(k). 
\end{equation*}
Moreover, the action of $s\in J_{b}(\Qp)$ satisfying $\kappa_{G}(s)=(0,1)$ induces a bijection
\begin{equation*}
\M_{b,\mu}^{(0,0)}(k)\cong \M_{b,\mu}^{(0,1)}(k). 
\end{equation*}
Note that $\M_{b,\mu}^{(0,j)}(k)$ is stable under field extension, that is, for an algebraically closed field $k'$ containing $k$ and $j\in \Z/2$, the image of the canonical map $\M_{b,\mu}^{(0,j)}(k)\rightarrow \M_{b,\mu}^{(0)}(k')$ lies in $\M_{b,\mu}^{(0,j)}(k')$. 

Here we will describe the set $\M_{b,\mu}^{(0,0)}(k)$ by means of certain lattices in $\L_{b,W(k)}:=\L_{b}\otimes_{W}W(k)$. 

\begin{dfn}
A special lattice in $\L_{b,W(k)}$ is a $W(k)$-lattice $L$ in $\L_{b,W(k)}$ satisfying 
\begin{equation*}
\length_{W(k)}L+\Phi_{b}(L)/L\leq 1. 
\end{equation*}
\end{dfn}
We denote by $\SpL(\L_{b,W(k)})$ the set of all special lattices in $\L_{b,W(k)}$. We prove the following: 

\begin{prop}\label{dlsl}
\emph{For any $M\in \DiL^{\varpi}_{M_{0},0}(N_{b,W(k)})$, 
\begin{equation*}
L(M):=\{v\in \L_{b,W(k)}\mid v(\bV_{b}(M))\subset \varpi^{-1}\bV_{b}(M)\}
\end{equation*}
is an element of $\SpL(\L_{b,W(k)})$. Moreover, the map 
\begin{equation*}
\DiL^{\varpi}_{M_{0},0}(N_{b,W(k)})\rightarrow \SpL(\L_{b,W(k)});M\mapsto L(M)
\end{equation*}
is a $J_{b}^{\der}(\Qp)$-equivariant bijection. }
\end{prop}

We give a proof of Proposition \ref{dlsl}. Let $\Lat^{\varpi}_{M_0,0}(N_{b,W(k)})$ be the set of $\varpi$-modular $O_{F,W(k)}$-lattices in $N_{b,W(k)}$ satisfying $\length_{O_{F,W(k)}}(M+M_{0,W(k)})/M_{0,W(k)} \in 2\Z$. 

\begin{lem}\label{pmgd}
\emph{For $g\in G^{\der}(\Frac W(k))$, we have $g(M_{0,W(k)})\in \Lat^{\varpi}_{M_0,0}(N_{b,W(k)})$. Moreover, the map
\begin{equation*}
G^{\der}(\Frac W(k))/K_{p,W(k)}\rightarrow \Lat^{\varpi}_{M_0,0}(N_{b,W(k)})
\end{equation*}
is bijective, where $K_{p,W(k)}:=\stab_{G^{\der}(\Frac W(k))}(M_{0,W(k)})$. }
\end{lem}

\begin{proof}
The first assertion is a consequence of \cite[Lemma 3.2]{Rapoport2017}. The injectivity of the map is clear. To prove the surjectivity, take $M\in \Lat^{\varpi}_{M_0,0}(N_{b,W(k)})$. By \cite[Proposition 8.1]{Jacobowitz1962}, there is $g\in G$ such that $\det\!{}_{F}(g)=1$ and $M=g(M_{0,W(k)})$. Then we have $\det\!{}_{F}(g)\in 1+\varpi O_{F,W(k)}$ by \cite[Lemma 3.2]{Rapoport2017}. On the other hand, Hilbert satz 90 implies that there is $t\in G(\Frac W(k))$ such that $\sml_{G}(t)=1$, $\det\!{}_{F}(t)=\det\!{}_{F}(g)$ and $t(M_{0,W(k)})=M_{0,W(k)}$ (note that we use the condition $\det\!{}_{F}(g)\in 1+\varpi O_{F,W(k)}$  to prove $t(M_{0,W(k)})=M_{0,W(k)}$). Then we have $gt^{-1}\in G^{\der}(\Frac W(k))$ and $gt^{-1}(M_{0,W(k)})=M$. 
\end{proof}

Let $\Lat^{1}(\L_{b,W(k)})$ be the set of $W(k)$-lattices $L$ in $\L_{b,W(k)}$ satisfying $L^{\vee}\subset L$ and
\begin{equation*}
\length_{W(k)}(L/L^{\vee})=1. 
\end{equation*}

\begin{lem}\label{ltlt}
\emph{
\begin{enumerate}
\item For any $M\in \Lat^{\varpi}_{M_0,0}(N_{b,W(k)})$, $L^{\sharp}(M):=\{v\in \L_{b,W(k)}\mid v(M)\subset \varpi^{-1}M\}$ is an element of $\Lat^{1}(\L_{b,W(k)})$. Moreover, we have $L^{\sharp}(M)^{\vee}=\{v\in \L_{b,W(k)}\mid v(M)\subset M\}$. 
\item The map
\begin{equation*}
\Lat^{\varpi}_{M_0,0}(N_{b,W(k)})\rightarrow \Lat^{1}(\L_{b,W(k)}); M\mapsto L^{\sharp}(M)
\end{equation*}
is a $G^{\der}(\Frac W(k))$-equivariant bijection. 
\end{enumerate}}
\end{lem}

\begin{proof}
(i): This follows from Lemma \ref{pmgd} and Proposition \ref{lbpr} (ii). 

(ii): The $G^{\der}(\Frac W(k))$-equivalence is clear by definition. The injectivity follows from Lemma \ref{pmgd} and Proposition \ref{eik0} (ii). The surjectivity follows from \cite[Lemma 29.9]{Shimura2010} (by considering dual lattices) and Proposition \ref{eik0} (i). 
\end{proof}

We regard $G\otimes_{\Qp}K_0$ as a subgroup of $\Res_{\Fbr/K_0}\GL(N_{b})$ by the basis $\e_1,\ldots ,\e_4$, and put 
\begin{equation*}
w_0:=\varpi \id_{N_{b}},\quad w_1:=\diag(1,1,-\varpi^2,-\varpi^2). 
\end{equation*}
Note that $w_0$ and $w_1$ are elements of $\varpi \cdot G^{\der}(\Frac W(k))$. We set 
\begin{equation*}
X(b_0)_{K_{p,W(k)}}:=\left\{g\in G^{\der}(\Frac W(k))/K_{p,W(k)}\,\middle| \,g^{-1}b_0\sigma(g)\in \bigcup_{i\in \{0,1\}}K_{p,W(k)}w_iK_{p,W(k)}\right\},
\end{equation*}
where $b_0\in G(K_0)$ is an element defined in Section \ref{ntrl}. Note that we have $\sml_{G}(b_0)^2=\det\!{}_{F}(b_0)=\varpi^{2}$. Then the map
\begin{equation*}
X(b_0)_{K_{p,W(k)}}\rightarrow \DiL^{\varpi}_{M_0,0}(N_{b,W(k)});g\mapsto g(M_{0,W(k)})
\end{equation*}
is an isomorphism. 

On the other hand, we regard $\SO(\L_{b,0})$ as a subgroup of $\GL(\L_{b,0})$ by the basis $x_1,\ldots,x_6$, and put
\begin{equation*}
w'_0:=\id_{\L_{b}},\quad w'_1:=\diag(\varpi^{-2},\varpi^{2},-1,-1,-1,-1). 
\end{equation*}
We denote by $K'_{p,W(k)}$ the stabilizer of $L_{0,W(k)}$ in $\SO(\L_{b})(\Frac W(k))$. Note that the surjection $\pi \colon G^0\otimes_{\Qp}K_0 \rightarrow \SO(\L_{b})$ maps $\varpi^{-1}w_i$ to $w'_i$ for each $i\in \{0,1\}$. Now put $b'_0:=\pi(\varpi^{-1}b_0)$, and set
\begin{equation*}
X(b'_0)_{K'_{p,W(k)}}:=\left\{h\in \SO(\L_{b})(\Frac W(k))/K'_{p,W(k)}\,\middle| \,h^{-1}b'_0\sigma(h)\in \bigcup_{i\in \{0,1\}}K_{p,W(k)}w'_iK_{p,W(k)}\right\}. 
\end{equation*}
Then the map
\begin{equation*}
X(b'_0)_{K'_{p,W(k)}}\rightarrow \SpL(\L_{b,W(k)});h\mapsto h(L_{0,W(k)}^{\vee})
\end{equation*}
is bijective. 

\begin{lem}\label{adlv}
\emph{The surjection $\pi$ induces an isomorphism
\begin{equation*}
X(b_0)_{K_{p,W(k)}}\xrightarrow{\cong}X(b'_0)_{K'_{p,W(k)}}. 
\end{equation*}
such that the following diagram commutes: 
\begin{equation*}
\xymatrix@C=46pt{
X(b_0)_{K_{p,W(k)}} \ar[r]^{\cong} \ar[d]_{g\mapsto g(M_{0,W(k)})}^{\cong} & X(b'_0)_{K'_{p,W(k)}} \ar[d]^{h\mapsto h(L_{0,W(k)}^{\vee})} \\
\DiL^{\varpi}_{M_0,0}(N_{b,W(k)}) \ar[r]^{\,\,M\mapsto L(M)}& \Lat^{1}(\L_{b,W(k)}). }
\end{equation*}}
\end{lem}

Note that Lemma \ref{adlv} immediately implies Proposition \ref{dlsl}. 

\begin{proof}[Proof of Lemma \ref{adlv}]
Put
\begin{equation*}
\SO(\L_{b})(\Frac W(k))_{1}:=\pi(G^{\der}(\Frac W(k))),\quad K'_{p,W(k),1}:=\pi(K_{p,W(k)}). 
\end{equation*}
Define
\begin{equation*}
X'(b'_0)_{K'_{p,W(k),1}}:=\left\{h\in \SO(\L_{b})(\Frac W(k))_{1}/K'_{p,W(k),1}\,\middle| \,h^{-1}b'_0\sigma(h)\in \bigcup_{i\in \{0,1\}}K'_{p,W(k),1}w'_i K'_{p,W(k),1}\right\}. 
\end{equation*}
Since $\Ker(\pi \!\mid_{G^{\der}(\Frac W(k))})=\{\pm \id_{N_{b,W(k)}}\}$ and $\pi(\varpi^{-1}w_i)=w'_i$ for $i\in \{0,1\}$, the homomorphism $\pi$ induces an isomorphism
\begin{equation*}
X(b_0)_{K_{p,W(k)}}\cong X'(b'_0)_{K'_{p,W(k),1}}. 
\end{equation*}
On the other hand, note that the sequence
\begin{equation*}
G^{\der}(\Frac W(k)) \xrightarrow{\pi \mid_{G^{\der}(K_0)}} \SO(\L_{b})(\Frac W(k)) \xrightarrow{\theta_{\L_{b,W(k)}}} \Frac W(k)^{\times}/(\Frac W(k))^{2}\cong \Z/2
\end{equation*}
is exact. Moreover, the direct computation implies the following: 
\begin{equation*}
\pi(\varpi \id_{N_{b,W(k)}})=-\id_{\L_{b,W(k)}},\quad \theta_{\L_{b,W(k)}}(-\id_{\L_{b,W(k)}})\neq 0. 
\end{equation*}
Hence we obtain equalities
\begin{equation*}
\SO(\L_{b})(\Frac W(k))=\pi(G^{\der}(\Frac W(k))\times \{\pm \id_{\L_{b,W(k)}}\},\quad K'_{p,W(k)}=\pi(K_{p,W(k)})\times \{\pm \id_{\L_{b,W(k)}}\}.
\end{equation*}
Hence the canonical map $X(b_0)_{K_{p,W(k)},1}\rightarrow X(b_0)_{K_{p,W(k)}}$ is injective. It is also surjective since $\theta_{\L_{b}}(w'_i)=0$ for $i\in \{0,1\}$ and $\theta_{\L_{b,W(k)}}(h^{-1}b'_0\sigma(h))=0$ for any $h\in G(\Frac W(k))$. 

Finally, the commutativity of the diagram is clear. 
\end{proof}

We give another application of the bijection $h\mapsto h(L_{0,W(k)}^{\vee})$, which will be used in Section \ref{vln1}. 

\begin{prop}\label{dlph}
\emph{For $L\in \SpL(\L_{b,W(k)})$, we have $\Phi_{b}(L)\subset p^{-1}L^{\vee}$. }
\end{prop}

\begin{proof}
Write $L=h(L_{0,W(k)}^{\vee})$, where $h\in \SO(\L_b)(\Frac W(k))$ satisfies $h^{-1}b'_0\sigma(h)\in K_{p,W(k)}w'_iK_{p,W(k)}$ for some $i\in \{0,1\}$. Note that $L^{\vee}=h(L_{0,W(k)})$. Write $h^{-1}b'_0\sigma(h)=h_1w'_1h_2$, where $h_1,h_2\in K_{p,W(k)}$. By replacing $h$ to $hh_1$ if necessary, we may assume $h_1=\id_{\L_{b}}$. 

If $i=0$, then we have $\Phi_{b}(L)=L$, and hence the assertion is trivial. Otherwise, we have 
\begin{align*}
w'_1(L_{0,W(k)}^{\vee})&=W(p^{-1}x_1)\oplus Wx_2\oplus Wx_3\oplus Wx_4\oplus Wx_5\oplus Wx_6 \\
&\subset W(p^{-1}x_1)\oplus W(p^{-1}x_2)\oplus Wx_3\oplus W(p^{-1}x_4)\oplus W(p^{-1}x_5)\oplus W(p^{-1}x_6)=p^{-1}L_0. 
\end{align*}
Here $x_1,\ldots,x_6$ is the basis of $\L_{b}$ defined in Section \ref{hgst}. Using this inclusion, we have
\begin{equation*}
\Phi_{b}(L)=h(h^{-1}b'_0\sigma(h))(L_{0,W(k)}^{\vee})=hw'_1(L_{0,W(k)}^{\vee})\subset h(p^{-1}L_{0,W(k)})=p^{-1}L^{\vee}. 
\end{equation*}
Therefore the assertion follows. 
\end{proof}

\subsection{Rapoport--Zink space for the quaternionic unitary group of degree $2$}

To understand the structure of $\M_{b,\mu}^{(0,0)}$, we will use a closed immersion from the Rapoport--Zink space  for the quaternionic unitary similitude group of degree $2$. We recall the definition and the Bruhat--Tits stratification of it constructed by \cite{Oki2019}. 

From now on, put $D:=F[y_6]$ and $O_D:=O_F[y_6]$, which are subrings of $\End^0(\X_0)$. Note that $D$ is the quaternion division algebra over $\Qp$, and $O_D$ is the maximal order of $D$. We define an involution $*$ on $D$ by
\begin{equation*}
*\colon D\rightarrow D;d\mapsto y_6(\trd_{D/\Qp}(d)-d)y_6^{-1}. 
\end{equation*}
Let $G_D$ be the quaternionic unitary similitude group of degree $2$. We regard $G_D$ as a subgroup of $G$ by the natural way in the sequel. For $S\in \nilp$, a $p$-divisible with $G_D$-structure is a triple $(X,\iota,\lambda)$, where
\begin{itemize}
\item $X$ is a $p$-divisible group over $S$ of dimension $4$ and height $8$,
\item $\iota \colon O_D\rightarrow \End(X)$ is a ring homomorphism,
\item $\lambda \colon X\rightarrow X^{\vee}$ is a quasi-isogeny so that $\lambda^{\vee}=-\lambda$, 
\end{itemize}
satisfying the following conditions for any $d\in O_D$;
\begin{itemize}
\item (Kottwitz condition) $\det(T-\iota(d)\mid \Lie(X))=(T^2-\trd_{D/\Qp}(d)+\nrd_{D/\Qp}(d))^2$,
\item $\lambda \circ \iota(d)=\iota(d^*)^{\vee}\circ \lambda$. 
\end{itemize}
We define a quasi-isogeny of $p$-divisible groups with $G_D$-structures by the same manner as Definition \ref{pdiv} (ii). 

Now we define $\iota_0'\colon O_D\rightarrow \End(\X_0)$ by the composite of the actions of $O_F$ and $y_6$ on $\X_0$, and put $\lambda'_0:=(y_6\circ \iota_0(\varpi))^{\vee}\circ \lambda_0$. Then $(\X_0,\iota_0',\lambda'_0)$ is a $p$-divisible group with $G_D$-structure over $\spec \Fpbar$. Furthermore, $\lambda'_0$ is an isomorphism since $\Ker(\lambda_0)=\Ker(\iota_0(\varpi))$ and $y_6$ is an isomorphism. We define $\M_D$ as the functor which parametrizes $(X,\iota,\lambda,\rho)$ for $S\in \nilp$, where $(X,\iota,\lambda)$ is a $p$-divisible group over $S$ of dimension $4$ and height $8$ such that $\lambda$ is an isomorphism, and
\begin{equation*}
\rho \colon X\times_{S}\overline{S}\rightarrow \X_0\otimes_{\Fpbar}\overline{S}
\end{equation*}
is a quasi-isogeny of $p$-divisible groups with $G_D$-structures. We define an equivalence of two $p$-divisible groups with $G_D$-structures by the same manner as that of two $p$-divisible groups with $(G,\mu)$-structures. The representability of $\M_D$ is a consequence of \cite[Theorem 3.25]{Rapoport1996a}. Moreover, there is a decomposition into open and closed formal subschemes
\begin{equation*}
\M_{D}=\coprod_{i\in \Z}\M_{D}^{(i)}, 
\end{equation*}
where $\M_{D}^{(i)}$ is the locus $(X,\iota,\lambda,\rho)$ of $\M_{D}$ for $\im(c(\rho))\subset p^i\Zpt$. 

Let $J_D$ be the algebraic group over $\Qp$ defined by
\begin{equation*}
J_D(R)=\{(g,c)\in (\End_{O_D}(\X_{0})\otimes_{\Zp}R)^{\times}\times R^{\times}\mid g^{\vee}\circ \lambda_0'\circ g=c\lambda_0'\}
\end{equation*}
for any $\Qp$-algebra $R$. The group $J_{D}(\Qp)$ acts on $\M_D$ by the same manner as the action of $J_b(\Qp)$ on $\M_{b,\mu}$. It is known that there is an isomorphism $J_{D}\cong \GSp_{4}$. See \cite[1.42]{Rapoport1996a}. 

Next, we recall the Bruhat--Tits stratification of $\M_D$ constructed in \cite[Section 5.3]{Oki2019}. Put
\begin{equation*}
\L_{D}:=\iota'_0(y_6\varpi)^{-1}\{f\in \End^0_{O_D}(\X_0)\mid \trd_{D/\Qp}(f)=0, f^{\vee}\circ \lambda'_0=\lambda'_0\circ f\}. 
\end{equation*}
We call a lattice $\Lambda$ in $\L_{D}$ is a \emph{vertex lattice} if we have $p\Lambda \subset \Lambda^{\vee}\subset \Lambda$. If $\Lambda$ is a vertex lattice in $\L_{D}$, we call $t(\Lambda):=\dim_{\Fp}(\Lambda/\Lambda^{\vee})$ as the type of $\Lambda$. Note that we have $t(\Lambda)\in \{1,3,5\}$ by \cite[Lemma 5.9]{Oki2019}. We denote by $\VL(\L_{D},t)$ the set of all vertex lattices in $\L_{D}$ of type $t$ for $t\in \{1,3,5\}$, and put $\VL(\L_{D}):=\sqcup_{t\in \{1,3,5\}}\VL(\L_{D},t)$. 

For $\Lambda \in \VL(\L_{D})$, let $\M_{D,\Lambda}^{(0)}$ be the locus $(X,\iota,\lambda,\rho)$ of $\M_{D}^{(0)}$ where $\rho^{-1}\circ \Lambda \circ \rho \subset \End(X)$. It is a closed formal subscheme of $\M_{D}^{(0)}$. 

\begin{prop}\label{gu2d} (\cite[Corollary 5.25, Theorems 5.26, 5.27]{Oki2019})
\emph{
\begin{enumerate}
\item For $\Lambda,\Lambda'\in \VL(\L_{D})$, we have 
\begin{equation*}
\M_{D,\Lambda}^{(0)}\cap \M_{D,\Lambda'}^{(0)}=
\begin{cases}
\M_{D,\Lambda \cap \Lambda'}^{(0)} & \text{if }\Lambda \cap \Lambda'\in \VL(\L_{D}), \\
\emptyset &\text{otherwise. }
\end{cases}
\end{equation*}
\item For $\Lambda \in \VL(\L_{D})$, the formal scheme $\M_{D,\Lambda}^{(0)}$ is a reduced scheme over $\spec \Fpbar$. Moreover, the following hold: 
\begin{itemize}
\item if $t(\Lambda)=1$, then $\M_{D,\Lambda}^{(0)}$ is a single point, 
\item if $t(\Lambda)=3$, then $\M_{D,\Lambda}^{(0)}$ is isomorphic to $\P^1_{\Fpbar}$, 
\item if $t(\Lambda)=5$, then $\M_{D,\Lambda}^{(0)}$ is isomorphic to the Fermat surface $F_p$ defined by
\begin{equation*}
x_0^{p+1}+x_1^{p+1}+x_2^{p+1}+x_3^{p+1}=0
\end{equation*}
in $\P_{\Fpbar}^{3}=\Proj \Fpbar[x_0,x_1,x_2,x_3]$. 
\end{itemize}
In particular, $\M_{D,\Lambda}^{(0)}$ is projective smooth and irreducible of dimension $(t(\Lambda)-1)/2$. 
\item We have equalities
\begin{equation*}
\M_{D}^{(0)}=\bigcup_{\Lambda \in \VL(\L_{D})}\M_{D,\Lambda}^{(0)}=\bigcup_{\Lambda \in \VL(\L_{D},5)}\M_{D,\Lambda}^{(0)}. 
\end{equation*}
In particular, any irreducible component of $\M_{D}^{(0)}$ is of the form $\M_{D,\Lambda}^{(0)}$ where $\Lambda \in \VL(\L_{D},5)$, which has $2$-dimensional. 
\end{enumerate}}
\end{prop}

We give a relation between $\M_{D}$ and $\M_{b,\mu}$. Let $N_{b,W(k)}^{0}$ be the eigenspace of $y_6$ on $N_{b,W(k)}$ with respect to $\varepsilon$. Then the restriction to $N_{b,W(k)}^{0}$ of the alternating form $(\,,\,)$ on $N_{b,W(k)}$ is non-degenerate. Moreover, $N_{b,W(k)}^{0}$ is stable under the $\sigma$-linear map $\F_{b,0}$. 

\begin{prop}\label{clim}
\emph{The morphism 
\begin{equation*}
i_{y_6}\colon \M_{D}\rightarrow \M_{b,\mu};(X,\iota,\lambda,\rho)\mapsto (X,\iota\vert_{O_F},\lambda \circ (\rho^{-1}y_6\circ \rho)\circ \iota(\varpi),\rho)
\end{equation*}
is a closed immersion. The image of the closed immersion $i_{y_6}$ equals the locus $(X,\iota,\lambda,\rho)$ of $\M_{b,\mu}$ where $\rho^{-1}\circ y_6\circ \rho \in \End(X)$. }
\end{prop}

To prove Proposition \ref{clim}, we need to pay attention to the Kottwitz condition. Let $\cZ(y_6)$ be the locus $(X,\iota,\lambda,\rho)$ of $\M_{b,\mu}$ where $\rho^{-1}\circ y_6\circ \rho \in \End(X)$. 

\begin{lem}\label{imfv}
\emph{For any algebraically closed field $k$ of characteristic $p$, the last assertion in Proposition \ref{clim} holds for the sets of $k$-valued points. }
\end{lem}

\begin{proof}
Put $\cZ^{(0)}(y_6):=\cZ(y_6)\cap \M_{b,\mu}^{(0)}$. It suffices to show the equality
\begin{equation*}
i_{y_6}(\M_{D}^{(0)})(k)=\cZ^{(0)}(y_6)(k). 
\end{equation*}
Let $\DiL^{\varpi}_{y_6}(N_{b,W(k)})$ be the subset of $\DiL^{\varpi}(N_{b,W(k)})$ consisting of $y_6$-stable elements. Then the bijection in Proposition \ref{hmlt} induces a bijection
\begin{equation*}
\cZ^{(0)}(y_6)(k)\cong \DiL^{\varpi}_{y_6}(N_{b,W(k)}). 
\end{equation*}

Put $N_{b,W(k)}^{0}:=N_{b}^{0}\otimes_{W}W(k)$, and $\Lat^{2}(N_{b,k}^{0})$ be the set of lattices $M^{0}$ in $N_{b,W(k)}^{0}$ satisfying the following: 
\begin{itemize}
\item $M^{0}\subset (M^{0})^{\vee}\subset p^{-1}M^{0}$ and $\length_{W(k)}((M^{0})^{\vee}/M^{0})=2$,
\item $M^{0}\subset \F_{b,0}((M^{0})^{\vee})\subset p^{-1}M^{0}$,
\item $p(M^{0})^{\vee}\subset \F_{b,0}(M^{0})\subset (M^{0})^{\vee}$. 
\end{itemize}
\begin{claim}
For $M\in \DiL^{\varpi}_{y_6}(N_{b,W(k)})$, $M\cap N_{b,W(k)}^{0}$ is an element of $\Lat^{2}(N_{b,W(k)}^{0})$. 
\end{claim}

We prove this claim. Since $(M^{0})^{\vee}=(\varpi^{-1}M)^{0}$, all the inclusion relations hold. Moreover, we have
\begin{equation*}
1\leq \length_{W(k)}(\F_{b,0}((M^{0})^{\vee})/M^{0})\leq 3
\end{equation*}
by the equality $y_6^2=\varepsilon^2$ and the Kottwitz condition for the $p$-divisible group with $(G,\mu)$-structure over $k$ corresponding to $M$. On the other hand, we have $\length_{W(k)}((M^{0})^{\vee}/M^{0})\in 2\Z$, and hence
\begin{equation*}
\length_{W(k)}((M^{0})^{\vee}/M^{0})=\length_{W(k)}(\F_{b,0}((M^{0})^{\vee})/M^{0})=2, 
\end{equation*}
as desired. 

By Claim, we obtain a map
\begin{equation*}
\DiL^{\varpi}_{y_6}(N_{b,W(k)})\rightarrow \Lat^{2}(N_{b,W(k)}^{0});M\mapsto M\cap N_{b,W(k)}^{0}. 
\end{equation*}
This is bijective since
\begin{equation*}
\Lat^{2}(N_{b,W(k)}^{0})\rightarrow \DiL^{\varpi}_{y_6}(N_{b,W(k)});M^{0}\mapsto M^{0}\oplus \varpi((M^{0})^{\vee})
\end{equation*}
gives an inverse map. On the other hand, the bijection in Proposition \ref{hmlt} induces a bijection
\begin{equation*}
i_{y_6}(\M_{D}^{(0)})(k)\xrightarrow{\cong}\Lat^{2}(N_{b,W(k)}^{0}),
\end{equation*}
cf. \cite[Proposition 2.4]{Wang2020}. Hence we obtain $i_{y_6}(\M_{D}^{(0)})(k)=\cZ^{(0)}(y_6)(k)$ as desired. 
\end{proof}

\begin{proof}[Proof of Proposition \ref{clim}]
The injectivity of $i_{y_6}$ as a functor is a consequence of the rigidity of quasi-isogenies (the assertion after \cite[2.9]{Rapoport1996a}). To prove the rest of the assertions, it suffices to show the equality $i_{y_6}(\M_{D})=\cZ(y_6)$. Indeed, this equality and the rigidity of quasi-isogeneies implies that $i_{y_6}$ is a closed immersion. 

The inclusion $i_{y_6}(\M_{D})\subset \cZ(y_6)$ is clear. Let us prove the another inclusion. Let $x=(X,\iota,\lambda,\rho)\in \cZ(y_6)(S)$ where $S\in \nilp$. By localization, we may assume that $S=\spec R$, where $R$ is a local ring. For $i\in \{0,1\}$, let $\Lie^i(X)$ be the eigenspace of $y_6$ on $\Lie(X)$ with respect to $(-1)^{i}\varepsilon$. Then the both are locally free of rank $2$ by applying Lemma \ref{imfv} for the base change of $x$ to an algebraic closure of the residue field of $R$. Hence we have $x\in i_{y_6}(\M_{D})(S)$. 
\end{proof}

\begin{prop}\label{md00}
\emph{For any algebraically closed field $k$ of characteristic $p$, the set $i_{y_6}(\M_{D}^{(0)})(k)$ is contained in $\M_{b,\mu}^{(0,0)}(k)$. }
\end{prop}

\begin{proof}
Take $x=(X,\iota,\lambda,\rho)\in \M_{D}^{(0)}(k)$, and let $M\in \DiL^{\varpi}_{M_0,0}(N_{b,W(k)})$ be corresponding to $\iota_{y_6}(x)\in \M_{b,\mu}^{(0)}(k)$. Then there is $g\in G_{D}^{\der}(\Frac W(k))$ such that $M=g(M_{0,W(k)})$ (see also the proof of Lemma \ref{imfv}). Hence we obtain $M\in \DiL^{\varpi}(N_{b,W(k)})$ by Lemma \ref{pmgd}. 
\end{proof}

\subsection{Vertex lattices, I}\label{vln1}

\begin{dfn}
A vertex lattice in $\L_{b,0}$ is a $\Zp$-lattice $\Lambda$ satisfying $p\Lambda \subset \Lambda^{\vee}\subset \Lambda$. We call $t(\Lambda):=\dim_{\Fp}(\Lambda/\Lambda^{\vee})$ as the type of $\Lambda$. 
\end{dfn}
We denote by $\VL(\L_{b,0})$ the set of all vertex lattices in $\L_{b,0}$. 

\begin{lem}
\emph{For $\Lambda \in \VL(\L_{b,0})$, we have $t(\Lambda)\in \{1,3,5\}$. }
\end{lem}

\begin{proof}
This follows by the same argument as the proof of \cite[Lemma 5.9]{Oki2019}. 
\end{proof}

For $t\in \{1,3,5\}$, we denote by $\VL(\L_{b,0},t)$ the set of vertex lattices in $\L_{b,0}$ of type $t$. 

We denote by $\cB(\L_{b,0})$ the Bruhat--Tits building of $\SO(\L_{b,0})(\Qp)$. 

\begin{prop}\label{sovl}
\emph{There is an $\SO(\L_{b,0})(\Qp)$-equivariant bijection between $\VL(\L_{b,0})$ and the set of vertices in $\cB(\L_{b,0})$ satisfying the following. 
\begin{enumerate}
\item The set $\VL(\L_{b,0},1)\sqcup \VL(\L_{b,0},5)$ corresponds to the set of special vertices in $\cB_{b}$. 
\item If $\Lambda_0,\Lambda_1\in \VL(\L_{b,0},3)$ are distinct, then the corresponding vertices in $\cB(\L_{b,0})$ are adjacent if and only if $\Lambda_i\in \VL(\L_{b,0},t)$, $\Lambda_{1-i}\in \VL(\L_{b,0},t')$ and $\Lambda_{i}\subset \Lambda_{1-i}$ for some $i\in \{0,1\}$ and $t,t'\in \{1,3,5\}$ satisfying $t<t'$. 
\end{enumerate}}
\end{prop}

\begin{proof}
Let $\cB'(\L_{b,0})$ be the Bruhat--Tits building of $\O(\L_{b,0})(\Qp)$. Then, by \cite{Abramenko2002} (for $\mathcal{D}=\Qp$, $(\overline{\,\cdot\,},\epsilon)=(\id,1)$ and $(n,r)=(6,2)$), all the assertions are true when we replace $\cB(\L_{b,0})$ with $\cB'(\L_{b,0})$. On the other hand, recall that $\SO(\L_{b,0})$ is the identity component of $\O(\L_{b,0})$, and the center of $\SO(\L_{b,0})$ is $\{\pm \id_{\L_{b,0}}\}\cong \mu_{2}$, which is finite over $\Qp$. Hence \cite[\S 2.1, p.44]{Tits1979} implies that there is a canonical isomorphism of simplicial complexes $\cB(\L_{b,0})\cong \cB'(\L_{b,0})$, which is $\SO(\L_{b,0})$-equivariant. Therefore we obtain the desired assertions. 
\end{proof}

\begin{lem}\label{ltcn}
\emph{For $\Lambda,\Lambda' \in \VL(\L_{b,0})$, there is a sequence of vertex lattices in $\L_{b,0}$
\begin{equation*}
\Lambda_0=\Lambda,\Lambda_1,\ldots,\Lambda_n=\Lambda'
\end{equation*}
such that $\Lambda_i\subset \Lambda_{i+1}$ or $\Lambda_i\supset \Lambda_{i+1}$ for each $i\in \{0,\ldots,n-1\}$. }
\end{lem}

\begin{proof}
This follows from the same argument as \cite[Proposition 5.1.5]{Howard2017}. 
\end{proof}

\begin{lem}\label{ppy6}
\emph{The perpendicular of $\Qp y_6$ in $\L_{b,0}$ equals $\L_{D}$. }
\end{lem}

\begin{proof}
The proof is the same as that of \cite[Proposition 3.20]{Oki2019} by using Lemma \ref{picl}. 
\end{proof}

\begin{lem}\label{y6tr}
\emph{
For $\Lambda \in \VL(\L_{b,0})$, there is $g\in J_{b}^{\der}(\Qp)$ and $\Lambda'\in \VL(\L_{D},t(\Lambda))$ satisfying
\begin{equation*}
g(\Lambda)=\Lambda'\oplus \Zp y_6. 
\end{equation*}}
\end{lem}

\begin{proof}
Take $\Lambda'' \in \VL(\L_{D},t(\Lambda))$. By Lemma \ref{ppy6}, we have $\Lambda'' \oplus \Zp y_6\in \VL(\L_{b,0},t(\Lambda))$. Hence \cite[Lemma 29.9]{Shimura2010} implies that there is $g'\in \SO(\L_{b,0})$ such that $g'(\Lambda)=\Lambda'' \oplus \Zp y_6$. On the other hand, since $\dim_{\Qp}(\L_{D})>3$, Proposition \ref{h1e0} for $\L_{D}$ implies that there is $h\in \SO(\L_{D})$ such that $\theta_{\L_{D}}(h)=\theta_{\L_{b,0}}(g')$. Hence $hg'$ is contained in the image of $J_{b}^{\der}(\Qp)\rightarrow \SO(\L_{b,0})$ by Proposition \ref{h1e0} for $\L_{b,0}$. Therefore, the element $g\in J_{b}^{\der}(\Qp)$ corresponding to $hg'$ and the vertex lattice $\Lambda':=h(\Lambda'')$ give a desired assertion. 
\end{proof}

For a $W(k)$-lattice $L$ in $\L_{b,W(k)}$ and $r\in \Znn$, we define another $W(k)$-lattice in $\L_{b,W(k)}$ as follows: 
\begin{equation*}
L^{(r)}:=L+\Phi_{b}(L)+\cdots +\Phi_{b}^{r}(L). 
\end{equation*}

\begin{prop}\label{spvl}
\emph{For $L\in \SpL(\L_{b,W(k)})$, there is the minimum integer $r\in \{0,1,2\}$ such that $L^{(r)}$ is $\Phi_{b}$-stable. Then, we have $L^{(r)}\in \VL(\L_{b,0},2r+1)$ and $(L^{(r)})^{\vee}=\{v\in L^{\vee}\mid \Phi_{b}(v)=v\}$. }
\end{prop}

\begin{proof}
We follow the proof of \cite[Lemma 2.7]{Cho2018}. Let 
\begin{equation*}
r_0:=\min\{r\in \Znn\mid \Phi_{b}(L^{(r)})=L^{(r)}\},\quad r_0^{\vee}:=\min\{r\in \Znn\mid \Phi_{b}((L^{\vee})^{(r)})=(L^{\vee})^{(r)}\}. 
\end{equation*}
Note that the existences of $r_0$ are $r_0^{\vee}$ are consequences of \cite[Proposition 2.17]{Rapoport1996a}. Moreover, the assertions $(L^{(r_0)})^{\vee}\subset L^{(r_0)}$ and $(L^{(r_0)})^{\vee}=\{v\in L^{\vee}\mid \Phi_{b}(v)=v\}$ are trivial by definition. 

In the sequel, we prove $pL^{(r_0)}\subset (L^{(r_0)})^{\vee}$, which implies $r_0\in \{0,1,2\}$ and $L^{(r_0)}\in \VL(\L_{b,0},2r_0+1)$. 
\begin{claim}
We have $r_0\leq r_0^{\vee}$. 
\end{claim}

By the same argument as the proof of \cite[Lemma 2.8]{Cho2018}, we have
\begin{equation*}
L^{(i)}\cap \Phi_b(L^{(i)})=\Phi_{b}(L^{(i-1)})
\end{equation*}
for $1\leq i<r_0$. If $r_0^{\vee}<r_0$, then the above equality implies
\begin{equation*}
L^{\vee}\subset \bigcap_{n\in \Z}\Phi_b^{n}(L^{(r_0^{\vee})})=\bigcap_{n\in \Z}\Phi_b^{n}(L)\subset L\cap \Phi_{b}(L)\subsetneq L. 
\end{equation*}
Since $\length_{W(k)}L/L^{\vee}=1$, we have $L^{\vee}=L\cap \Phi_{b}(L)$, and hence $(L^{\vee})^{(1)}=L$ by taking duals. Therefore we obtain $r_0^{\vee}=r_0+1>r_0$, which contradicts to the assumption. 

By Proposition \ref{dlph} and the definition of $r_0$, we have
\begin{equation*}
pL^{(r_0)}\subset \bigcap_{n\in \Z}\Phi_b^{n}((L^{\vee})^{(r_0-1)}). 
\end{equation*}
On the other hand, the same argument as the proof of \cite[Lemma 2.8]{Cho2018} implies
\begin{equation*}
(L^{\vee})^{(i)}\cap \Phi_b((L^{\vee})^{(i)})=\Phi_{b}((L^{\vee})^{(i-1)})
\end{equation*}
for $1\leq i<r_0^{\vee}$. Since $r_0-1<r_0^{\vee}$ by Claim, we obtain
\begin{equation*}
\bigcap_{n\in \Z}\Phi_b^{n}((L^{\vee})^{(r_{0}-1)})=\bigcap_{n\in \Z}\Phi_b^{n}(L^{\vee})=(L^{(r_0)})^{\vee}. 
\end{equation*}
Therefore the assertion $pL^{(r_0)}\subset (L^{(r_0)})^{\vee}$ is obtained. 
\end{proof}

For $\Lambda \in \VL(\L_{b,0})$, let $\M_{b,\mu,\Lambda}^{(0,0)}(k)$ be the subset of $\M_{b,\mu}^{(0,0)}(k)$ corresponding to $\{L\in \SpL(\L_{b,W(k)})\mid \Lambda^{\vee} \subset L\}$ under the bijection in Proposition \ref{dlsl}. 

\begin{cor}\label{cvkv}
\emph{We have equalities
\begin{equation*}
\M_{b,\mu}^{(0,0)}(k)=\bigcup_{\Lambda \in \VL(\L_{b,0})}\M_{b,\mu,\Lambda}^{(0,0)}(k)
=\bigcup_{g\in J_{b}^{\der}(\Qp)}\bigcup_{\Lambda \in \VL(\L_{D},5)}g(i_{y_6}(\M_{D,\Lambda}^{(0)}))(k). 
\end{equation*}}
\end{cor}

\begin{proof}
The first equality follows from Proposition \ref{spvl}. The second equality follows from Lemma \ref{y6tr}, Corollary \ref{exis} and Proposition \ref{gu2d} (iii). 
\end{proof}

\subsection{Connected components}

Here we study the connected components of the underlying topological space $\M_{b,\mu}^{(0),\red}$ of $\M_{b,\mu}^{(0)}$. For $x\in \M_{b,\mu}^{(0)}$, we denote by $\overline{k(x)}$ an algebraic closure of the residue field of $x$. 

\begin{dfn}\label{mbm0}
For $j\in \Z/2$, we define a subset $\M_{b,\mu}^{(0,j),\red}$ of the topological space $\M_{b,\mu}^{(0),\red}$ as follows: 
\begin{equation*}
\M_{b,\mu}^{(0,j),\red}:=\{x\in \M_{b,\mu}^{(0),\red}\mid x\in \M_{b,\mu}^{(0,j),\red}(\overline{k(x)})\}. 
\end{equation*}
\end{dfn}
By definition, we have 
\begin{equation*}
\M_{b,\mu}^{(0),\red}=\M_{b,\mu}^{(0,0),\red}\sqcup \M_{b,\mu}^{(0,1),\red}. 
\end{equation*}
Moreover, the action of $s\in J_{b}(\Qp)$ satisfying $\kappa_{J_b}(s)=(0,1)$ induces a bijection
\begin{equation*}
\M_{b,\mu}^{(0,0),\red}\cong \M_{b,\mu}^{(0,1),\red}. 
\end{equation*}

Here we prove the following: 

\begin{thm}\label{dcnt}
\emph{The subsets $\M_{b,\mu}^{(0,0),\red}$ and $\M_{b,\mu}^{(0,1),\red}$ are the connected components in $\M_{b,\mu}^{(0),\red}$. }
\end{thm}

By Theorem \ref{dcnt}, we obtain a decomposition into open and closed formal subschemes
\begin{equation*}
\M_{b,\mu}^{(0)}=\M_{b,\mu}^{(0,0)}\sqcup \M_{b,\mu}^{(0,1)}
\end{equation*}
such that the underlying topological space of $\M_{b,\mu}^{(0,j)}$ is $\M_{b,\mu}^{(0,j),\red}$ for each $j\in \Z/2$. 

From now, we give a proof of Theorem \ref{dcnt}. 

\begin{lem}\label{mbdu}
\emph{
We have an equality
\begin{equation*}
\M_{b,\mu}^{(0,0),\red}=\bigcup_{g\in J_{b}^{\der}(\Qp)}\bigcup_{\Lambda \in \VL(\L_{D},5)}g(i_{y_6}(\M_{D,\Lambda}^{(0)})). 
\end{equation*}}
\end{lem}

\begin{proof}
This follows form Corollary \ref{cvkv}. 
\end{proof}

\begin{cor}\label{mbir}
\emph{Any element of $\Irr(\M_{b,\mu}^{(0),\red})$ is of the form $g(i_{y_6}(\M_{D,\Lambda}^{(0)}))$ where $g\in J_{b}^{1}(\Qp)$ and $\Lambda \in \VL(\L_{D},5)$. In particular, any $C\in \Irr(\M_{b,\mu}^{(0)})$ satisfies either $C\subset \M_{b,\mu}^{(0,0),\red}$ or $C\subset \M_{b,\mu}^{(0,1),\red}$. }
\end{cor}

\begin{proof}
By Lemma \ref{mbdu}, we have
\begin{equation*}
\M_{b,\mu}^{(0),\red}=\bigcup_{g\in J_{b}^{1}(\Qp)}\bigcup_{\Lambda \in \VL(\L_{D},5)}g(i_{y_6}(\M_{D,\Lambda}^{(0)})). 
\end{equation*}
Hence the assertion follows from the above equality. 
\end{proof}

\begin{proof}[Proof of Theorem \ref{dcnt}]
It suffices to show that $\M_{b,\mu}^{(0,0),\red}$ is connected and open in $\M_{b,\mu}^{(0),\red}$. Indeed, these imply the same assertion for $\M_{b,\mu}^{(0,1),\red}$, and hence the assertion follows. 

First, we prove the connectedness of $\M_{b,\mu}^{(0,0),\red}$. By Proposition \ref{gu2d} (i) and Corollary \ref{mbir}, it suffices to show the following: 
\begin{claim}
For any $g,g'\in J_{b}^{\der}(\Qp)$ and $\Lambda,\Lambda'\in \VL(\L_{D})$, there is a sequence
\begin{equation*}
\Lambda_0=g(\Lambda \oplus \Zp y_6),\Lambda_1=g_1(\Lambda_{0,1}\oplus \Zp y_6),\ldots,\Lambda_n=g'(\Lambda'\oplus \Zp y_6)
\end{equation*}
where $g_i\in J^{\der}(\Qp)$ and $\Lambda_{0,i}\in \VL(\L_{D})$ such that $\Lambda_i\subset \Lambda_{i+1}$ or $\Lambda_i\supset \Lambda_{i+1}$ for each $i\in \{0,\ldots,n-1\}$. 
\end{claim}
However, this claim follows from Lemma \ref{ltcn} since Lemma \ref{y6tr} implies the equality
\begin{equation*}
\VL(\L_{b,0})=\{g(\Lambda \oplus \Zp y_6)\mid g\in J_{b}^{\der}(\Qp),\Lambda \in \VL(\L_{D})\}. 
\end{equation*}
Second, we prove that $\M_{b,\mu}^{(0,0),\red}$ is open in $\M_{b,\mu}^{(0),\red}$. Let $x\in \M_{b,\mu}^{(0,0),\red}$, and take a quasi-compact open subset $V'$ of $\M_{b,\mu}^{(0),\red}$ containing $x$. Then there are only finitely many irreducible components which intersects with $V'$. Now let $\Irr(\M_{b,\mu}^{(0,0),\red})_{V',1}$ be the set of $C\in \Irr(\M_{b,\mu}^{(0,0),\red})$ satisfying $C\cap V'\neq \emptyset$ and $C\subset \M_{b,\mu}^{(0,1)}$. Then the set
\begin{equation*}
U':=V'\setminus \bigcup_{C\in \Irr(\M_{b,\mu}^{(0),\red})_{V',1}}C. 
\end{equation*}
is an open neighborhood of $x$ contained in $\M_{b,\mu}^{(0,0),\red}$ by Corollary \ref{mbir} and the finiteness of the set $\Irr(\M_{b,\mu}^{(0),\red})_{V',1}$. 
\end{proof}

\subsection{Bruhat--Tits stratification, I}\label{bts1}

Here we will describe the structure of $\M_{b,\mu}^{(0,0),\red}$ by means of vertex lattices. 

\begin{dfn}
Let $\Lambda \in \VL(\L_{b,0})$. We define $\M_{b,\mu,\Lambda}^{(0,0)}$ as the locus $(X,\iota,\lambda,\rho)$ of $\M_{b,\mu}^{(0,0)}$ where $\rho^{-1}\circ \Lambda^{\vee}\circ \rho \subset \End(X)$. It is a closed formal subscheme of $\M_{b,\mu}^{(0,0)}$. 
\end{dfn}
By definition, the set of $k$-valued points of $\M_{b,\mu,\Lambda}^{(0,0)}$ equals $\M_{b,\mu,\Lambda}^{(0,0)}(k)$ for any algebraically closed field of characteristic $p$. 

\begin{prop}\label{cpbt}
\emph{
\begin{enumerate}
\item For $\Lambda \in \VL(\L_{D})$, the closed immersion $i_{y_6}$ in Proposition \ref{clim} induces an isomorphism
\begin{equation*}
\M_{D,\Lambda}^{(0)}\cong \M_{b,\mu,\Lambda \oplus \Zp y_6}^{(0,0)}. 
\end{equation*}
\item For $\Lambda \in \VL(\L_{b,0})$, there is $g\in J_{b}^{\der}(\Qp)$ and $\Lambda' \in \VL(\L_{D},t(\Lambda))$ such that $g$ induces an isomorphism
\begin{equation*}
\M_{b,\mu,\Lambda}^{(0,0)} \cong \M_{b,\mu,\Lambda'\oplus \Zp y_6}^{(0,0)}. 
\end{equation*}
\end{enumerate}}
\end{prop}

\begin{proof}
(i): This follows from the description of the image of $i_{y_6}$ in Proposition \ref{clim}. 

(ii): This follows from Lemma \ref{y6tr} and Corollary \ref{exis}. 
\end{proof}

\begin{cor}\label{irdm}
For $\Lambda \in \VL(\L_{b,0})$, the formal scheme $\M_{b,\mu,\Lambda}^{(0,0)}$ is a reduced scheme over $\spec \Fpbar$. Moreover, the following hold: 
\begin{itemize}
\item if $t(\Lambda)=1$, then $\M_{b,\mu,\Lambda}^{(0,0)}$ is a single point,
\item if $t(\Lambda)=3$, then $\M_{b,\mu,\Lambda}^{(0,0)}$ is isomorphic to $\P^1_{\Fpbar}$,
\item if $t(\Lambda)=5$, then $\M_{b,\mu,\Lambda}^{(0,0)}$ is isomorphic to the Fermat surface $F_p$. 
\end{itemize}
\end{cor}

\begin{proof}
These follow from Propositions \ref{gu2d} (i) and \ref{cpbt}. 
\end{proof}

\begin{prop}\label{mlui}
\emph{
\begin{enumerate}
\item We have an equality
\begin{equation*}
\M_{b,\mu}^{(0,0)}=\bigcup_{\Lambda \in \VL(\L_{b,0})}\M_{b,\mu,\Lambda}^{(0,0)}. 
\end{equation*}
\item For $\Lambda,\Lambda'\in \VL(\L_{b,0})$, we have
\begin{equation*}
\M_{b,\mu,\Lambda}^{(0,0)}\cap \M_{b,\mu,\Lambda'}^{(0,0)}=
\begin{cases}
\M_{b,\mu,\Lambda \cap \Lambda'}^{(0,0)} & \text{if }\Lambda \cap \Lambda'\in \VL(\L_{b,0}), \\
\emptyset &\text{otherwise. }
\end{cases}
\end{equation*}
\end{enumerate}}
\end{prop}

\begin{proof}
The proof is the same as \cite[Theorem 5.26]{Oki2019} by using Corollary \ref{irdm}. 
\end{proof}

For $\Lambda \in \VL(\L_{b,0})$, we set
\begin{equation*}
\BT_{b,\mu,\Lambda}^{(0,0)}:=\M_{b,\mu,\Lambda}^{(0,0)}\setminus \bigcup_{\Lambda'\subsetneq \Lambda}\M_{b,\mu,\Lambda'}^{(0,0)}. 
\end{equation*}

\begin{thm}\label{mtot}
\emph{
\begin{enumerate}
\item We have
\begin{equation*}
\M_{b,\mu}^{(0,0),\red}=\coprod_{\Lambda \in \VL(\L_{b,0})}\BT_{b,\mu,\Lambda}^{(0,0)}. 
\end{equation*}
\item For $\Lambda \in \VL(\L_{b,0})$, $\BT_{b,\mu,\Lambda}^{(0,0)}$ is a reduced scheme over $\Fpbar$. Moreover, we have
\begin{equation*}
\M_{b,\mu,\Lambda}^{(0,0),\red}=\coprod_{\Lambda'\subset \Lambda}\BT_{b,\mu,\Lambda'}^{(0,0)}, 
\end{equation*}
and the closure of $\BT_{b,\mu,\Lambda}^{(0,0)}$ equals $\M_{b,\mu,\Lambda}^{(0,0)}$. 
\end{enumerate}}
\end{thm}

\begin{proof}
(i): This follow from Proposition \ref{mlui} (i). 

(ii): The reducedness of $\BT_{b,\mu,\Lambda}^{(0,0)}$ is a consequence of Corollary \ref{irdm}. The rest of the assertions follow from Proposition \ref{mlui} (ii) and the definition of $\BT_{b,\mu,\Lambda}^{(0,0)}$. 
\end{proof}

We call the equality in Theorem \ref{mtot} (i) the \emph{Bruhat--Tits stratification of $\M_{b,\mu}^{(0,0)}$}. 

\begin{thm}\label{nfsl}
\emph{The non-formally smooth locus of $\M_{b,\mu}^{(0,0)}$ equals $\coprod_{\Lambda \in \VL(\L_{b,0},1)}\M_{b,\mu,\Lambda}^{(0,0)}$. }
\end{thm}

\begin{proof}
This follows from the same argument as the proof of \cite[Theorem 5.28]{Oki2019} by using Proposition \ref{dlsl} and Corollary \ref{irdm}. 
\end{proof}

\begin{cor}\label{ctnt}
\emph{
\begin{enumerate}
\item Each irreducible component of $\M_{b,\mu}^{(0,0)}$ contains exactly $(p+1)(p^2+1)$-non-formally smooth points of $\M_{b,\mu}^{(0,0)}$. 
\item Each non-formally smooth point of $\M_{b,\mu}^{(0,0)}$ is contained exactly $(p+1)(p^2+1)$-irreducible components. 
\item Let $C\in \Irr(\M_{b,\mu}^{(0,0)})$. For $C'\in \Irr(\M_{b,\mu}^{(0,0)})\setminus \{C\}$, the intersection $C\cap C'$ is either the empty set, a single point or isomorphic to $\P^1$. Moreover, the following hold: 
\begin{gather*}
\#\{C'\in \Irr(\M_{b,\mu}^{(0,0)}) \mid C\cap C'\text{ is a single point}\}=p(p+1)(p^2+1),\\
\#\{C'\in \Irr(\M_{b,\mu}^{(0,0)}) \mid C\cap C'=\P_{\Fpbar}^{1}\}=(p+1)(p^2+1)
\end{gather*}
\end{enumerate}}
\end{cor}

\begin{proof}
These follow from the Bruhat--Tits stratification of $\M_{b,\mu}^{(0,0)}$ and Theorem \ref{nfsl}. 
\end{proof}

\section{Bruhat--Tits stratification in the neutral case, I\hspace{-.1em}I}\label{btn2}

In this section, we give the definition of the vertex lattices in $C_{b}$, and give a connection between them and the vertex lattices in $\L_{b,0}$. Moreover, we construct a Bruhat--Tits stratification of $\M_{b,\mu}^{(0)}$ indexed by the vertex lattices in $C_{b}$. 

\subsection{Vertex lattices, I\hspace{-.1em}I}

\begin{dfn}
An $O_F$-lattice $T$ in $C_b$ is called a \emph{vertex lattice} if we have $T\subset T^{\vee}\subset \varpi^{-1}T$. We call $t(T):=\dim_{\Fp}(T^{\vee}/T)$ the \emph{type} of $T$. 
\end{dfn}

We denote by $\VL_{b}$ the set of vertex lattices. By \cite[Lemma 3.2, Lemma 3.3]{Rapoport2014a}, we have
\begin{equation*}
\VL_{b}=\VL_{b}(0)\sqcup \VL_{b}(2)\sqcup \VL_{b}(4),
\end{equation*}
where $\VL_{b}(t):=\{T\in \VL_{b} \mid t(T)=t\}$ for $t\in \{0,2,4\}$. 

We introduce a partial order $\leq$ on $\VL_{b}$. For distinct $T,T'\in \VL_{b}$, we have $T<T'$ if one of the following holds: 
\begin{itemize}
\item $T\in \VL_{b}(4)$, $T'\in \VL_{b}(0)$ and $T\subset T'$, 
\item $T\in \VL_{b}(4)$, $T'\in \VL_{b}(2)$ and $T\subset (T')^{\vee}$, 
\item $T\in \VL_{b}(0)$, $T'\in \VL_{b}(2)$ and $T\subset (T')^{\vee}$. 
\end{itemize}

We can interpret vertex lattices in $C_{b}$ by means of the Bruhat--Tits building $\cB_{b}$ of $J_{b}^{\der}(\Qp)=\SU(C_{b})(\Qp)$. 

\begin{prop}\label{rtwb}(\cite[Proposition 3.4]{Rapoport2014a})
\emph{There is a $J_{b}^{\der}(\Qp)$-equivariant bijection between $\VL_{b}(0)\sqcup \VL_{b}(4)$ and the set of vertices in $\cB_{b}$ satisfying the following. 
\begin{enumerate}
\item The set $\VL_{b}(4)$ corresponds to the set of special vertices in $\cB_{b}$. 
\item If $T_0,T_1\in \VL_{b}(0)\sqcup \VL_{b}(4)$ are distinct, then the corresponding vertices in $\cB_{b}$ are adjacent if and only if either one of the following hold: 
\begin{itemize}
\item $T_i\in \VL_{b}(0)$, $T_{1-i}\in \VL_{b}(4)$ and $T_{1-i}\subset T_{i}$ for some $i\in \{0,1\}$,
\item $T_0,T_1\in \VL_{b}(4)$ and $T_0+T_1\in \VL_{b}(2)$. 
\end{itemize}
In particular, $\VL_{b}(2)$ corresponds to the set of midpoints of edges connecting two adjacent special vertices in $\cB_{b}$. 
\end{enumerate}}
\end{prop}

Let $\bT_0$ be the $\F_{b,0}$-fixed part of $M_0$ (note that $M_0$ is $\F_{b,0}$-stable), and put
\begin{equation*}
\VL_{b,\bT_0}(4):=\{T\in \VL_{b}(4)\mid \length_{O_F}(T+\bT_0)/\bT_0\in 2\Z\}. 
\end{equation*}
We also introduce an order $\preceq$ on $\VL_{b}(0)\sqcup \VL_{b}(4)$. For distinct $T,T'\in \VL_{b}(0)\sqcup \VL_{b}(4)$, we have $T\prec T'$ if one of the following hold: 
\begin{itemize}
\item $T\in \VL_{b,\bT_0}(4)$, $T'\in \VL_{b}(0)$ and $T\subset T'$, 
\item $T\in \VL_{b,\bT_0}(4)$, $T'\in \VL_{b}(4)\setminus \VL_{b,\bT_0}(4)$ and $T\subset \varpi^{-1}T'$, 
\item $T\in \VL_{b}(0)$, $T'\in \VL_{b}(4)\setminus \VL_{b,\bT_0}(4)$ and $T\subset \varpi^{-1}T'$. 
\end{itemize}

\begin{prop}\label{vlvl}
\emph{There is a $J_{b}^{\der}(\Qp)$-equivariant isomorphism of ordered sets
\begin{equation*}
(\VL_{b}(0)\sqcup \VL_{b}(4),\preceq)\cong (\VL(\L_{b,0}),\subset)
\end{equation*}
which induces the following bijections:
\begin{equation*}
\VL_{b,\bT_0}(4)\cong \VL(\L_{b,0},1),\quad \VL_{b}(0)\cong \VL(\L_{b,0},3),\quad \VL_{b}(4)\setminus \VL_{b,\bT_0}(4)\cong \VL(\L_{b,0},5). 
\end{equation*}}
\end{prop}

\begin{proof}
Since $J_{b}^{\der}\rightarrow \SO(\L_{b,0})$ is a central isogeny by Corollary \ref{exis}, \cite[\S 4.2.15]{Bruhat1984} implies that there is a $J_{b}^{\der}(\Qp)$-equivariant isomorphism of simplicial complexes
\begin{equation*}
\cB_{b}\cong \cB(\L_{b,0}). 
\end{equation*}
We may assume that $\bT_{0}$ corresponds to a vertex lattice in $\L_{b,0}$ of type $1$. Then, we obtain the desired isomorphism by Propositions \ref{sovl} and \ref{rtwb}. 
\end{proof}

\subsection{Bruhat--Tits stratification, I\hspace{-.1em}I}\label{cbbt}

\begin{lem}\label{dtst}
\emph{Let $T\in \VL_{b}$, and $a\in 1+\varpi O_{F}$ a norm $1$ element. Then there is $h\in J^1(\Qp)$ such that $\det\!{}_{F}(h)=a$, $h(T)=T$. }
\end{lem}

\begin{proof}
By Hilbert satz 90, there is $t\in F^{\times}$ so that $a=t^{-1}\overline{t}$. By assumption on $a$, we may assume $t\in O_{F}^{\times}$. Take a basis of $C_{b}$ whose Gram matrix with respect to $\langle \,,\,\rangle$ is the anti-diagonal matrix of size $4$ that has $1$ at every non-zero entry. Then $h:=\diag(t,1,1,t^{-1})\in J^1(\Qp)$ satisfies $\det\!{}_{F}(h)=a$ and $h(T)=T$. 
\end{proof}

\begin{dfn}
For $T\in \VL_{b}$, we define a closed formal subscheme $\M_{b,\mu,T}^{(0)}$ of $\M_{b,\mu}^{(0)}$ as follows. Take $h\in J_{b}(\Qp)$ satisfying $\kappa_{J_{b}}(h)=(0,1)$. 
\begin{itemize}
\item If $T\in \VL_{b,\bT_0}(4)$, then let $\M_{b,\mu,T}^{(0)}:=\M_{b,\mu,\Lambda(T)}^{(0,0)}$. 
\item If $T\in \VL_{b}(4)\setminus \VL_{b,\bT_0}(4)$, then let $\M_{b,\mu,T}^{(0)}:=h\left(\M_{b,\mu,\Lambda(h^{-1}(T))}^{(0,0)}\right)$. 
\item If $T\in \VL_{b}(0)$, then let $\M_{b,\mu,T}^{(0)}:=\M_{b,\mu,\Lambda(T)}^{(0,0)}\sqcup h\left(\M_{b,\mu,\Lambda(h^{-1}(T))}^{(0,0)}\right)$. 
\item If $T\in \VL_{b}(2)$, then let $\M_{b,\mu,T}^{(0)}:=\M_{b,\mu,\Lambda(T_0)}^{(0,0)}\sqcup h\left(\M_{b,\mu,\Lambda(h^{-1}(T_1))}^{(0,0)}\right)$. Here $T_{0}\in \VL_{b}(4)\setminus \VL_{b,\bT_0}(4)$ and $T_1\in \VL_{b,\bT_0}(4)$ satisfying $T_0+T_1=T$ (note that $T_0$ and $T_1$ are unique by Lemma \ref{rtwb} (ii)). 
\end{itemize}
We also define a locally closed subscheme $\BT_{b,\mu,T}^{(0)}$ of $\M_{b,\mu}^{(0)}$ by the same manner as above. 
\end{dfn}

Note that $\M_{b,\mu,T}^{(0)}$ and $\BT_{b,\mu,T}^{(0)}$ are independent of the choice of $h$ by Lemma \ref{dtst} and the equality $h(\M_{b,\mu,T'}^{(0,0)})=\M_{b,\mu,h(T')}^{(0,0)}$ for $h\in J_{b}^{\der}(\Qp)$ and $T'\in \VL(0)\sqcup \VL(4)$. 

\begin{thm}\label{mthm}
\emph{
\begin{enumerate}
\item We have
\begin{equation*}
\M_{b,\mu}^{(0),\red}=\coprod_{T\in \VL_{b}}\BT_{b,\mu,T}^{(0)}. 
\end{equation*}
\item For $T\in \VL_{b}$, $\M_{b,\mu,T}^{(0)}$ and $\BT_{b,\mu,T}^{(0)}$ are reduced schemes over $\Fpbar$. Moreover, we have
\begin{equation*}
\M_{b,\mu,T}^{(0),\red}=\coprod_{T'\leq T}\BT_{b,\mu,T'}^{(0)}. 
\end{equation*}
\item Let $T\in \VL_{b}$. 
\begin{itemize}
\item if $t(T)=4$, then $\M_{b,\mu,T}^{(0)}$ is a single point,
\item if $t(T)=0$, then $\M_{b,\mu,T}^{(0)}$ is isomorphic to two copies of $\P^1_{\Fpbar}$, 
\item if $t(T)=2$, then $\M_{b,\mu,T}^{(0)}$ is isomorphic to two copies of the Fermat surface $F_p$.  
\end{itemize}
\item There is a bijection between $\Irr(\M_{b,\mu}^{(0),\red})$ and $\VL(4)$. 
\item The non-formally smooth locus of $\M_{b,\mu}^{(0)}$ equals $\coprod_{T\in \VL_{b}(4)}\M_{b,\mu,T}$. 
\end{enumerate}}
\end{thm}

We call the equality in Theorem \ref{mthm} by the Bruhat--Tits stratification of $\M_{b,\mu}^{(0)}$. 

\begin{proof}
The assertions (i)--(iii) follow from Theorem \ref{mtot} and the definitions of $\M_{b,\mu,T}^{(0)}$ and $\BT_{b,\mu,T}^{(0)}$. The assertion (iv) is a consequence of Corollary \ref{mbir}. The assertion (v) follows from Theorem \ref{nfsl}. 
\end{proof}

Finally, we give a connection beween the Bruhat--Tits strata of $\M_{b,\mu}^{(0)}$ and some generalized Deligne--Lusztig varieties for split even special orthogonal groups appeared in Proposition \ref{lwsp}. For $T\in \VL_{b}$, put
\begin{equation*}
\B_{b}(T):=
\begin{cases}
T/\varpi T^{\vee}&\text{if }t(T)\in \{0,4\},\\
T^{\vee}/\varpi T &\text{if }t(T)=2. 
\end{cases}
\end{equation*}
We regard it as a quadratic space over $\Fp$ by $(x,y)\mapsto (\varpi x,y)\bmod p$. Then $\B_{b}(T)$ is a split quadratic space over $\Fp$ of dimension
\begin{equation*}
\begin{cases}
0 &\text{if }t(T)=4,\\
4 &\text{if }t(T)=0,\\
6 &\text{if }t(T)=2. 
\end{cases}
\end{equation*}

\begin{thm}\label{dlnt}
\emph{Let $T\in \VL_{b}$. 
\begin{enumerate}
\item There is an isomorphism $\M_{b,\mu,T}^{(0)}\cong S_{\B_{b}(T)}$ (see Definition \ref{grso} for the definition of the right-hand side). 
\item If $t(T)=0$, then there is an isomorphism $\BT_{b,\mu,T}^{(0)}\cong X_{P_{4,0}^{+}}(t_{4}^{-})\sqcup X_{P_{4,0}^{-}}(t_{4}^{+})$. 
\item If $t(T)=2$, then there is an isomorphism $\BT_{b,\mu,T}^{(0)}\cong X_{P_{6,1}^{+}}(t_{6}^{-}s_{6,1})\sqcup X_{P_{6,1}^{-}}(t_{6}^{+}s_{6,1})$. 
\end{enumerate}
In particular, every connected component of $\BT_{b,\mu,T}$ is isomorphic to a generalized Deligne--Lusztig variety for $\SO(\B_b(T))$, which is non-classical if $T\in \VL_{b}(0)\sqcup \VL_{b}(2)$. }
\end{thm}

\begin{proof}
The assertion (i) for $t(T)=0$ is a consequence of Theorem \ref{mthm} (iii). The remaining assertions follow from Propositions \ref{lwsp}, \ref{cpbt} and \cite[Corollary 6.10]{Oki2019}. 
\end{proof}

\section{Bruhat--Tits stratification in the non-neutral case}\label{btnn}

In this section, we define the notion of vertex lattices, and construct a stratification of $\M_{b,\mu}^{(0)}$ by means of the vertex lattices for a non-$\mu$-neutral $b$. The method is almost the same as \cite{Wu2016} except for the handling of connected components. 

\begin{lem}\label{dlp1}
\emph{For $M\in \DiL^{\varpi}(N_{b,W(k)})$, we have $\length_{O_{F,W(k)}}(M+\F_{b,0}(M))/M=1$. } 
\end{lem}

\begin{proof}
This follows from Proposition \ref{nteq} and the non-$\mu$-neutralness of $b$. 
\end{proof}

\begin{dfn}
A \emph{vertex lattice} in $C_b$ is an $O_F$-lattice $T$ in $C_b$ satisfying $T\subset T^{\vee}\subset \varpi^{-1}T$. We call $t(T):=\dim_{O_F}(T^{\vee}/T)$ as the \emph{type} of $T$. 
\end{dfn}

\begin{prop}\label{divl}
\emph{For $M\in \DiL^{\varpi}(N_b)$, there is the minimum integer $r\in \{1,2\}$ such that
\begin{equation*}
M^{(r)}:=M+\F_{b,0}(M)+\cdots +\F_{b,0}^{r}(M)
\end{equation*}
is $\F_{b,0}$-stable. Then we have $M^{(r)}\in \VL_{b}(2r)$ and $(M^{(r)})^{\vee}=\{x\in M\mid \F_{b,0}(x)=x\}$. }
\end{prop}

\begin{proof}
This follows from the same argument as the proof of \cite[Proposition 4.3]{Rapoport2014a} (by considering dual lattices) by using Lemma \ref{dlp1}. 
\end{proof}

We denote by $M_0$ the Dieudonn{\'e} module of $\X_0$ in $N_{b}$. 

\begin{dfn}
Let $T\in \VL_b$. 
\begin{enumerate}
\item We denote by $(X_{T},\iota_T,\lambda_T)$ the $p$-divisible group with $(G,\mu)$-structure over $\Fpbar$ corresponding to $T\otimes_{\Zp}W$ by the covariant Dieudonn{\'e} theory. Then $\lambda_{T}$ is an isogeny whose kernel is contained in $\Ker(\iota_{T}(\varpi))$. Moreover, there is a quasi-isogeny $\rho_{T}\colon \X_0 \rightarrow X_{T}$ such that the following diagram commutes: 
\begin{equation*}
\xymatrix{
\X_0\ar[r]^{\lambda_0}\ar[d]_{\rho_{T}}& \X_{0}^{\vee}\\
X_{T}\ar[r]^{\lambda_{T}}&X_{T}^{\vee}\ar[u]_{\rho_{T}^{\vee}}. }
\end{equation*}
\item We define $\M_{b,\mu,T}^{(0)}$ as the locus $(X,\iota,\lambda,\rho)$ of $\M_{b,\mu}^{(0)}$ where the quasi-isogeny 
\begin{equation*}
X\times_{S}\Sbar \xrightarrow{\rho} \X_0\otimes_{\Fpbar}\Sbar \xrightarrow{\rho_{T}\otimes \id_{\Sbar}} X_{T}\otimes_{\Fpbar}\Sbar
\end{equation*}
lift to an isogeny. It is a closed formal subscheme of $\M_{b,\mu}^{(0)}$. 
\end{enumerate}
\end{dfn}

By definition, $\M_{b,\mu,T}^{(0)}(k)$ corresponds to the set $\{M\in \DiL^{\varpi}(M_{0,W(k)})\mid M\subset T\otimes_{\Zp}W(k)\}$ under the bijection $\M_{b,\mu}^{(0)}(k)\cong \DiL^{\varpi}(M_{0,W(k)})$ for any field extension $k/\Fpbar$. 

We study the structure of $\M_{b,\mu,T}^{(0)}$ for $T\in \VL_{b}$. For $T\in \VL_{b}$, put $\B_{b}(T):=T/\varpi T^{\vee}$. We regard it as a quadratic space over $\Fp$ by $(x,y)\mapsto (\varpi x,y)\bmod p$. Then it is a non-split quadratic space of dimension $4-t(T)$ over $\Fp$. 

\begin{prop}\label{btgr}
\emph{For $T\in \VL_{b}$, there is an isomorphism $\M_{b,\mu,T}^{(0),\red}\cong S_{\B_{b}(T)}$ (see Definition \ref{grso} for the definition of the right-hand side). }
\end{prop}

\begin{proof}
Let $(X,\iota,\lambda,\rho)\in \M_{b,\mu,T}^{(0)}(S)$ where $S\in \nilp$ is reduced (in particular, we have $S=\Sbar$). Then $(\rho_{T}\circ \rho)^{\vee}$ is an isogeny whose kernel is contained in $\Ker(\iota_{T}(\varpi)\circ \lambda_{T}^{-1})$ (note that $\iota_{T}(\varpi)\circ \lambda_{T}^{-1}$ is an isogeny). By \cite[Corollary 4.7]{Vollaard2011}, $\Ker(D((\rho_{T}\circ \rho)^{\vee}))$ is an $\O_S$-direct summand of
\begin{equation*}
\Ker(D(\iota_{T}(\varpi)\circ \lambda_{T}^{-1}))_{S}=\B_{b}(T)\otimes_{\Fp}\O_S
\end{equation*}
(here we use $S=\Sbar$). Moreover, it is of rank $2-t(T)/2$ since $\Ht(\iota_{T}(\varpi)\circ \lambda_{T}^{-1})=2\Ht((\rho_{T}\circ \rho)^{\vee}))$. Hence we obtain a morphism 
\begin{equation*}
\M_{b,\mu,T}^{(0),\red}\rightarrow \OGr(\B_{b}(T))
\end{equation*}
By Lemma \ref{dlp1}, this morphism induces a morphism
\begin{equation*}
\M_{b,\mu,T}^{(0),\red}\rightarrow S_{\B_{b}(T),\Fpbar}, 
\end{equation*}
which is bijective on the sets of $k$-valued points for any field extension $k/\Fpbar$. Since $\M_{b,\mu,T}^{(0),\red}$ is projective and $S_{\B_{b}(T),\Fpbar}$ is smooth, the induced morphism is an isomorphism by the same argument as the proof of \cite[Theorem 3.9]{Howard2014}. 
\end{proof}

\begin{thm}\label{mtdm}
\emph{Let $T\in \VL_{b}$. 
\begin{enumerate}
\item If $t(T)=2$, then $\M_{b,\mu,T}^{(0),\red}$ consists of two points. 
\item If $t(T)=0$, then $\M_{b,\mu,T}^{(0),\red}$ is isomorphic to two copies of $\P_{\Fpbar}^{1}$. 
\end{enumerate}}
\end{thm}

\begin{proof}
These follow from Proposition \ref{btgr} and \cite[Section 3.3]{Howard2014}. 
\end{proof}

\begin{prop}\label{mtnn}
\emph{
\begin{enumerate}
\item We have an equality
\begin{equation*}
\M_{b,\mu}^{(0),\red}=\bigcup_{T\in \VL_{b}}\M_{b,\mu,T}^{(0),\red}. 
\end{equation*}
\item For $T,T'\in \VL_{b}$, we have $\M_{b,\mu,T'}^{(0)}\subset \M_{b,\mu,T}^{(0)}$ if and only if $T'\subset T$. 
\end{enumerate}}
\end{prop}

\begin{proof}
These follow from the same argument as the proof of \cite[Proposition 2.5]{Vollaard2010} by using Proposition \ref{divl} and Theorem \ref{mtdm}. 
\end{proof}

For $T\in \VL_{b}$, set
\begin{equation*}
\BT_{b,\mu,T}^{(0)}:=\M_{b,\mu,T}^{(0),\red}\setminus \bigcup_{T'\subsetneq T}\M_{b,\mu,T'}^{(0),\red}. 
\end{equation*}

\begin{thm}\label{nnbt}
\emph{
\begin{enumerate}
\item We have an equality
\begin{equation*}
\M_{b,\mu}^{(0),\red}=\coprod_{T\in \VL_{b}}\BT_{b,\mu,T}^{(0)}. 
\end{equation*}
\item For $T,T'\in \VL_{b}$, we have 
\begin{equation*}
\M_{b,\mu,T}^{(0),\red}=\coprod_{T'\subset T}\BT_{b,\mu,T'}^{(0)}. 
\end{equation*}
\end{enumerate}}
\end{thm}

\begin{proof}
The proof is the same as Theorem \ref{mtot} by using Proposition \ref{mtnn}. 
\end{proof}

The equality in Theorem \ref{nnbt} is called the \emph{Bruhat--Tits stratification} of $\M_{b,\mu}^{(0)}$. 

For $j\in \Z/2$, we define a subset $\DiL^{\varpi}_{M_0,j}(k)$ of $\DiL^{\varpi}(M_{0,W(k)})$ and a subset $\M_{b,\mu}^{(0,j)}(k)$ of $\M_{b,\mu}^{(0)}$ by the same way as Definition \ref{dlm0} for any field extension $k/\Fpbar$. Moreover, we define a subspace $\M_{b,\mu}^{(0,j)}$ of $M_{b,\mu}^{(0)}$ by the same manner as Definition \ref{mbm0}. 

\begin{lem}\label{mtnd}
\emph{For $T\in \VL_{b}$, we have a decomposition into open and closed subsets
\begin{equation*}
\M_{b,\mu,T}^{(0)}=\M_{b,\mu,T}^{(0,0)}\sqcup \M_{b,\mu,T}^{(0,1)}, 
\end{equation*}
where $\M_{b,\mu,T}^{(0,j)}:=\M_{b,\mu,T}^{(0)}\cap \M_{b,\mu}^{(0),\red}$. In particular, the decomposition as above becomes as schemes. }
\end{lem}

\begin{proof}
This follows from the definition of $\M_{b,\mu,T}^{(0,j)}$ and Proposition \ref{btgr}. 
\end{proof}

\begin{cor}\label{ctnn}
\emph{Let $C\in \Irr(\M_{b,\mu}^{(0,0)})$. For $C'\in \Irr(\M_{b,\mu}^{(0,0)})\setminus \{C\}$, the intersection $C\cap C'$ is either the empty set or a single point. Moreover, we have
\begin{equation*}
\#\{C'\in \Irr(\M_{b,\mu}^{(0,0)}) \mid C\cap C'\text{ is a single point}\}=p(p^2+1)
\end{equation*}}
\end{cor}

\begin{proof}
This follows from Proposition \ref{mtnn} and Lemma \ref{mtnd}. 
\end{proof}

Next, we study the connected components of $\M_{b,\mu}^{(0),\red}$. 

\begin{lem}\label{irnn}
\emph{Any irreducible component of $\M_{b,\mu}^{(0),\red}$ is of the form $\M_{b,\mu,T}^{(0,0),\red}$ or $\M_{b,\mu,T}^{(0,1)}$, where $T\in \VL_{b}(0)$. }
\end{lem}

\begin{proof}
This follows from Proposition \ref{mtnn} and Lemma \ref{mtnd}. 
\end{proof}

\begin{thm}\label{dcnn}
\emph{The subsets $\M_{b,\mu}^{(0,0),\red}$ and $\M_{b,\mu}^{(0,1),\red}$ are the connected components of $\M_{b,\mu}^{(0),\red}$. }
\end{thm}

\begin{proof}
The proof is the same as Theorem \ref{dcnt} by using Lemma \ref{irnn}. 
\end{proof}

Finally, we give a relation between the Bruhat--Tits strata of $\M_{b,\mu}^{(0)}$ and some Deligne--Lusztig varieties for even special orthogonal groups appeared in Proposition \ref{stns}. 

\begin{thm}\label{dlnn}
\emph{For $T\in \VL_{b}$, there is an isomorphism $\BT_{b,\mu,T}^{(0)}\cong X_{B_{d_b(T)}}(w_{d_b(T)}^{+})\sqcup X_{B_{d_b(T)}}(w_{d_b(T)}^{-})$. In particular, $\BT_{b,\mu,T}^{(0)}$ is isomorphic to the disjoint union of two Deligne--Lusztig varieties for $\SO(\B_b(T))$ associated with Coxeter elements. }
\end{thm}

\begin{proof}
This follows from Propositions \ref{stns} and \ref{btgr}. 
\end{proof}

\begin{rem}
We use the notations in Section \ref{lpel}. We have $G\otimes_{\Qp}F \cong \GL_{4,F}\otimes_{F}\G_{m,F}$. Set
\begin{equation*}
\mu_1 \colon \G_{m,F}\rightarrow G\otimes_{\Q}F;z\mapsto (\diag(1,1,1,z),z). 
\end{equation*}
Then a non-$\mu$-neutral $b$ is $\mu_1$-neutral. Moreover, we have $\M_{b,\mu}=\M_{b,\mu_1}$, where $\M_{b,\mu_1}$ is the Rapoport--Zink space for $(G,b,\mu_1)$ defined by the same manner as $\M_{b,\mu}$.  In \cite{Wu2016}, he first determined the connected components $\M_{b,\mu_1}^{(0,0)}$ and $\M_{b,\mu_1}^{(0,1)}$ of $\M_{b,\mu_1}^{(0)}$, and constructed the Bruhat--Tits stratification of $\M_{b,\mu_1}^{(0,0)}$. However, it seems to need a modification for determining $\M_{b,\mu_1}^{(0,j)}$. 
\end{rem}

\begin{rem}
Our method in this section can be generalized for the Rapoport--Zink space for $\GU(1,n-1)$ with even $n$ and a neutral acceptable $b$ appeared in \cite{Wu2016}. The results are exactly the same as \cite[Theorem 1]{Wu2016}. 
\end{rem}

\section{Supersingular loci of unitary Shimura varieties}\label{slus}

In this section, we study the supersingular loci of Shimura varieties for unitary similitude group of signature $(2,2)$ associated to an imaginary quadratic field $L$ over an odd prime $p$ which ramifies in $L$. 

\subsection{Rapoport--Zink integral models of unitary Shimura varieties}\label{ushv}

Let $L=\Q(\sqrt{\Delta})$ be an imaginary quadratic field, where $\Delta \in \Z_{<0}$ is square-free. We denote by $a\mapsto \overline{a}$ the non-trivial Galois automorphism of $L/\Q$. Throughout of this section, we fix an injection $\varphi \colon L\hookrightarrow \Qbar$ and an isomorphism $\Qpbar \cong \C$. Let $(U,\langle\,,\,\rangle_U)$ be an $L/\Q$-hermitian space of signature $(2,2)$. For each prime $\ell$, put 
\begin{equation*}
\inv_{\ell}(U):=\det(\langle x_i,x_j\rangle_U)_{1\leq i,j\leq 4}\in {\Q}_{\ell}^{\times}/\N_{L_{\ell}/\Ql}(L_{\ell}^{\times}), 
\end{equation*}
where $L_{\ell}:=L\otimes_{\Q}\Ql$ and $x_1,\ldots, x_4$ is an $L_{\ell}$-basis of $U_{\ell}:=U\otimes_{\Q}\Ql$. Note that $U_{\ell}$ is split as an $L_{\ell}/\Ql$-hermitian space if and only if $\inv_{\ell}(U)=1$. 

Let $p>2$ be an odd prime number. \emph{Here we assume that $p$ ramifies in $L$ and $\inv_p(U)=1$. }

Put
\begin{equation*}
(\,,\,)_U:=\frac{1}{2}\tr_{L/\Q}(\langle\,,\,\rangle_U). 
\end{equation*}
Then we have $(ax,y)_{U}=(x,\overline{a}y)_{U}$ for $x,y\in U$ and $a\in L$. We define an algebraic group $\bG$ over $\Q$ by
\begin{equation*}
\bG(R)=\{(g,c)\in \GL_{L\otimes_{\Q}R}(U\otimes_{\Q}R)\times R^{\times}\mid (g(x),g(y))_U=c(x,y)_{U}\text{ for any }x,y\in U\}
\end{equation*}
for any $\Q$-algebra $R$. It is a reductive connected group over $\Q$. Note that $\bG\otimes_{\Q}\Qp$ is quasi-split over $\Q$ since $\inv_{p}(U)=1$. 

\begin{dfn}
Let $S$ be an $\Zp$-scheme. A \emph{polarized $L$-abelian variety of signature $(2,2)$ over $S$} is a triple $(A,\iota,\lambda)$, where
\begin{itemize}
\item $A$ is an abelian variety of dimension $4$ over $S$,
\item $\iota \colon O \rightarrow \End(A)$ is a ring homomorphism,
\item $\lambda \colon A\rightarrow A^{\vee}$ is a polarization,
\end{itemize}
satisfying the following conditions for any $a\in O_L$: 
\begin{itemize}
\item $\det(T-\iota(a)\mid \Lie(A))=(T-\varphi(a))^2(T-\varphi(\overline{a}))^2$, 
\item $\lambda \circ \iota(a)=\iota(\overline{a})^{\vee}\circ \lambda$. 
\end{itemize}
\end{dfn}

From now on, we fix an order $O$ of $L$ which is maximal at $p$, and put $O_{p}:=O\otimes_{\Z}\Zp$. Let $K^p$ be a sufficiently small compact open subgroup of $\bG(\A_{f}^{p})$, where $\A_{f}^{p}$ is the finite adel{\'e} ring without $p$-component. On the other hand, let $K_p$ be a special maximal parahoric subgroup of $\bG(\Qp)$. More precisely, it is a stabilizer of a $\sqrt{\Delta}$-modular $O_{p}$-lattice in $U_{p}$. 

\begin{dfn}\label{rzit}
We define an $\Zp$-scheme $\sS_{K,U}$ as the moduli space of the functor which parametrizes equivalence classes of $(A,\iota,\lambda,\overline{\eta}^p)$ for any noetherian scheme $S$, where
\begin{itemize}
\item $(A,\iota,\lambda)$ is a polarized $L$-abelian variety of signature $(2,2)$ over $S$ satisfying
\begin{equation*}
\ord_p(\#\ker(\lambda))=\ord_p(\#\ker(\iota(\sqrt{\Delta}))),
\end{equation*}
\item $\overline{\eta}^p$ is a $K^p$-level structure, that is, a $\pi_1(S,\sbar)$-invariant $K^p$-orbit of an isomorphism 
\begin{equation*}
H_1(A_{\sbar},\A_f^p)\xrightarrow{\cong}U\otimes_{\Q}\A_f^p
\end{equation*}
which respects $(L\otimes_{\Q}\A_f^p)/\A_f^p$-hermitian forms up to an $(\A_f^p)^{\times}$-multiple. Here $\sbar$ is a geometric point of $S$. 
\end{itemize}
Two quadruples $(A,\iota,\lambda,\overline{\eta}^p)$ and $(A',\iota',\lambda',(\overline{\eta}')^p)$ are equivalent if there is a prime-to-$p$ isogeny $A\rightarrow A'$ which respects $O_L$-actions, polarizations and $K^p$-level structures. 
\end{dfn}

Set $\bS:=\Res_{\C/\R}\G_m$. Fix a $\C$-basis $e_1,\ldots ,e_4$ of $U\otimes_{\Q}\R$ whose Gram matrix of $\langle\,,\,\rangle_{U}$ is $\diag(1,1-1,-1)$, and regard $\bG \otimes_{\Q}\R$ as a subgroup of $\Res_{\C/\R}\GL_{4}$ by the above basis. Let $\bX$ the $\bG(\R)$-conjugacy class in $\Hom(\bS,\bG \otimes_{\Q}\R)$ containing
\begin{equation*}
h_U \colon \bS \rightarrow \bG \otimes_{\Q}\R;z\mapsto \diag(z^{(2)},\overline{z}^{(2)}). 
\end{equation*}
Then $(\bG,\bX)$ is a Shimura datum whose reflex field is $\Q$. Moreover, the generic fiber of $\sS_{K,U}$ equals $\Sh_{K}(\bG,\bX)\otimes_{\Q}\Qp$. This follows from that the Hasse principal holds for $\bG$, which is pointed out in \cite[Section 6.3]{Vollaard2010}. 

\subsection{Non-emptiness of the basic Newton strata}

Put $F:=L_{p}$, $O_{F}:=O\otimes_{\Z}\Zp$ and $V:=U_{p}$. Set $G:=\bG \otimes_{\Q}\Qp$, and let $\mu$ be the cocharacter of $G$ over $\Qp$ associated to the cocharacter
\begin{equation*}
\G_{m}\xrightarrow{z\mapsto (z,1)} \G_{m}\otimes_{\C}\G_{m}\cong \bS \otimes_{\R}\C \xrightarrow{h_{U,\C}}\bG\otimes_{\Q}\C
\end{equation*}
and the fixed isomorphism $\Qpbar \cong \C$. Then the datum $(F,O_{F},V,G,\mu)$ coincides with the one defined in Section \ref{lpel}. Hence we can apply the results in Sections \ref{rzsp}--\ref{btnn}. 

We study the Newton strata of the geometric special fiber $\sS_{K,U,\Fpbar}:=\sS_{K,U}\otimes_{\Zp}\Fpbar$ of $\sS_{K,U}$. 

\begin{dfn}
For $[b]\in B(G)$, we define $\sS_{K,U,[b]}$ as the a locally closed subscheme of $\sS_{K,U,\Fpbar}$ by
\begin{equation*}
\sS_{K,U,[b]}(k)=\{(A,\iota,\lambda,\overline{\eta}^p)\in \sS_{K,U}(k)\mid \D(A)_{\Q}\cong N_{b}\text{ for some (all) }b\in [b]\}
\end{equation*}
for any algebraically closed field $k$ of characteristic $p$. We call $\sS_{K,U,[b]}$ the \emph{Newton stratum} of $\sS_{K,U,\Fpbar}$ with respect to $[b]$. 
\end{dfn}
The Newton stratum $\sS_{K,U,[b]}$ is non-empty only if $\ord_p(\sml_{G}(b))=1$ for some (all) $b\in [b]$. Moreover, it is closed in $\sS_{K,U,\Fpbar}$ if $[b]\in B(G)_{1,\bs}$. 

Let $b_0$ and $b_1$ be elements of $G(K_0)$ defined in Section \ref{ntrl}. We have $[b_j]\in B(G)_{1,\bs}$ for $j\in \{0,1\}$. We prove the following: 

\begin{prop}\label{nsne}
\emph{The Newton strata $\sS_{K,U,[b_0]}$ and $\sS_{K,U,[b_1]}$ are non-empty. }
\end{prop}

To prove Proposition \ref{nsne}, we introduce the notion of invariants for abelian varieties with $\bG$-structures, which are essentially defined in \cite[Section 3]{Kudla2015}. 

\begin{dfn}
Let $k$ be an algebraically closed field of characteristic $p$, and $(A,\iota,\lambda)$ a polarized $L$-abelian variety with signature $(2,2)$ over $k$. 
\begin{itemize}
\item For a prime $\ell \neq p$, consider the rational Tate module $V_{\ell}A$ of $A$. Then $\iota$ and $\lambda$ induces an $L_{\ell}$-action and a non-degenerate alternating form $(\,,\,)_{\ell}$ on $V_{\ell}A$. Moreover, there is an $L_{\ell}/\Ql$-hermitian form $\langle\,,\,\rangle_{\ell}$ on $V_{\ell}A$ satisfying the same formula as $(\,,\,)_{U}$. Put
\begin{equation*}
\inv_{\ell}(A,\iota,\lambda):=\det\!{}_{L_{\ell}}(\langle x_i,x_j\rangle_{\ell})_{1\leq i,j\leq 4}\in \Ql^{\times}/N_{L_{\ell}/\Ql}(L_{\ell}^{\times}),
\end{equation*}
where $x_1,\ldots,x_4$ is an $L_{\ell}$-basis of $V_{\ell}A$. It is independent of the choice of $x_1,\ldots,x_4$. 
\item For a prime $p$, we consider the rational Dieudonn{\'e} module $\D(A)_{\Q}$ with $L$-action and non-degenerate alternating form $(\,,\,)_{p}$, cf. Section \ref{pigs} (i). Then there is an $(L\otimes_{\Z}W(k))/\Frac W(k)$-hermitian form $\langle\,,\,\rangle_{p}$ on $\D(A)_{\Q}$ satisfying the same relation as $(\,,\,)_{U}$. Consider the isocrystal 
\begin{equation*}
(S(A),\F_{S(A)}):=\left(\bigwedge_{L\otimes_{\Z}W(k)}^{4}\D(A)_{\Q},\bigwedge_{L\otimes_{\Z}W(k)}^{4}\F\right)
\end{equation*}
with the $(L\otimes_{\Z}W(k))/\Frac W(k)$-hermitian form $\langle\,,\,\rangle_{S(A)}$ associated to $\langle \,,\,\rangle_{p}$. Then there is $z\in S(A)\setminus \{0\}$ satisfying $\F_{S(A)}(z)=p^2z$. Now put
\begin{equation*}
\inv_{p}(A,\iota,\lambda):=\langle z,z\rangle_{S(A)}\in \Qpt/N_{L_{p}/\Qp}(L_{p}^{\times}). 
\end{equation*}
It is independent of the choice of $z$. 
\end{itemize}
\end{dfn}

It suffices to show the following: 
\begin{prop}\label{expa}
\emph{
\begin{enumerate}
\item There is a polarized $O_L$-abelian variety $(A_0,\iota_0,\lambda_0)$ of signature $(2,2)$ over $\Fpbar$ such that $\inv_{\ell}(A_0,\iota_0,\lambda_0)=\inv_{\ell}(U)$ for any prime $\ell$. 
\item There is a polarized $O_L$-abelian variety $(A_1,\iota_1,\lambda_1)$ of signature $(2,2)$ over $\Fpbar$ such that $\inv_{\ell}(A_1,\iota_1,\lambda_1)=\inv_{\ell}(U)$ for any prime $\ell \neq p$, and $\inv_{p}(A_1,\iota_1,\lambda_1)\neq \inv_{p}(U)$. 
\end{enumerate}}
\end{prop}

\begin{proof}[Proposition \ref{expa} $\Rightarrow$ Proposition \ref{nsne}]
For $j\in \{0,1\}$, let $(A_j,\iota_j,\lambda_j)$ be as in Proposition \ref{expa}. By replacing $A_j$ to an isogeneous one, we may assume that $\D(A_j)$ is a $\varpi$-modular lattice in $\D(A_j)_{\Q}$. Note that there is an isomorphism of isocrystals with $G$-structures $\D(A_j)_{\Q}\cong N_{b_j}$ over $\Fpbar$. This follows from the equality of invariants at $p$. 

We construct an isomorphism $\eta^p\colon H_1(A_j,\A_f^p)\xrightarrow{\cong} U\otimes_{\Q}\A_f^p$ for $j\in \{0,1\}$. Let $S$ be the set of primes $\ell$ satisfying at least one of the following:
\begin{itemize}
\item $\ell$ ramifies in $L$,
\item $O\otimes_{\Z}\Zl$ is not a maximal order in $L_{\ell}$,
\item $\inv_{\ell}(V)\neq 1$,
\item $\ell \mid \deg(\lambda)$,
\item $\ell=2$. 
\end{itemize}
Put $\Zhat^S:=\prod_{\ell \not\in S}\Z_{\ell}$. Then there is a lattice $\bLambda_0$ in $U$ and an isomorphism of $(O_L\otimes_{\Z}\Zhat^S)/\Zhat^S$-hermitian spaces 
\begin{equation*}
\eta_j^{S}\colon \prod_{\ell \not\in S}T_{\ell}A \xrightarrow{\cong}\bLambda_0\otimes_{\Z}\Zhat^S. 
\end{equation*}
On the other hand, there is an isometry of $L_{\ell}/\Ql$-hermitian spaces 
\begin{equation*}
\eta_{j,\ell}\colon V_{\ell}A_j\xrightarrow{\cong}U_{\ell}
\end{equation*}
by the equalities of invariants at all $\ell \in S\setminus \{p\}$. Combining with $\eta_j^S$ and $\eta_{j,\ell}$ for $\ell \in S\setminus \{p\}$, we obtain an isomorphism of $(L\otimes_{\Q}\A_f^p)/\A_f^p$-hermitian spaces
\begin{equation*}
\eta_{j}^p\colon H_1(A_j,\A_f^p)\xrightarrow{\cong}U\otimes_{\Q}\A_f^p. 
\end{equation*}
Therefore $(A_j,\iota_j,\lambda_j,\eta_j^p \bmod K^p)$ gives an $\Fpbar$-valued point of $\sS_{K,U,[b_j]}$. 
\end{proof}

\begin{proof}[Proof of Proposition \ref{expa}]
(i): Let $E$ be a supersingular elliptic curve over $\Fpbar$. Fix an injection $\iota_E\colon O_L\rightarrow \End(E)$ (note that $\End(E)$ is a maximal order of the quaternion algebra over $\Q$ which is ramified exactly at $p$ and $\infty$) and a principal polarization $\lambda_E$ of $E$. Then we have $\lambda_E\circ \iota_E(a)=\iota_E(\abar)^{\vee}\circ \lambda_E$ for any $a\in O_L$. 

There is $c\in \Q^{\times}$ so that $c=\inv_{\ell}(U)$ in $\Q_{\ell}^{\times}/\N_{L_{\ell}/\Ql}(L_{\ell}^{\times})$ for any prime $\ell$ of $\Q$. In particular, we have $c>0$ and $c\in \N_{L_p/\Qp}(L_p^{\times})$. By replacing $c$ to $ce$ for some $e\in \N_{L/\Q}(L^{\times})$ if necessary, we may assume that $\ord_{p}(c)=0$. Write $c=c_1/c_2$, where $c_i\in \Zpn$ and $p\nmid c_i$. Then, $(E^4,\iota_E^4,(c_1\lambda_E)\times (c_2\lambda_E)^{3})$ is a desired polarized $L$-abelian $4$-fold over $\Fpbar$. 

(ii): Write $\Hom(L,\Qpbar)=\{\varphi_1,\varphi_2\}$. By the theory of complex multiplication \cite{Shimura1961}, for $i\in \{1,2\}$, there is an elliptic curve $\cE_i$ over $\Zp$ with an action $\iota_i$ on $\cE_i$ such that $\cE_{i,\Fpbar}$ is supersingular, and $L$ acts on $\Lie(\cE_{i,\Qpbar})$ by $\varphi_i$. Put $\cA:=\cE_1\times \cE_2$, and define an action $\iota$ of $O_L$ on $\cA$ by $\iota_1\times \iota_2$. Then
we have 
\begin{equation*}
\det(T-\iota(a)\mid \Lie(\cA_{\Qp}))=T^2-\tr_{L/\Q}(a)T+\N_{L/\Q}(a)
\end{equation*}
for any $a\in O_L$. By the same argument as the proof of \cite[Lemma 3.5]{Kudla2015}, there is a polarization $\lambda$ of $\cA$ such that $\ker(\lambda)=\ker(\iota(\sqrt{\Delta}))$. Note that we have $\inv_{\ell}(\cA_{\Fpbar},\iota_{\Fpbar},\lambda_{\Fpbar})=1$ for any prime $\ell$. On the other hand, by the same argument in the proof of (i), there is a polarized $L$-abelian surface $(A',\iota',\lambda')$ over $\Fpbar$ such that $\inv_{\ell}(A',\iota',\lambda')=\inv_{\ell}(U)$ for any prime $\ell \neq p$ and $\inv_{p}(A',\iota',\lambda')\neq \inv_{p}(U)$. Hence $(\cA_{\Fpbar}\times A',\iota_{\Fpbar} \times \iota',\lambda_{\Fpbar}\times \lambda')$ is a desired polarized $O_L$-abelian $4$-fold over $\Fpbar$. 
\end{proof}

\subsection{Structure of the supersingular loci}

We define the supersingular locus $\sS_{K,U}^{\si}$ as the reduced closed subscheme of $\sS_{K,U,\Fpbar}$ satisfying
\begin{equation*}
\sS_{K,U}^{\si}(k)=\{(A,\iota,\lambda,\overline{\eta}^p)\in \sS_{K,U}(k)\mid A\text{ is supersingular}\}
\end{equation*}
for any algebraically closed field $k$ of characteristic $p$. 

\begin{thm}\label{imdc}
\emph{
\begin{enumerate}
\item There is a decomposition into open and closed subschemes
\begin{equation*}
\sS_{K,U}=\sS_{K,U}^{\loc}\sqcup (\sS_{K,U}\setminus \sS_{K,U}^{\loc}). 
\end{equation*}
\item The scheme $\sS_{K,U}^{\loc}$ is flat over $\Zp$ and regular of dimension $5$. 
\item The scheme $\sS_{K,U}\setminus \sS_{K,U}^{\loc}$ is smooth over $\Fp$ and $3$-dimensional.
\item The non-smooth locus of $\sS_{K,U}^{\loc}\otimes_{\Zp}W$ is contained in $\sS_{K,U}^{\si}$. 
\end{enumerate}}
\end{thm}

\begin{proof}
These follow from the same argument as \cite[Theorem 7.5]{Oki2019} using Proposition \ref{lcmd} and \cite[Theorem 6.4]{Haines2005}. 
\end{proof}

\begin{thm}\label{ssbs}
\emph{The supersingular locus $\sS_{K,U}^{\si}$ equals $\sS_{K,U,[b_0]}\sqcup \sS_{K,U,[b_1]}$. }
\end{thm}

\begin{proof}
This follows from the equality $B(G)_{1,\bs}=\{[b_0],[b_1]\}$ proved in Section \ref{ntrl} and Proposition \ref{bsss}. 
\end{proof}

We give a connection between $\sS_{K,U}^{\si}$ the Rapoport--Zink spaces $\M_{b_0,\mu}$ and $\M_{b_1,\mu}$. For $j\in \{0,1\}$, let $\widehat{\sS}_{K,U,[b_j]}$ be the completion of $\sS_{K,U}\otimes_{\Zp}W$ along $\sS_{K,U,[b_j]}$. 

\begin{prop}\label{paut}
\emph{Let $j\in \{0,1\}$. Then there is an isomorphism
\begin{equation*}
I_j(\Q)\backslash (\M_{b_j,\mu}\times \bG(\A^p_f)/K^p)\xrightarrow{\cong} \widehat{\sS}_{K,U,[b_j]}
\end{equation*}
of formal schemes over $\spf W$. Here $I_j$ is an algebraic group over $\Q$ defined by 
\begin{equation*}
I_j(R)=\{(g,c)\in (\End_{O_L}(A_j)\otimes_{\Z}R)\times R^{\times} \mid g^{\vee}\circ \lambda_j\circ g=c\lambda_j\}
\end{equation*}
for any $\Q$-algebra $R$. }
\end{prop}

\begin{proof}
If we can apply the $p$-adic uniformization theorem \cite[Theorem 6.30]{Rapoport1996a}, then the assertion follows. For this, it suffices to show the non-emptiness of $\sS_{K,U,[b_j]}$ and that the Hasse principle for $\bG$ holds. However, the first assertion follows from Proposition \ref{nsne}, and the second assertion is already pointed out in Section \ref{ushv}. 
\end{proof}

Now we can prove the theorems on explicit descriptions of basic loci in Section \ref{mtim} by using Proposition \ref{paut} and the result on the Rapoport--Zink spaces. 

\begin{thm}\label{irg0}
\emph{
\begin{enumerate}
\item The scheme $\sS_{K,U,[b_0]}$ is purely $2$-dimensional. Any irreducible component is birational to the Fermat surface
\begin{equation*}
x_0^{p+1}+x_1^{p+1}+x_2^{p+1}+x_3^{p+1}=0
\end{equation*}
in $\P_{\Fpbar}^{3}=\Proj \Fpbar[x_0,x_1,x_2,x_3]$. 
\item Let $C$ be an irreducible component of $\sS_{K,U,[b_0]}$. Then, for any irreducible component $C'\neq C$ of $\sS_{K,U,[b_0]}$, the intersection $C\cap C'$ is either the empty set, a single point or birational to $\P_{\Fpbar}^{1}$. Moreover, the following hold: 
\begin{gather*}
\#\{C'\in \Irr(\sS_{K,U,[b_0]})\setminus \{C\}\mid C\cap C'\text{ is a single point}\}\leq p(p+1)(p^2+1),\\
\#\{C'\in \Irr(\sS_{K,U,[b_0]})\setminus \{C\}\mid C\cap C'\text{ is birational to }\P_{\Fpbar}^{1}\}\leq (p+1)(p^2+1). 
\end{gather*}
\item Each $C\in \Irr(\sS_{K,U,[b_0]})$ contains at most $(p+1)(p^2+1)$-non-smooth points of $\sS_{K,U}$. 
\item Each non-smooth point of $\sS_{K,U}$ is contained in at most $(p+1)$-irreducible components of $\sS_{K,U,[b_0]}$. 
\end{enumerate}}
\end{thm}

\begin{proof}
This follows from the same argument as the proof of \cite[Theorem 7.15]{Oki2019} by using Theorems \ref{paut}, \ref{mthm} and Corollary \ref{ctnt}. 
\end{proof}

\begin{thm}\label{cti0}
\emph{
\begin{enumerate}
\item The number of connected components of $\sS_{K,U,[b_0]}$ equals the cardinality of 
\begin{equation*}
I_0(\Q)\backslash (J_{b_0}^{0,0}\backslash J_{b_0}(\Qp)\times \bG(\A_f^p)/K^p), 
\end{equation*}
where $J_{b_0}^{0,0}:=\{g\in J_{b_0}^{0}(\Qp)\mid \ord_{p}(\sml(g))=0\}$. 
\item The number of irreducible components of $\sS_{K,U,[b_0]}$ equals the cardinality of 
\begin{equation*}
I_0(\Q)\backslash (J_{b_0}(\Qp)/K_{b_0,4}\times \bG(\A_f^p)/K^p), 
\end{equation*}
where $K_{b_0,4}$ is the stabilizer of a vertex lattice in $\VL_{b_0}(4)$. 
\item The number of non-smooth points of $\sS_{K,U}$ equals the cardinality of 
\begin{equation*}
I_0(\Q)\backslash (J_{b_0}/K_{b_0,4}\times \bG(\A_f^p)/K^p),
\end{equation*}
where $K_{b_0,4}$ is the same as one in (ii). 
\end{enumerate}}
\end{thm}

\begin{proof}
(i): The proof is the same as \cite[Theorem 7.13 (i)]{Oki2019} by using Theorems \ref{paut} and \ref{dcnt}. 

(ii): The proof is the same as \cite[Theorem 7.13 (ii)]{Oki2019} by using Theorems \ref{paut} and \ref{mthm} (iv). 

(iii): This follows from the same argument as the proof of \cite[Theorem 7.13 (iii)]{Oki2019} by using Theorems \ref{paut} and \ref{mthm} (v). 
\end{proof}

\begin{thm}\label{irg1}
\emph{
\begin{enumerate}
\item The scheme $\sS_{K,U,[b_1]}$ is purely $1$-dimensional. Any irreducible component is birational to $\P_{\Fpbar}^{1}$. 
\item Let $C$ be an irreducible component of $\sS_{K,U,[b_1]}$. Then, for any irreducible component $C'\neq C$ of $\sS_{K,U,[b_1]}$, the intersection $C\cap C'$ is either the empty set or a single point. Moreover, we have
\begin{equation*}
\#\{C'\in \Irr(\sS_{K,U,[b_1]})\setminus \{C\}\mid C\cap C'\text{ is a single point}\}\leq p(p^2+1). 
\end{equation*}
\end{enumerate}}
\end{thm}

\begin{proof}
This follows from the same argument as the proof of \cite[Theorem 7.15]{Oki2019} by using Theorems \ref{paut}, \ref{mtdm} and Corollary \ref{ctnn}. 
\end{proof}

\begin{thm}\label{cti1}
\emph{
\begin{enumerate}
\item The number of connected components of $\sS_{K,U,[b_1]}$ equals the cardinality of 
\begin{equation*}
I_1(\Q)\backslash (J_{b_1}^{0,0}\backslash J_{b_1}(\Qp)\times \bG(\A_f^p)/K^p), 
\end{equation*}
where $J_{b_1}^{0,0}:=\{g\in J_{b_1}^{0}(\Qp)\mid \ord_{p}(\sml(g))=0\}$. 
\item The number of irreducible components of $\sS_{K,U,[b_1]}$ equals the cardinality of 
\begin{equation*}
I_1(\Q)\backslash (J_{b_1}(\Qp)/K_{b_1,0}\times \bG(\A_f^p)/K^p), 
\end{equation*}
where $K_{b_1,0}$ is the stabilizer of a vertex lattice in $\VL_{b_1}(0)$. 
\end{enumerate}}
\end{thm}

\begin{proof}
(i): The proof is the same as \cite[Theorem 7.13 (i)]{Oki2019} by using Theorems \ref{paut} and \ref{dcnn}. 

(ii): This follows from the same argument as the proof of \cite[Theorem 7.13 (ii)]{Oki2019} by using Theorem \ref{paut} and Corollary \ref{irnn}. 
\end{proof}

\end{document}